\documentclass[a4paper]{amsart}
\usepackage{amsmath}
\usepackage{amsthm}
\usepackage{amssymb}
\usepackage{enumerate}
\usepackage{latexsym}
\usepackage{empheq}
\usepackage{comment}
\newtheorem{Thm}{Theorem}[section]
\newtheorem{Prop}[Thm]{Proposition}
\newtheorem{Cor}[Thm]{Corollary}
\newtheorem{Lem}[Thm]{Lemma}
\newtheorem{Clm}[Thm]{Claim}

\theoremstyle{definition}
\newtheorem{notation}[Thm]{Notation}
\newtheorem{Def}[Thm]{Definition}

\newtheorem{Asu}[Thm]{Assumption}
\theoremstyle{remark}
\newtheorem{Rem}[Thm]{Remark}

\newcommand{\Ric}{\mathrm{Ric}}
\newcommand{\tr}{\mathrm{tr}}
\newcommand{\Imag}{\mathrm{Im}}
\newcommand{\Ker}{\mathrm{Ker}}
\newcommand{\Id}{\mathrm{Id}}
\newcommand{\Vol}{\mathrm{Vol}}
\newcommand{\diam}{\mathrm{diam}}
\newcommand{\Card}{\mathrm{Card}}
\newcommand{\Span}{\mathrm{Span}}

\newcommand{\LIP}{\mathrm{LIP}}
\newcommand{\Lip}{\mathrm{Lip}}

\newcommand{\Hess}{\mathrm{Hess}}

\newcommand{\Inj}{\mathrm{inj}}
\newcommand{\Reach}{\mathrm{Reach}}
\newcommand{\Sect}{\mathrm{Sect}}

\newcommand{\R}{\mathbb{R}}
\newcommand{\Z}{\mathbb{Z}}
\newcommand{\X}{\mathbb{X}}
\newcommand{\Pj}{\mathbb{P}}
\newcommand{\1}{\mathbf{1}}

\newcommand{\RCD}{\mathrm{RCD^\ast}}

\newcommand{\Ha}{\mathcal{H}}

\newcommand{\D}{\mathcal{D}}

\newcommand{\K}{\mathcal{K}}
\newcommand{\M}{\mathcal{M}}
\newcommand{\La}{\mathcal{L}}
\newcommand{\Le}{\mathcal{L}}
\DeclareMathOperator*{\esssup}{ess\,sup}
\makeatletter
\@namedef{subjclassname@2020}{\textup{2020} Mathematics Subject Classification}
\makeatother

\title[Laplacian Eigenmaps for Submanifolds with Singularities]{Convergence of Laplacian Eigenmaps and its Rate for Submanifolds with Singularities}
\author{Masayuki Aino}
\address{RIKEN, Center for Advanced Intelligence Project AIP, 1-4-1 Nihonbashi, Tokyo 103-0027, Japan}
\email{masayuki.aino@riken.jp}
\begin{document}
\maketitle
\begin{abstract}
In this paper, we give a  spectral approximation result for the Laplacian on submanifolds of Euclidean spaces with singularities by the $\epsilon$-neighborhood graph constructed from random points on the submanifold.
Our convergence rate for the eigenvalue of the Laplacian is $O\left(\left(\log n/n\right)^{1/(m+2)}\right)$, where $m$ and $n$ denote the dimension of the manifold and  the sample size, respectively.
\end{abstract} 
\tableofcontents
\section{Introduction}
In this paper, we give some results on spectral approximation of the weighted Laplacian on submanifolds of Euclidean spaces with singularities.
We construct a graph called an $\epsilon$-neighborhood graph and consider its graph Laplacian under the assumption that  we are given i.i.d. sample from the weighted Riemannian volume measure on the manifold.
As we see below, the edges of the graph and their weights are defined using the Euclidean distance, since the geodesic distance of the Riemannian manifold is not given.
The method for dimensionality reduction using eigenvectors of such graph Laplacian, or some other variant of it, is known as Laplacian Eigenmaps and was proposed in \cite{BN1}.
Some geometric properties of the embedding of Riemannian manifolds, or metric measure spaces that are not necessarily smooth but satisfy certain geometric conditions, into Euclidean spaces using eigenfunctions of the Laplacian have been studied in \cite{ASPT}, \cite{Bates} and \cite{Por}.
In particular, \cite{ASPT} argues under weaker assumptions such as lower bound on the Ricci curvature.
On the other hand, the spectral approximation of the Laplacian on submanifolds of Euclidean spaces by the graph Laplacian has been discussed in \cite{BN2}, \cite{CT} and \cite{TGHS} assuming lower bound for a quantity called reach $\Reach\geq R$ (Definition \ref{DR}), which leads to bounds on the sectional curvature $|\Sect|\leq 1/R^2$ and the injectivity radius $\Inj\geq \pi R$ (Theorem \ref{CR}).
Singularities are not allowed under the assumptions about reach, for example, the reach of a square boundary $\partial([0,1]^2)$ is $0$.
The goal of this paper is to give the spectral approximation of the Laplacian  replacing the assumption on the reach with a weaker one assuming bounded sectional curvature and injectivity radius.
Under our assumptions, intrinsic singularities do not appear, but singularities as submanifolds do, and as we will see in Appendix \ref{SwDS}, submanifolds with dense singularities can appear.
Moreover, the square boundary $\partial([0,1]^2)$ satisfies our assumptions.

In this work, we deal with submanifolds with singularities by approximating them with smooth submanifolds satisfying the following.
\begin{Def}
Let $m\in \Z_{>0}$ be an integer.
For each $K,i_0>0$, we define $\M_1=\M_1(m,K,i_0)$ to be the set of (isometry class of) $m$-dimensional closed (i.e., compact, connected and without boundary) Riemannian manifolds $(M,g)$ satisfying $|\Sect_g|\leq K$ and $\Inj_g\geq i_0.$
\end{Def}

We consider a metric space $M$ and a map $\iota\colon M\to \R^d$ (possibly non-smooth), approximated by $(M_i,g_i)\in\M_1(m,K,i_0)$ and an isometric immersion $\iota_i\colon M_i\to \R^d$ with bounded $L^1$-norm of the second fundamental form. 
More precisely, we consider the following assumptions.
\begin{Asu}\label{Asu}
Let $m,d\in \Z_{>0}$ be integers with $m<d$ and take constants $S,K,i_0,L>0$.
Suppose that we are given a compact metric space $M$ with distance function $d_M$ and a map $\iota\colon M\to \R^d$ such that there exist a sequence $\{(M_i,g_i)\}_{i=1}^\infty\subset \M_1(m,K,i_0)$, a sequence of positive real numbers $\{\epsilon_i\}_{i=1}^\infty\subset \R_{>0}$ with $\lim_{i\to\infty}\epsilon_i=0$, a sequence of isometric immersions $\{\iota_i\colon M_i\to \R^d\}_{i=1}^\infty$ and a sequence of maps
$\{\psi_i\colon M\to M_i\}_{i=1}^\infty$ satisfying the following properties:
\begin{itemize}
\item[(i)] For any $x,y\in M$,
we have
$
d_M(x,y)\leq L\|\iota(x)-\iota(y)\|_{\R^d}.
$
\item[(ii)] For any $i\in \Z_{>0}$, we have
$
\int_{M_i} |II_i|\,d\Vol_{g_i} \leq S
$, where $II_i$ denotes the second fundamental form of $\iota_i$.
\item[(iii)] For any $i\in\Z_{>0}$ and $x,y\in M$, we have
$
|d_M(x,y)-d_{g_i}(\psi_i(x),\psi_i(y))|\leq \epsilon_i.
$
\item[(iv)] For any $i\in\Z_{>0}$ and $y\in M_i$, there exists $x\in M$ such that
$
d_{g_i}(y,\psi_i(x))\leq \epsilon_i.
$
\item[(v)] For any $x\in M$, we have
$
\lim_{i\to \infty}\iota_i(\psi_i(x))=\iota(x).
$
\end{itemize}
\end{Asu}
Under Assumption \ref{Asu}, by the $C^{1,\alpha}$ compactness theorem \cite[Theorem 11.4.7]{Pe3}, it turns out that $M$ is a smooth manifold with a $C^{1,\alpha}$ Riemannian metric $g$ ($\alpha\in(0,1)$), i.e., the Riemannian distance with respect to $g$ coincides with $d_M$.
Therefore, from now on, $M$ satisfying Assumption \ref{Asu} will be treated as a Riemannian manifold $(M,g)$.

We next introduce our weighted Laplacian.
Suppose that we are given a closed Riemannian manifold $(M,g)$ and positive Lipschitz function $\rho\colon M\to\R_{>0}$.
Then, we define the operator $\Delta_\rho$ by
$\Delta_\rho f:=\rho\Delta f-2\langle\nabla \rho,\nabla f\rangle$, where $\Delta$ denotes the Laplacian without wight defined by $\Delta=-\tr \Hess$.
For the normalized case, we consider the operator $\Delta_\rho^N:=(1/\rho)\Delta_\rho$.
Let $0=\lambda_0(\Delta_\rho)<\lambda_1(\Delta_\rho)\leq \lambda_2(\Delta_\rho)\leq\cdots \to \infty$ and
$0=\lambda_0(\Delta_\rho^N)<\lambda_1(\Delta_\rho^N)\leq \lambda_2(\Delta_\rho^N)\leq\cdots \to \infty$
be the eigenvalues of $\Delta_\rho$ and $\Delta_\rho^N$ counted with multiplicities, respectively.

Let us explain how to construct the graph Laplacian.
Suppose that we are given a manifold $M$, a map $\iota\colon M\to\R^d$, a function $\eta\colon [0,\infty)\to[0,\infty)$, a constant $\epsilon>0$ and points $x_1,\ldots, x_n\in M$.
Then, we define $n\times n$ matrices $\mathcal{K},\mathcal{D},\La\in \R^{n\times n}$ as follows:
\begin{equation}\label{defGLap}
\mathcal{K}_{i j}=\eta \left(\frac{\|\iota(x_i)-\iota(x_j)\|}{\epsilon}\right),\quad
\mathcal{D}_{i j}=\delta_{i j}\sum_{l=1}^n \mathcal{K}_{i l},\quad \La=\mathcal{D}-\mathcal{K}.
\end{equation}
Let
$
0=\lambda_0(\La)\leq \lambda_1(\La)\leq\cdots \leq\lambda_{n-1}(\La)
$
be the eigenvalues of $\La$ counted with multiplicities, and $u^0,\ldots,u^{n-1}\in \R^n$ be corresponding eigenvectors.
Note that we have $u^0_1=\cdots=u^0_n$.
If we are given an embedding dimension $k\in\{1,\ldots,n-1\}$, we define $y_1,\ldots,y_n\in \R^k$ so that
$(y_1,\ldots,y_n)=(u^1,\ldots,u^k)^\ast$, where $A^\ast$ denotes the transpose of $A$ for any matrix $A$. 
Then, $y_1,\ldots,y_n$ is the output of the Laplacian eigenmaps.
Under our assumptions $u^j$ corresponds to an eigenfunction $f_j$ of an appropriate Laplacian, and $y_i$ corresponds to $(f_1(x_i),\ldots, f_k(x_i))\in\R^k$.
Here,  we used the unnormalized graph Laplacian $\La$ of the $\epsilon$-neighborhood graph, and in the normalized case we consider the eigenvalue problem $\La u=\lambda\D u$.
In our analysis, we assume the following conditions on $\eta$ and on the probability density function $\rho$ from which the sample points are taken. 
\begin{Asu}\label{Asu3}
Suppose that we are given a Riemannian manifold $(M,g)$ and functions $\eta\colon [0,\infty)\to [0,\infty)$, $\rho\colon M\to \R_{>0}$ such that following properties hold:
\begin{itemize}
\item[(i)] $\eta$ is non-increasing, $\Lip(\eta|_{[0,1]})\leq L_\eta$ holds for some $L_\eta>0$, $\eta(3/4)>0$ and $\eta|_{(1,\infty)}\equiv 0$.
\item[(ii)] $\Lip(\rho)\leq L_{\rho}$ for some $L_\rho>0$, $1/\alpha\leq \rho\leq \alpha$ holds for some $\alpha>0$ and
$$
\int_M \rho\,d\Vol_g=1.
$$
\end{itemize}
\end{Asu}
For example, functions $\eta|_{[0,1]}\equiv 1$ and $\eta|_{[0,1]}(t)=e^{-t^2}$ are often used.

We now state our main result on the convergence of the eigenvalues of the Laplacian.
\begin{Thm}\label{MainResult1}
Suppose that the pair $((M,g),\iota)$ satisfies Assumption \ref{Asu}, and we are given functions $\eta, \rho$ satisfying Assumption \ref{Asu3}.
For any $k\in \Z_{>0}$, there exist constants $C_1=C_1(m,\alpha,k)>0$ and $C_2=C_2(m,S,K,i_0,L,\eta,\alpha,L_\rho,k)>0$
 such that, for any $\gamma\in(1,\infty)$ and
$n\in\Z_{>0}$ with $\gamma^{1/2}\epsilon\leq C_2^{-1}$, where $\epsilon=\epsilon_n:=(\log n/n)^{1/(m+2)}$,
and i.i.d. sample $x_1,\ldots,x_n\in M$ from $\rho\Vol_g$, defining $\La$ as $(\ref{defGLap})$, we have 
$$
\left|\lambda_k(\Delta_{\rho})-\frac{2}{\sigma_\eta n \epsilon^{m+2}}\lambda_k(\La)\right|\leq C_2\gamma^{1/2} \left(\frac{\log n}{n}\right)^{1/(m+2)}
$$
with probability at least $1-C_1n^{-\gamma}$, where we defined
$$
\sigma_\eta:=\frac{\Vol(S^{m-1})}{m}\int_0^1\eta(t) t^{m+1}\,d t.
$$
\end{Thm}
Note that the assumption that $\eta(3/4)>0$ is not specific.
In fact, defining $\widetilde\eta(t)=\eta(3t/4)$ and $\tilde \epsilon =3\epsilon/4$,
the matrices determined in (\ref{defGLap}) remains the same regardless of whether $(\eta,\epsilon)$ or $(\widetilde \eta,\tilde \epsilon)$ is used.
Moreover,  for the normalization term, $\sigma_\eta n\epsilon^{m+2}=\sigma_{\widetilde \eta} n\tilde \epsilon^{m+2}$ is obtained.

Despite our weaker assumptions,
our convergence rate of eigenvalues is better than those of \cite{CT} and \cite{TGHS}, which are $O\left((\log n)^{p_m/2}/n^{1/(2m)}\right)$ ($p_2=3/4$ and $p_m=1/m$ for $m\geq 3$) and $O\left(\left(\log n/n\right)^{1/(m+4)}\right)$, respectively.
In Theorem \ref{MainResult1}, $\epsilon$ is given explicitly for sample size $n$ and intrinsic dimension $m$ of the manifold, but the convergence (with possibly different rate) is obtained when $\lim_{n\to\infty }\epsilon_n=0$ and
$$
\lim_{n\to\infty} \frac{1}{\epsilon_n}\left(\frac{\log n}{n}\right)^{1/m}=0.
$$
See Theorem \ref{UNMainEval} for details.

We next state our main result on convergence to eigenfunctions of the Laplacian corresponding to isolated eigenvalues.
\begin{Thm}\label{MainResult2}
Suppose that the pair $((M,g),\iota)$ satisfies Assumption \ref{Asu}, and we are given functions $\eta, \rho$ satisfying Assumption \ref{Asu3}.
Take arbitrary $k\in \Z_{>0}$ and assume
$
G:=\min\left\{|\lambda_k(\Delta_\rho)-\lambda_{k-1}(\Delta_\rho)|,|\lambda_{k+1}(\Delta_\rho)-\lambda_k(\Delta_\rho)|\right\}>0.
$
Then, there exist constants $C_1=C_1(m,\alpha,k)>0$ and $C_2=C_2(m,S,K,i_0,L,\eta,\alpha,L_\rho,k)>0$ such that, for any
$n\in\Z_{>0}$, $\gamma\in(1,\infty)$ with
$\gamma^{1/2} \epsilon\leq  C_2^{-1}G^2$, where $\epsilon:=(\log n/n)^{1/(m+2)}$,
and i.i.d. sample $x_1,\ldots,x_n\in M$ from $\rho\Vol_g$, defining $\La$ as $(\ref{defGLap})$, we have the following with probability at least $1-C_1n^{-\gamma}.$
For the eigenvector $u^k$ of $\La$ with $\sum_{i=1}^n (u^k_i)^2/n=1$ corresponding to the eigenvalue $\lambda_k(\La)$, we can take the eigenfunction $f$ of $\Delta_\rho$ with $\int_M f^2\rho\,d\Vol_g=1$ corresponding to the eigenvalue $\lambda_k(\Delta_\rho)$ so that
$$
\frac{1}{n}\sum_{i=1}^n \left(f_k(x_i)-u^k_i\right)^2\leq \frac{C_2}{G^2}\gamma^{1/2} \left(\frac{\log n}{n}\right)^{1/(m+2)}.
$$
\end{Thm}
Since the eigenvalue of the Laplacian $\lambda_k(\Delta_\rho)$ is assumed to be isolated, the only freedom in taking $f_k$ is in the sign.
See Theorem \ref{UNMainEvec} for the general cases where the eigenvalues of the Laplacian have multiplicities or very small spectral gaps.
Our rate of convergence to eigenfunctions is worse than that of \cite{CT}, but the proof in \cite{CT} is based on the following pointwise approximation of the Laplacian:
\begin{equation}\label{LapApprox}
\Delta_\rho f(x)\approx\frac{2}{\sigma_\eta \epsilon^{m+2}}\int_M\eta\left(\frac{\|\iota(x)-\iota(y)\|}{\epsilon}\right)(f(x)-f(y))\rho(y)\,d y
\end{equation}
(it turns out that the  $L^2$ approximation is enough for their proof), and as we will see in Appendix \ref{sensitivity}, the $L^p$ approximation ($p\in[1,\infty]$) of the Laplacian of this form is sensitive to singularities, so their proof is not applicable under our assumptions, at least not as is.

Here we gave results for the convergence of the unnormalized graph Laplacian, but we can give similar results for the normalized case.
See Theorem \ref{NMainEval} and \ref{NMainEvec}.

This paper can be read independently of the previous studies \cite{BN2}, \cite{CT} and \cite{TGHS}.
In a technical step we use the corresponding assertion to \cite[Proposition 2.11]{CT}, but in Appendix \ref{MapStD} we give a simple proof of it in the form we use based on arguments from standard Riemannian geometry.

The structure of this paper is as follows.

In section 2, we fix our notation and introduce some definitions,  give easy consequences of our assumptions, and summarize the prerequisite knowledge.

In section 3, we give our main results.
In subsection 3.1, given a function on the manifold, we study its properties on the random points.
In subsection 3.2, given a function on the random points, we study some properties of a corresponding function on the manifold.
In subsection 3.3, we give our main results for the case of unnormalized graph Laplacian, and in subsection 3.4 for the case of normalized graph Laplacian.

In Appendix \ref{GradientEstimate}, we give the $L^\infty$ estimate and the gradient estimate for eigenfunctions of our Laplacian $\Delta_\rho$ and $\Delta_\rho^N$.

In Appendix \ref{LinearAlgebra}, we summarize the linear algebraic arguments needed to complete the proof of our main results.

In Appendix \ref{MapStD}, we construct a map from our manifold to random points that almost preserve the distance and the measure.

In Appendix \ref{ApproxSmooth}, we explain why approximation by a sequence of smooth submanifolds yields the main results in the limit.

In Appendix \ref{Reach}, we discuss the relationship between the reach and other geometric quantities.

In Appendix \ref{HausdorffMeasure}, we discuss the consistency of the measures under Assumption \ref{Asu}.
Theorem \ref{HAP} asserts that the Hausdorff measure determined by given distance function and the Hausdorff measure determined by the Euclidean distance function coincide with each other under our assumptions.

In Appendix \ref{sensitivity}, we show by example that the $L^p$ approximation of the Laplacian (\ref{LapApprox}) is sensitive to singularities for $p\in [1,\infty]$ even if the singularities are simple and the probability density function is constant, where the case of $p=\infty$ corresponds to the pointwise approximation.

In Appendix \ref{SwDS}, we construct an example of a submanifold with dense singularities under our assumptions.

\begin{sloppypar}
{\bf Acknowledgments}.\ 
I am grateful to Professor Shouhei Honda for answering my questions about the gradient estimate for eigenfunctions.
This work was supported by RIKEN Special Postdoctoral Researcher Program.
\end{sloppypar}

\section{Preliminaries}
\subsection{Notation and Definitions}
In this subsection, we fix our notation and prepare some definitions.
For given real numbers $u_1,\ldots,u_l$, let $C(u_1,\ldots,u_l)$ denotes a constant depending only on $u_1,\ldots,u_l$.
If we want to distinguish between the constants, express them as $C_1,C_2,\ldots$.
When the constants appearing in each claim are of the form $C(u_1,\ldots,u_l)$, the constants $C$ appearing in the proof depend only on at most $u_1,\ldots,u_l$, unless otherwise noted.
For brevity, we sometimes denote $C(\ldots,\eta(0),1/\eta(3/4),L_\eta,\ldots)$ by $C(\ldots,\eta,\ldots)$ under Assumption \ref{Asu3}.
For a set $X$, $\Card X$ denotes the cardinal number of $X$.

We summarize our notation for Riemannian manifolds and isometric immersions. 
\begin{notation}\label{Not}
Let $(M,g)$ be a closed Riemannian manifold.
\begin{enumerate}[(i)]
\item $d_g$ denotes the Riemannian distance function.
If there is no confusion, we simply write $d$.
For any $x\in M$ and $r\in (0,\infty)$, let $B_r(x)=B_r^M(x)$ denotes
$
B_r(x):=\{y\in M:d(x,y)<r\}.
$
We sometimes write it by $B(x,r)$.
Similarly, $\overline{B}_r(x):=\{y\in M: d(x,y)\leq r\}$.
\item For two points $x,y\in\R^d$, $\|x-y\|_{\R^d}$ denotes the Euclidean ($l^2$) distance.
If there is no confusion, we simply write $\|x-y\|$.
For any $x\in \R^d$ and $r\in (0,\infty)$, let $B_r^{\R^d}(x)$ denotes
$
B_r^{\R^d}(x):=\{y\in \R^d:\|x-y\|<r\}.
$
We sometimes write it by $B^{\R^d}(x,r)$.
\item $\Sect_g$, $\Ric_g$, $\Inj_g$, $\diam_g(M)$ and $\Vol_g$ denote the sectional curvature, the Ricci curvature, the injectivity radius, the diameter and the Riemannian volume measure of $(M,g)$, respectively.
When the dimension of $M$ is $m$, $\Vol_g$ coincides with the $m$-dimensional Hausdorff measure $\Ha^m$ determined by the Riemannian distance $d_g$.
For any integrable function $f\colon M\to \R$, we sometimes use the following notation:
$$
\int_M f(x)\,d x=\int_M f\,d \Vol_g.
$$
\item $\nabla$ denotes the Levi-Civita connection.
We also use the notation $\nabla$ for the gradient of functions.
\item For any $k\in \Z_{>0}$, let $S^{k}:=\{x\in \R^{k+1}:\|x\|_{\R^{k+1}}=1\}$ denotes the $k$-dimensional standard sphere with standard Riemannian metric.
Note that we have $\Vol(S^k)=(k+1)\Le^{k+1}\left(B^{\R^{k+1}}(0,1)\right)$, where $\Le^{k+1}$ denotes the $(k+1)$-dimensional Lebesgue measure.
\item Let $T_x M$ denotes the tangent space at $x\in M$, and $U_x M:=\{v\in T_x M:|v|=1\}$, where $|v|:=g(v,v)^{1/2}$.
When the dimension of $M$ is $m$, using an orthonormal basis and identifying $T_x M$ with $R^m$ and $U_x M$ with $S^{m-1}$, we consider an $m$-dimensional Lebesgue measure on $T_x M$ and a Riemannian volume measure on $U_x M$ determined from $S^{m-1}$, respectively.
These measures are determined independently of the choice of the orthonormal basis.
\item Given points $x,y\in M$, let $\gamma_{x,y}$ denotes one of minimal geodesics with unit speed such that $\gamma_{x,y}(0)=x$ and $\gamma_{x,y}(d(x,y))=y$, and $c_{x,y}$ denotes one of minimal geodesics with constant speed such that $c_{x,y}(0)=x$ and $c_{x,y}(1)=y$.
Note that we can take $c_{x,y}$ so that $c_{x,y}(d(x,y) t)=\gamma_{x,y}(t)$.
For given $x\in M$ and $u\in T_x M$, let $\gamma_{u}\colon \mathbb{R}\to M$ denotes the geodesic with constant speed such that $\gamma_u(0)=x$ and $\dot{\gamma}_u(0)=u$.
The exponential map $\exp_x \colon T_x M\to M$ at $x\in M$ is defined by
$
\exp_x(u)=\gamma_u(1).
$
\item For any $x\in M$ and $u\in U_x M$, put
$$
t(u):=\sup\{t\in\mathbb{R}_{>0}: d(x,\gamma_u(t))=t\}.
$$
For any $x\in M$, we define $\widetilde J_x\in T_x M$ and $J_x\subset M$ by
\begin{align*}
\widetilde J_x :=&\{t u: u\in U_x M, \,0\leq t< t(u)\},\\
J_x:=&\exp_x (\widetilde J_x)=\{\gamma_u (t): u\in U_x M, \, 0\leq t< t(u)\}.
\end{align*}
Then, $J_x\subset M$ is open, $\exp_x|_{\widetilde J_x}\colon \widetilde J_x\to J_x$ is diffeomorphic and $\Vol(M\setminus J_x)=0$ \cite[III Lemma 4.4]{Sa}.
For any $y\in J_x$, $\gamma_{x,y}$ and $c_{x,y}$ are uniquely determined. The function $d(x,\cdot)\colon M\to \mathbb{R}$ is differentiable in $J_x\setminus\{x\}$ and $\nabla d(x,\cdot)(y)=\dot{\gamma}_{x,y}(d(x,y))$ holds for any $y\in J_x\setminus \{x\}$ \cite[III Proposition 4.8]{Sa}.
\item Let $\Delta$ denotes the Laplacian (without weight) acting on functions defined by $\Delta=-\tr \Hess$.
If $(M,g)$ is the $m$-dimensional Euclidean space with the standard metric, then $\Delta=-\sum_{i=1}^m\partial^2/\partial x_i^2$.
Note that some authors use the opposite sign for the Laplacian.
Let
$$
0=\lambda_0(\Delta)<\lambda_1(\Delta)\leq \lambda_2(\Delta)\leq\cdots \to \infty
$$
be the eigenvalues of the Laplacian $\Delta$ counted with multiplicities.
\item Suppose that we are given a positive Lipschitz function $\rho\colon M\to \R_{>0}$.
Let $\Delta_\rho\colon W^{2,2}(M)\to L^2(M)$ and $\Delta_\rho^N\colon W^{2,2}(M)\to L^2(M)$ be operators define by
\begin{align*}
\Delta_\rho f:=&\rho\Delta f-2\langle\nabla \rho,\nabla f\rangle,\\
\Delta_\rho^N f:=&\frac{1}{\rho}\Delta_\rho f=\Delta f-\frac{2}{\rho}\langle\nabla \rho,\nabla f\rangle,
\end{align*}
where $\nabla \rho$ and $\nabla f$ denote the gradient vector fields of $\rho$ and $f$ respectively, and $\langle\nabla \rho,\nabla f\rangle:=g(\nabla \rho,\nabla f)$.
Note that $\nabla \rho$ is defined as an $L^\infty$ vector field.
Let
\begin{align*}
0=&\lambda_0(\Delta_\rho)<\lambda_1(\Delta_\rho)\leq \lambda_2(\Delta_\rho)\leq\cdots \to \infty\\
0=&\lambda_0(\Delta_\rho^N)<\lambda_1(\Delta_\rho^N)\leq \lambda_2(\Delta_\rho^N)\leq\cdots \to \infty
\end{align*}
be the eigenvalues of $\Delta_\rho$ and $\Delta_\rho^N$ counted with multiplicities, respectively.
\item Suppose that we are given points $\X=\{x_1,\ldots,x_n\}\subset M$.
Then, for any function $f\colon M\to \R$, we sometimes regard $f|_{\X}$ as an element of $\R^n$  by
$$
f|_{\X}=(f(x_1),\ldots,f(x_n))\in\R^n.
$$
\item Suppose that we are given an isometric immersion $\iota \colon M\to \R^d$.
Then, $II\in\Gamma(TM\otimes TM\otimes TM^\perp)$ denotes the second fundamental form, and we define
$$
|II|(x):=\max\{\|II(v,v)\|_{\R^d}:v\in U_x M\}
$$
for each $x\in M$.
Here, $\iota\colon M\to\R^d$ is an isometric immersion means that the pullback of the Euclidean metric coincides with $g$, and does not impose the injectivity of $\iota$.
\end{enumerate}
\end{notation}

We next introduce our notation for the graph Laplacian.
\begin{notation}[Graph Laplacian]
Suppose that we are given a manifold $M$ a map $\iota\colon M\to\R^d$, a function $\eta\colon [0,\infty)\to[0,\infty)$, a constant $\epsilon>0$ and points $x_1,\ldots, x_n\in M$.
Then, we define $n\times n$ matrices $\mathcal{K},\mathcal{D},\La\in \R^{n\times n}$ as follows:
\begin{align*}
\mathcal{K}_{i j}= \eta \left(\frac{\|\iota(x_i)-\iota(x_j)\|}{\epsilon}\right),\quad
\mathcal{D}_{i j}=\delta_{i j}\sum_{l=1}^n \mathcal{K}_{i l},\quad
\La=\mathcal{D}-\mathcal{K}.
\end{align*}
Let
$$
0=\lambda_0(\La)\leq \lambda_1(\La)\leq\cdots \leq\lambda_{n-1}(\La)
$$
be the eigenvalues of $\La$ counted with multiplicities.
Let
$$
0=\lambda_0(\La,\D)\leq \lambda_1(\La,\D)\leq\cdots \leq\lambda_{n-1}(\La,\D)
$$
be the eigenvalues for the eigenvalue problem $\La v=\lambda \D v$ counted with multiplicities.
Note that $\lambda$ is an eigenvalue for such an eigenvalue problem if and only if $\lambda$ is an eigenvalue of the symmetric matrix $\D^{-1/2}\La\D^{-1/2}$.
\end{notation}
%

We summarize the definitions needed to relate functions on discrete points to functions on manifolds.
\begin{Def}\label{defDtC}
Suppose that we are given a closed Riemannian manifold $(M,g)$, points
$
x_1,\ldots,x_n\in M,
$
a positive real number $\epsilon>0$ and functions $\eta,\rho$ as Assumption \ref{Asu3}.
\begin{enumerate}[(i)]
\item Define $\sigma_\eta\in \R_{>0}$ by
$$
\sigma_\eta:=\frac{\Vol(S^{m-1})}{m}\int_0^1\eta(t) t^{m+1}\,d t.
$$
\item Define a map $\psi\colon [0,\infty)\to [0,\infty)$ by
$$
\psi(t):=\int_t^\infty \eta(s)s\,d s
$$
Note that $\psi$ is a Lipschitz function, $\psi(t)\leq \eta(t)/2$ for any $t\in[0,\infty)$ and $\psi|_{[1,\infty)}\equiv 0$.
\item Define a map $\overline{\theta}_\epsilon\colon M\to \R_{>0}$ by
$$
\overline{\theta}_\epsilon(x):=\int_M\psi\left(\frac{d(x,y)}{\epsilon}\right)\rho(y)\, d y.
$$
\item Define a map $\theta_\epsilon\colon M\to \R_{>0}$ by
$$
\theta_\epsilon(x):=\frac{1}{n}\sum_{i=1}^n\psi\left(\frac{d(x,x_i)}{\epsilon} \right).
$$
\item If $\theta_\epsilon(x)>0$ holds for every $x\in M$, define an operator $\Lambda_\epsilon\colon \R^n\to \LIP(M)$ by
$$
\Lambda_\epsilon u(x):=\frac{1}{n\theta_\epsilon(x)}\sum_{i=1}^n\psi\left(\frac{d(x,x_i)}{\epsilon}\right)u_i.
$$
Note that if $u_i=1$ for any $i=1,\ldots,n$, then we have $\Lambda_\epsilon u\equiv 1$.
\item If $\theta_\epsilon(x)>0$ holds for every $x\in M$, we define a map $\widetilde\psi\colon M\times M\to \R$ by
$$
\widetilde \psi(x,y):=\frac{1}{\theta_\epsilon(x)}\psi\left(\frac{d(x,y)}{\epsilon}\right).
$$
Note that we have $\sum_{i=1}^n \widetilde \psi(x,x_i)/n=1$ for any $x\in M$, and $\Lambda_\epsilon u(x)=\sum_{i=1}^n \widetilde \psi(x,x_i) u_i/n$ for any $x\in M$ and $u\in \R^n$.
\item Define a Borel map $\Psi\colon M\times M\to TM$ by
$$
\Psi(x,y):=\eta\left(\frac{d(x,y)}{\epsilon}\right)\frac{d(x,y)}{\epsilon^2}\dot\gamma_{x,y}(0)\in T_x M.
$$
Note that, for any $x\in M$, $\Psi(x,y)$ is uniquely determined and
$$
\nabla_x \psi\left(\frac{d(x,y)}{\epsilon}\right)=\Psi(x,y)
$$
for a.e. $y\in M$.
\end{enumerate}
\end{Def}
\subsection{Our Assumptions and its Easy Consequences}
In this subsection, we discuss the assumptions used in this paper and their easy consequences.

\begin{Rem}\label{RemA}
Let us give several comments on Assumption \ref{Asu} and \ref{Asu3}.
\begin{enumerate}[(a)]
\item By Assumption \ref{Asu} (i), $\iota\colon M\to \R^d$ is injective, but $\iota_i$ is not necessarily so.
\item For the proof of our results, it is enough that Assumption \ref{Asu} (i) holds only locally, i.e., only for $x,y\in M$ with $\|\iota(x)-\iota(y)\|\leq r_0$ for some constant $r_0\in [\epsilon,\infty)$. However, this condition implies that $d(x,y)\leq \max\{L,\diam(M)/r_0\}\|x-y\|$ holds for any $x,y\in M$.
Under the assumption that the reach (Definition \ref{DR}) is bounded from below, $d(x,y)\leq \|x-y\|+C\|x-y\|^3$ holds locally (Corollary \ref{compreach}), so (i) holds for some constant $L\in (0,\infty)$.
\item Under Assumption \ref{Asu}, we have $\lim_{i\to \infty}\diam_{g_i}(M_i)=\diam(M)$ by (iii) and (iv).
\item Under Assumption \ref{Asu}, we call the sequence $\{((M_i,g_i),\iota)\}$ the approximation sequence for $(M,\iota)$.
We call the map $\psi_i\colon M\to M_i$ the approximation map.
Its properties imply that $(M_i,g_i)$ converges to $M$ in the sense of Gromov-Hausdorff.
Under Assumption \ref{Asu}, by the $C^{1,\alpha}$ compactness theorem \cite[Theorem 11.4.7]{Pe3}, it turns out that $M$ is a smooth manifold with a $C^{1,\alpha}$ Riemannian metric $g$ ($\alpha\in(0,1)$), and we can replace $\psi_i$ so that $\psi_i$ is a $C^{2,\alpha}$ diffeomorphism for sufficiently large $i$ and $\psi_i^\ast g_i\to g$ in $C^{1,\alpha}$, by taking a subsequence if necessary.
Therefore, $M$ satisfying Assumption \ref{Asu} will be treated as a Riemannian manifold $(M,g)$.
\item Suppose that we are given $(M,g)\in \M_1(m,K,i_0)$ and a function $\rho$ as Assumption \ref{Asu3} (ii).
Then, we have $\Vol_g(M)\leq \alpha$, so the volume comparison (Theorem \ref{VolLow}) implies $\diam_g(M)\leq C(m,K,i_0,\alpha)$.
\item Suppose that we are given an $m$-dimensional closed Riemannian manifold $(M,g)$ with $\Ric\geq -(m-1)K g$ for some $K>0$, and a function $\rho$ as Assumption \ref{Asu3} (ii).
Then, we have $\Vol_g(M)\geq 1/\alpha$, so the volume comparison (Theorem \ref{BishopGromov} (iii)) implies $\diam_g(M)\geq 1/C(m,K,\alpha)$.
In particular, for any $k\in \Z_{>0}$, we have $\lambda_k(\Delta)\leq C(m,K,\alpha,k)$ by \cite[Corollary 2.3]{Chen}.
\end{enumerate}
\end{Rem}
Our assumptions about the approximation lead to the following Lemma, independent of the curvature assumptions.
Note that Assumption \ref{Asu} implies the assumptions of the following lemma.
\begin{Lem}\label{UnifConv}
Suppose that we are given a compact metric space $M$ with distance function $d$ and a map $\iota\colon M\to \R^d$ such that there exist a sequence of compact metric spaces $\{(M_i,d_i)\}_{i=1}^\infty$, a sequence of positive real numbers $\{\epsilon_i\}_{i=1}^\infty\subset \R_{>0}$ with $\lim_{i\to\infty}\epsilon_i=0$, sequences of maps $\{\iota_i\colon M_i\to \R^d\}_{i=1}^\infty$ and $\{\psi_i\colon M\to M_i\}_{i=1}^\infty$ satisfying the following properties:
\begin{itemize}
\item[(i)] For any $i\in\Z_{>0}$ and $x,y\in M$, we have
$
|d(x,y)-d_{i}(\psi_i(x),\psi_i(y))|\leq \epsilon_i.
$
\item[(ii)] For any $i\in\Z_{>0}$ and $y\in M_i$, there exists $x\in M$ such that
$
d_{i}(y,\psi_i(x))\leq \epsilon_i.
$
\item[(iii)] For any $i\in\Z_{>0}$ and $x,y\in M_i$, we have
$\|\iota_i(x)-\iota_i(y)\|\leq d_i(x,y)$.
\item[(iv)] For any $x\in M$, we have
$
\lim_{i\to \infty}\iota_i(\psi_i(x))=\iota(x).
$
\end{itemize}
Then, we have the following properties:
\begin{itemize}
\item[(a)] For any $x,y\in M$, we have
$
\|\iota(x)-\iota(y)\|_{\R^d}\leq d(x,y).
$
\item[(b)] The convergence of  $(iv)$ is uniform, i.e.,
$$
\lim_{i\to \infty}\sup_{x\in M}\|\iota_i(\psi_i(x))-\iota(x)\|_{\R^d}=0.
$$
\item[(c)] Suppose that there exists $L\in(0,\infty)$ such that 
$
d(x,y)\leq L\|\iota(x)-\iota(y)\|_{\R^d}
$
holds for any $x,y\in M$.
Then, there exists a sequence $\{\tau_i\}\subset \R_{>0}$ such that
$\lim_{i\to \infty}\tau_i=0$ and for any $i\in\Z_{>0}$ and $x,y\in M_i$, we have
$$
d_{i}(x,y)\leq L\|\iota_i(x)-\iota_i(y)\|_{\R^d}+\tau_i.
$$ 
\end{itemize}
\end{Lem}
\begin{proof}
We immediately get (a) by (i), (iii) and (iv).

Let us prove (b).
Take arbitrary $\epsilon\in(0,\infty)$.
Then, by the compactness of $M$, we can find $x_1,\ldots, x_l\in M$ such that $M=\bigcup_{k=1}^l B_\epsilon(x_k)$.
Then, there exists $N\in\Z_{>0}$ such that
$\epsilon_i\leq \epsilon$ holds for any $i\in\Z_{>0}$ with $i\geq N$,
and
$$
\|\iota_i(\psi_i(x_k))-\iota(x_k)\|\leq \epsilon
$$
holds
for any $i\in\Z_{>0}$ with $i\geq N$ and $k\in\{1,\ldots,l\}$ by (iv).
Then, for any $i\in\Z_{>0}$ with $i\geq N$ and $y\in M$, we can find $k\in\{1,\ldots,l\}$ with $y\in B_{\epsilon}(x_k)$, and
\begin{align*}
\|\iota(y)-\iota_i(\psi_i(y))\|
\leq&\|\iota(y)-\iota(x_k)\|+\|\iota(x_k)-\iota_i(\psi_i(x_k))\|+\|\iota_i(\psi_i(x_k))-\iota_i(\psi_i(y))\|\\
\leq &d(y,x_k)+\epsilon+d_i(\psi_i(y),\psi_i(x_k))\leq 4\epsilon
\end{align*}
by (i), (ii) and (a).
This implies (b).

Let us prove (c).
Take arbitrary $x,y\in M_i$.
By (ii), there exist $x',y'\in M$ such that $d_i(x,\psi_i(x'))\leq \epsilon_i$ and $d_i(y,\psi_i(y'))\leq \epsilon_i$.
Then, by (i) and (iii), we have
\begin{align*}
d_i(x,y)
\leq &d_i(\psi_i(x'),\psi_i(y'))+2\epsilon_i
\leq d(x',y')+3\epsilon_i
\leq L\|\iota(x')-\iota(y')\|+3\epsilon_i\\
\leq &L\|\iota_i(\psi_i(x'))-\iota_i(\psi_i(y'))\|+2L\sup_{z\in M}\|\iota_i(\psi_i(z))-\iota(z)\|+3\epsilon_i\\
\leq &L\|\iota_i(x)-\iota_i(y)\|+2L\sup_{z\in M}\|\iota_i(\psi_i(z))-\iota(z)\|+(2L+3)\epsilon_i.
\end{align*}
Thus, we get (c) by (b).
\end{proof}

Several results hold under weaker condition than $\M_1(m,K,i_0)$, so we give the following definition.
\begin{Def}\label{Ric}
Let $m\in \Z_{>0}$ be an integer.
For each $K>0$, we define $\M_2=\M_2(m,K)$ to be the set of $m$-dimensional closed Riemannian manifolds $(M,g)$ with $\Ric_g\geq -(m-1)K g$.
\end{Def}
Note that the assumption $|\Sect_g|\leq K$ implies $\Ric_g\geq -(m-1)K g$, and so $\M_1(m,K,i_0)\subset \M_2(m,K)$.

Regarding Lemma \ref{UnifConv}, we consider the following assumptions for smooth immersions.
\begin{Asu}\label{Asu2}
Let $m,d\in \Z_{>0}$ be integers with $m<d$ and take constants $S,K,i_0,L,\tau>0$.
We say a pair $((M,g),\iota)$ of a Riemannian manifold $(M,g)$ and a map $\iota\colon M\to \R^d$ satisfies Assumption \ref{Asu2} (a) if the following conditions (i), (iii) and (iv) hold, and that it satisfies Assumption \ref{Asu2} (b) if (ii), (iii) and (iv) hold.
\begin{itemize}
\item[(i)] $(M,g)\in \M_1(m,K,i_0)$.
\item[(ii)] $(M,g)\in \M_2(m,K)$.
\item[(iii)] We have
$
\int_{M} |II|\,d\Vol_{g} \leq S.
$
\item[(iv)] For any $x,y\in M$,
we have
$
d(x,y)\leq L\|\iota(x)-\iota(y)\|_{\R^d}+\tau.
$
\end{itemize}
\end{Asu}
Assumption \ref{Asu2} (b) is weaker than (a). By replacing $L$ by $\max\{L,1\}$, we can assume that $L\geq 1$.
Our approach is to approximate the pair $(M,\iota)$ satisfying Assumption \ref{Asu} by $((M_i,g_i),\iota_i)$ satisfying Assumption \ref{Asu2} (a), and then show the Laplacian spectral approximation result for $((M_i,g_i),\iota_i)$.

\subsection{Integration in Sphere Bundles and Geodesic Flows}
In this subsection, we present some basic elements of integrals on sphere bundles, which will  be necessary for later discussions.
For details we refer Section 4 of Chapter II in \cite{Sa}.

Let $(M,g)$ be an $m$-dimensional closed Riemannian manifold, and $p\colon T M\to M$ be the tangent bundle, where $p$ denotes the projection to the base space $M$.
Take arbitrary $u\in T_x M$ ($x\in M$).
Then, the derivative of $p$ defines the map $d p=(d p)_u\colon T_u T M\to T_x M$, which is determined by $(d p_u)(\dot u(0))=\frac{d}{d s}|_{s=0}p(u(s))$ for any smooth curve $u(s)$ in $TM$ with $u(0)=u$.
We define an injection $i=i_u\colon T_x M\to T_u T M$ by $i_u(\xi)=\frac{d}{d s}|_{s=0}(u+s\xi)$ for any $\xi\in T_x M$, where $u+s\xi$ is regarded as a smooth curve in $T M$.
The maps $(d p)_u$ and $i_u$ are determined independently of the Riemannian metric, and the sequence
\begin{equation}\label{SES}
0\to T_x M\stackrel{i_u}{\longrightarrow}T_u TM\stackrel{(d p)_u}{\longrightarrow} T_x M\to 0
\end{equation}
is a short exact sequence.
The Levi-Civita connection $\nabla$ gives a splitting of (\ref{SES}) as follows.
We define a map $K=K_u\colon T_u TM\to T_x M$ so that for any curve $u(s)$ in $T M$ with $u(0)=u$, $K(\dot u(0))=\nabla_{\frac{\partial}{\partial s}} u(s)|_{s=0}$, where $u(s)$ is regarded as a vector field along $p(u(s))$.
Then, $K_u$ is well-defined and $K_u\circ i_u=\Id_{T_x M}$, so $K_u$ gives a splitting of $(\ref{SES})$.
We define $V_u:=\Imag i_u\subset T_u TM$ and $H_u:=\Ker K_u\subset T_u TM$.
Then, we have $T_u T M=V_u\oplus H_u$.
We define a Riemannian metric $G$ on $T M$ by
$$
G(\eta_1,\eta_2):=g\left((d p)_u (\eta_1),(d p)_u(\eta_2)\right)+ g\left(K_u(\eta_1),K_u(\eta_2)\right)
$$
for any $\eta_1,\eta_2\in T_u TM$.
Then, $T_u T M=V_u\oplus H_u$ is an orthogonal decomposition with respect to $G$.
Define a $2$-form $\alpha\colon T_u TM\times T_u TM\to\R$ by
$$
\alpha(\eta_1,\eta_2):=g\left((d p)_u (\eta_1),K_u(\eta_2)\right)- g\left(K_u(\eta_1),(d p)_u(\eta_2)\right)
$$
for any $\eta_1,\eta_2\in T_u TM$.
Then, identifying $T M$ and $T^\ast M$ through the map $TM\to T^\ast M,\, v\to g(v,\cdot)$, $\alpha$ gives the standard symplectic form on $T^\ast M$.
Given an appropriate orientation (even if $M$ is not orientable, $T M$ is always orientable), $\alpha^m/(m !)$ coincides with the volume form $\Vol_G$ on $TM$ for the metric $G$.
We regard $\Vol_G$ as the Riemannian volume measure on $TM$.
For any integrable function $F\colon TM\to\R$, we have
$$
\int_{TM} F \,d \Vol_G=\int_M\int_{T_x M} F(u)\,d u\,d x
$$
by \cite[II Lemma 5.6]{Sa}.
We define the geodesic flow $\phi_t \colon TM \to TM$ ($t\in\R$) by
$\phi_t(u):=\dot\gamma_u(t)$.
Then, $\phi_t$ preserves $\Vol_G$ since it preserves $\alpha$ \cite[II Lemma 4.4]{Sa}, so for any integrable function $F\colon TM\to\R$, we have
$$
\int_M\int_{T_x M} F(u)\,d u\,d x=\int_M\int_{T_x M} F(\phi_t(u))\,d u\,d x.
$$
We next consider the sphere bundle $U M:=\{v\in TM: g(v,v)=1\}$.
The sphere bundle $U M$ is a submanifold of $T M$ and $G$ defines a Riemannian metric on $U M$ by restriction.
For any $u\in U M$, $i_u (u)\in T_u TM$ is the unit normal vector of $UM$, so $(\Vol_{U M})_u:=\iota(i_u (u)) (\Vol_G)_u$ is the volume form at $u$ on $U M$ with appropriate orientation, where $\iota$ denotes the interior product.
For any $u\in UM$ and $t\in \R$, we have $\phi_t(u)\in UM$, and so the geodesic flow also defines the flow $\phi_t|_{U M}\colon UM\to UM$ on $U M$.
Then, $\phi_t|_{UM}$ preserves $\Vol_{U M}$ since we have $i_{\phi_t(u)}(\phi_t(u))=p_{V}\left((d\phi_t)_u (i_u(u))\right)$ for any $u\in UM$, where $p_{V}\colon T_{\phi_t(u)} TM\to V_{\phi_t(u)}$ denotes the orthogonal projection.
Regarding $\Vol_{U M}$ as the Riemannian volume measure on $UM$, the integral on $UM$ is computed in the same way as the integral on $TM$, and for any integrable function $F\colon UM\to \R$ and $t\in \R$ it follows that
\begin{equation}\label{GeodesicFlow}
\int_{UM}F\,d \Vol_{UM}=
\int_M\int_{U_x M} F(u)\,d u\,d x=\int_M\int_{U_x M} F(\phi_t(u))\,d u\,d x.
\end{equation}

\subsection{Basic Elements of Comparison Geometry}
In this subsection, we introduce some basic assertions about comparison geometry.
For each $K\in \R$, we define a function $s_K\colon\R\to \R$ by
$$
s_{K}(t):=\begin{cases}
\sqrt{\frac{1}{K}}\sin\sqrt{K}t &(K>0)\\
\qquad t&(K=0)\\
\sqrt{\frac{1}{-K}}\sinh\sqrt{-K}t & (K<0)
\end{cases}
$$
Let $(M,g)$ be an $m$-dimensional closed Riemannian manifold.
Take a point $x\in M$.
Let $\theta_x \colon \widetilde J_x\to \R_{>0}$ be the density function of the Riemannian volume measure through the identification $\exp_x\colon \widetilde J_x\to J_x\subset M$:
$$
\Vol_g=\theta_x \Le^m|_{\widetilde J_x}.
$$
Then, we have $\theta_x(0)=1$.
We summarize the Bishop-Gromov comparison theorem and its easy consequences as follows
\begin{Thm}\label{BishopGromov}
Suppose that $\Ric\geq (m-1)K$. Take arbitrary $u\in U_x M$.
\begin{itemize}
\item[(i)] We have
$$
\frac{t^{m-1}\theta_x(t u)}{s_K(t)^{m-1}}
$$
is monotonically non-increasing for $t\in (0,t(u))$.
\item[(ii)] We have $t^{m-1} \theta_x(t u)\leq s_K(t)^{m-1}$ for any $t\in(0,t(u))$.
\item[(iii)] For any $r,R\in(0,\infty)$ with $r\leq R\leq \sqrt{1/K}\pi$ $($$\sqrt{1/K}:=\infty$ if $K\leq 0$$)$, we have
\begin{align*}
\Vol_g(B_r(x))\leq& \Vol(S^{m-1})\int_0^r s_K(t)^{m-1}\,d t,\\
\Vol_g(B_r(x))\geq &\frac{\int_0^r s_K(t)^{m-1}\,d t}{\int_0^R s_K(t)^{m-1}\,d t}\Vol_g(B_R(x)).
\end{align*}
\item[(iv)] Suppose that $K<0$. Take arbitrary $R,s,t\in (0,\infty)$ with $s/2\leq t\leq s\leq R$ and $u\in U_x M$ with $s<t(u)$.
Then, we have
$$
\frac{s^{m-1}\theta_x(s u)}{t^{m-1}\theta_x(t u)}\leq \frac{s_K(s)^{m-1}}{s_K(t)^{m-1}}\leq \frac{s_K(s)^{m-1}}{s_K(s/2)^{m-1}}\leq \frac{s_K(R)^{m-1}}{s_K(R/2)^{m-1}}.
$$
\end{itemize}
\end{Thm}
For (i) and (ii), see \cite[IV Theorem 3.1 (2)]{Sa}.
If $K<0$ and $t\in [0,\theta\pi/\sqrt{-K}]$ for some $\theta\in(0,\infty)$, we have
\begin{align*}
s_K(t)^{m-1}\leq& t^{m-1}+\frac{\sinh^{m-1} (\theta\pi)-(\theta\pi)^{m-1}}{(\theta\pi)^{m+1}}(-K)t^{m+1},\\
s_K(t)^{m-1}\leq& \left(\sinh(\theta\pi)/(\theta\pi)\right)^{m-1}t^{m-1}.
\end{align*}
When, $K>0$, a similar estimate of $s_K^{m-1}$ is obtained for the lower bound. 
Integrating (i) and (ii), we get (iii). See \cite[IV Corollary 3.2 (2), Theorem 3.3]{Sa}.
If $K<0$ and $\diam(M)\leq D$ for some positive constant $D>0$, (iii) implies that for any $r\in(0,D]$,
$$
\Vol_g(B_r(x))\geq \frac{r^m}{m\int_0^D s_K(t)^{m-1}\,d t}\Vol_g(M).
$$
We get (iv) by (i) and properties of $s_K$.

We next consider the upper bound on the sectional curvature.
\begin{Thm}[IV Theorem 3.1 (1) of \cite{Sa}]\label{VolLow}
Suppose that $\Sect_g\leq K$.
Take arbitrary $u\in U_x M$.
Then, for any $t\in (0,\sqrt{1/K}\pi)$ with $t<t(u)$, we have
$$
t^{m-1}\theta_x(t u)\geq s_K(t)^{m-1}.
$$
\end{Thm}
If $K>0$, for any $t\in [0,\sqrt{1/K}\pi]$, we have
$$
s_K(t)^{m-1}\geq t^{m-1}-\frac{m-1}{6}K t^{m+1}.
$$
Thus, if $K>0$, under the assumptions of Theorem \ref{VolLow}, we have
$$
t^{m-1}\theta_x(t u)\geq t^{m-1}-\frac{m-1}{6}K t^{m+1}
$$
for any $t\in(0,t(u))$, since $\theta_x(t u)>0$.

Finally, we consider the comparison for the Hessian of distance function.
Define $d_x\colon M\to R$ by $d_x(y)=d(x,y)$.
\begin{Thm}[IV Lemma 2.9 of \cite{Sa}]\label{HessComp}
Suppose that $|\Sect_K|\leq K$.
Take arbitrary $u\in U_x M$ and $t\in (0,\sqrt{1/K}\pi)$ with $t <t(u)$.
Then, for any $v\in U_{\gamma_u(t)} M$ with $v\perp \dot\gamma_u(t)$, we have
$$
\frac{\dot s_K(t)}{s_K(t)}\leq (\Hess d_x)_{\gamma_u(t)}(v,v)\leq \frac{\dot s_{-K}(t)}{s_{-K}(t)}.
$$
\end{Thm}

Note that without the curvature condition, we have 
$$
(\Hess d_x)_{\gamma_u(t)}(\dot\gamma_u(t),v)=0
$$
for any $v\in T_{\gamma_u(t)}M$ if $0<t<t(u)$.
Suppose that $K>0$.
Then, for any $t\in(0,\pi/(2\sqrt{K})]$,
we have
\begin{align*}
0\leq \frac{\dot s_K(t)}{s_K(t)},\quad\frac{\dot s_{-K}(t)}{s_{-K}(t)}\leq \frac{1}{t}\frac{(\pi/2) \cosh (\pi/2)}{\sinh(\pi/2)}.
\end{align*}

\subsection{Berntein Inequality}
We frequently use the Bernstein inequality in the following form.
See Remark 1.4.4 and Theorem 2.2.1 of \cite{Vershynin}.
\begin{Thm}\label{Ber}
Let $(\Omega,\mathcal{F},\mu)$ be a probability space.
Let $f\in L^\infty(\mu)$ satisfies
\begin{align*}
\left|f(x)-\int_\Omega f\,d \mu\right|\leq& C\quad (x\in \Omega),\\
\int_\Omega f^2\,d \mu-\left(\int_\Omega f \, d\mu\right)^2\leq& \sigma^2
\end{align*}
for some $C,\sigma>0$.
Take arbitrary $t\in (0,\infty)$.
Then, for i.i.d. sample $x_1,\ldots,x_n\in \Omega$ from $\mu$, we have
$$
\frac{1}{n}\sum_{i=1}^n f(x_i)-\int_\Omega f\,d\mu\leq \frac{t}{n}
$$
holds with probability at least 
$$1-\exp\left(-\frac{t^2}{2 n \sigma^2+2C t/3}\right).$$
\end{Thm}
\begin{Rem}
Applying the above theorem to $-f$, we immediately get
$$
\left|\frac{1}{n}\sum_{i=1}^n f(x_i)-\int_\Omega f\,d\mu\right|\leq \frac{t}{n}
$$
with probability at least 
$$1-2\exp\left(-\frac{t^2}{2 n \sigma^2+2C t/3}\right).$$
\end{Rem}
\begin{Rem}
The claim that some condition $P$ holds for i.i.d. sample $x_1,...,x_n\in \Omega$ from $\mu$ with probability at least $a$  means that  there exists a measurable set $V\subset \Omega\times \cdots\times \Omega$ ($n$-times product) such that the condition $P$ holds for any $(x_1,\ldots,x_n)\in V$ and $\mu^{\otimes n}(V)\geq a$ holds, where $\mu^{\otimes n}$ denotes the $n$-times product measure of $\mu$.
\end{Rem}
\begin{Cor}\label{CorBer}
Let $(\Omega,\mathcal{F},\mu)$ be a probability space.
Take arbitrary $f\in L^\infty(\mu)$, $\gamma\in(1,\infty)$ and $\widetilde \delta\in(0,\infty)$ with $\gamma^{1/2}\widetilde\delta\leq 1$.
Then,  for i.i.d. sample $x_1,\ldots,x_n\in \Omega$ from $\mu$, we have
$$
\left|\frac{1}{n}\sum_{i=1}^n f(x_i)-\int_\Omega f\,d\mu\right|\leq 3\|f\|_{L^\infty}\gamma^{1/2}\widetilde\delta
$$
with probability at least $1-2\exp\left(-3n \gamma\widetilde\delta^2/2\right)$.
\end{Cor}
\begin{proof}
Putting $t=3 n \|f\|_{L^\infty} \gamma^{1/2}\widetilde \delta$ and applying Theorem \ref{Ber}, we get the assertion.
\end{proof}
\begin{Cor}\label{LinfBer}
Let $(\Omega,\mathcal{F},\mu)$ be a probability space.
Take arbitrary $f_1,\ldots, f_k\in L^\infty(\mu)$ $($$k\in\Z_{>0}$$)$, $\gamma\in(1,\infty)$ and $\widetilde \delta\in(0,\infty)$ with $\gamma^{1/2}\widetilde\delta\leq 1$.
Then,  for i.i.d. sample $x_1,\ldots,x_n\in \Omega$ from $\mu$, we have
$$
\left|\frac{1}{n}\sum_{i=1}^n f^2(x_i)-\int_\Omega f^2\,d\mu\right|\leq 3(3k-2)\max_{s}\{\|f_s\|^2_{L^\infty}\}\gamma^{1/2}\widetilde\delta
$$
for every $f=\sum_{s=1}^k a_s f_s$ $($$a_s\in \R$ with $\sum_{s=1}^k a_s^2=1$$)$
with probability at least $1-k(k+1)\exp\left(-3n \gamma\widetilde\delta^2/2\right)$.
\end{Cor}
\begin{proof}
Applying Corollary \ref{CorBer} to $f_s^2$ and $(f_s+f_t)^2$ ($s<t$), we get
\begin{align*}
\left|\frac{1}{n}\sum_{i=1}^n (f_s+f_t)^2(x_i)-\int_\Omega (f_s+f_t)^2\,d\mu\right|\leq 12\max_{s}\{\|f_s\|^2_{L^\infty}\}\gamma^{1/2}\widetilde\delta
\end{align*}
for every $s,t\in\{1,\ldots,k\}$ with probability at least $1-k(k+1)\exp\left(-3n \gamma\widetilde\delta^2/2\right)$.
This implies
$$
\left|\frac{1}{n}\sum_{i=1}^n f_s(x_i) f_t(x_i)-\int_\Omega f_s f_t\,d\mu\right|\leq 9\max_{s}\{\|f_s\|^2_{L^\infty}\}\gamma^{1/2}\widetilde\delta,
$$
and so we get the corollary.
\end{proof}

\section{Proofs of Main Results}
In this section, we give the proofs of our main results.
To do this, we compare the Rayleigh quotient of each eigenvalue problem  to use the minimax principle.
By the minimax principle, the eigenvalue $\lambda_k(\Delta_\rho)$ of $\Delta_\rho$ is expressed as
$$
\lambda_k(\Delta_\rho):=\inf\left\{\sup_{f\in V\setminus\{0\}}\frac{\int_M |\nabla f|^2\rho^2\,d\Vol_g}{\int_M f^2\rho\,d\Vol_g}:\begin{array}{l} V\subset W^{1,2}(M) \text{ is} \\
\text{a $(k+1)$-dimensional subspace}\end{array}\right\}.
$$
Given points $\X=\{x_1,\ldots,x_n\}\subset M$ and $\epsilon\in (0,\infty)$, the eigenvalue $\lambda_k(\La)$ of the matrix $\La$ defined as (\ref{defGLap}) has a similar expression using the Rayleigh quotient
$
\sum_{i=1}^n (\La u)_i u_i /\sum_{i=1}^n u_i^2
$
for $u\in \R^n$.
Here, a straightforward calculation implies 
$$\sum_{i=1}^n (\La u)_i u_i=\frac{1}{2}\sum_{i,j=1}^n \K_{i j} (u_i-u_j)^2= \frac{1}{2}\sum_{i,j=1}^n \eta\left(\frac{\|\iota(x_i)-\iota(x_j)\|}{\epsilon}\right) (u_i-u_j)^2.
$$
We compare these Rayleigh quotients through the maps $\LIP(M)\to \R^n, \,f\mapsto f|_{\X}$ and $\R^n\to\LIP(M),\, u\mapsto \Lambda_\epsilon u$. For the definition of $\Lambda_\epsilon$, see Definition \ref{defDtC} (v).
In subsection 3.1, given a function $f\in\LIP(M)$, we estimate $\sum_{i,j=1}^n \K_{i j} (f(x_i)-f(x_j))^2$.
In subsection 3.2, given a vector $u\in \R^n$, we estimate $\int_M |\nabla \Lambda_\epsilon u|^2\rho^2\,d\Vol_g$.
Combining these estimates, we show the main results for the unnormalized case in subsection 3.3.
For the normalized case, similar expressions of the eigenvalues hold, and we show the main results for this case in subsection 3.4.
\subsection{From Continuous to Discrete}
The following lemma is fundamental to the arguments under Assumption \ref{Asu2}.
\begin{Lem}\label{fund}
Let $(M,g)$ be a Riemannian manifold and $\iota\colon M\to \R^d$ be an isometric immersion.
Take a geodesic $\gamma\colon [0,l]\to M$ with unit speed, and suppose that
$$
\int_0^l |II|\circ \gamma(t)\,d t\leq S_{\gamma}
$$
holds for some positive constant $S_\gamma>0$.
Then, we have the following:
\begin{itemize}
\item[(i)] $\|\iota(\gamma(l))-\iota(\gamma(0))\|\geq \left(1-S_\gamma^2/2\right)l$.
\item[(ii)] $\|\iota(\gamma(l))-\iota(\gamma(0))\|\geq \left(1-S_\gamma/\sqrt{2}\right)l$.
\item[(iii)] If in addition $l\leq L\|\iota(\gamma(l))-\iota(\gamma(0))\|+\tau$ holds for some positive constants $L\geq 1$ and $\tau>0$, we have
$$
l\leq (1+C_L S_\gamma)\|\iota(\gamma(l))-\iota(\gamma(0))\|+\tau,
$$
where $C_L:=(L(L-1)/2)^{1/2}$.
\end{itemize}
\end{Lem}
\begin{proof}
In this proof, for the sake of brevity, we regard $\gamma$ as a curve in Euclidean space $\iota\circ\gamma$, and $\dot\gamma$ and $\ddot\gamma$ will denote the first-order and second-order derivatives as  a curve in Euclidean space, respectively.

We first prove (i).
Take $l_0\in (0,l)$ so that
$$
\int_0^{l_0} |II|\circ \gamma(t)\,d t=\int_{l_0}^l|II|\circ\gamma (t)\,d t\leq \frac{1}{2}S_\gamma
$$
and put
$$
v:=\dot\gamma(l_0)\in \R^d.
$$
Define $h\colon [0,l]\to \R$ by
$$
h(t):=\langle \dot\gamma(t),v\rangle_{\R^d},
$$
where $\langle\cdot,\cdot\rangle_{\R^d}$ denotes the Euclidean inner product.
Then, we have
$
h(l_0)=1,
$
$|h(t)|\leq 1$
and
\begin{equation}\label{gperp}
\|\ddot \gamma(t)\|^2 h(t)^2+h'(t)^2\leq \|\ddot\gamma(t)\|^2
\end{equation}
for any $t\in [0,l]$, since $\dot\gamma(t)\perp \ddot\gamma(t)$.
Define
\begin{align*}
l_+:=&\sup\left\{t\in[l_0,l]: \sup_{s\in [l_0,t]}(1-h(s)^2)\leq\left(\frac{2}{S_\gamma}\right)^2\right\},\\
l_-:=&\inf\left\{t\in[0,l_0]: \sup_{s\in [t,l_0]}(1-h(s)^2)\leq\left(\frac{2}{S_\gamma}\right)^2\right\},\\
A:=&\sup_{t\in[l_-,l_+]}(1-h(t)^2)^{1/2}.
\end{align*}
Then, by the continuity of $h$, we have $A\leq 2/S_\gamma$.
Since we have $\ddot\gamma=II(\dot\gamma,\dot\gamma)$, we have
$$
h'(t)^2\leq \|\ddot\gamma(t)\|^2(1-h(t)^2)\leq A^2\left(|II|\circ\gamma(t)\right)^2
$$
for any $t\in[l_-,l_+]$ by (\ref{gperp}).
This implie
\begin{equation}\label{hbou}
h(t)=1+\int_{l_0}^t h'(s)\,d s
\geq 1-\frac{1}{2}A S_\gamma
\end{equation}
for any $t\in [l_-,l_+]$.
For any $t\in [l_-,l_+]$, by $A S_\gamma/2\leq 1$ we get
$$
h(t)^2\geq \left(1-\frac{1}{2}A S_\gamma\right)^2\geq 1-A S_\gamma,
$$
so $1-h(t)^2\leq A S_\gamma$.
Thus, by the definition of $A$, we have $A\leq S_\gamma$, so
$$
\sup_{t\in[l_-,l_+]}(1-h(t)^2)\leq S_\gamma^2.
$$
Suppose that $S_\gamma<\sqrt{2}$.
Then, we have $S_\gamma^2<2<(2/S_\gamma)^2$, so
$l_+=l$ and $l_-=0$ by the continuity of $h$.
Therefore, we get
$$
h(t)\geq 1-\frac{1}{2}AS_\gamma\geq 1-\frac{1}{2}S_\gamma^2
$$
for any $t\in [0,l]$.
Thus, we have
\begin{align*}
\|\gamma(l)-\gamma(0)\|\geq\langle\gamma(l)-\gamma(0),v\rangle=\int_0^l h(t)\,d t\geq \left(1-\frac{1}{2}S_\gamma^2\right)l.
\end{align*}
If $S_\gamma\geq \sqrt 2$, this inequality holds automatically, so we get (i).

(ii) is an immediate consequence of (i).

Let us prove (iii).
We first suppose that $S_\gamma^2/2\leq (L-1)/L$.
Since we have $1+L x\geq 1/(1-x)$ for $x\in \R$ with $0\leq x\leq (L-1)/L$, we get
\begin{align*}
l\leq \frac{1}{1-S_\gamma^2/2}\|\gamma(l)-\gamma(0)\|\leq& \left(1+\frac{1}{2}LS_\gamma^2\right)\|\gamma(l)-\gamma(0)\|\\
\leq &(1+C_LS_\gamma)\|\gamma(l)-\gamma(0)\|
\end{align*}
by (i).
If $S_\gamma^2/2> (L-1)/L$, we have
\begin{align*}
l\leq (1+(L-1))\|\gamma(l)-\gamma(0)\|+\tau\leq (1+C_L S_\gamma)\|\gamma(l)-\gamma(0)\|+\tau
\end{align*}
by the assumption.
Thus, we get (iii) for both cases.
\end{proof}

The following Lemma gives a comparison of the integral form of the Riemannian and Euclidean distances.
In the proof, we use Assumption \ref{Asu2} (iv) only for $x,y\in M$ with $y\in B_\epsilon^{\R^d}(x)\setminus B_\epsilon(x)$.
\begin{Lem}\label{extvsint}
Suppose that the pair $((M,g),\iota)$ satisfies Assumption \ref{Asu2} $($b$)$, and we are given functions $\eta,\rho$ satisfying Assumption \ref{Asu3}.
Then, there exists a constant $C=C(m,S,K,L,\eta,\alpha)>0$ such that we have
\begin{align*}
&\int_{M\times M} \eta\left(\frac{\|\iota(x)-\iota(y)\|}{\epsilon}\right)(f(x)-f(y))^2\rho(x)\rho(y)\,d x\,d y\\
\leq& \int_{M\times M} \eta\left(\frac{d(x,y)}{\epsilon}\right)(f(x)-f(y))^2\rho(x)\rho(y)\,d x\,d y+C\Lip(f)^2 \epsilon^{m+3}\left(1+\frac{\tau}{\epsilon^2}\right)
\end{align*}
for any $\epsilon\in [\tau,\pi/(2\sqrt K)]$ and $f\in \LIP(M)$.
\end{Lem}
\begin{proof}
We have
\begin{equation}\label{er1}
\begin{split}
&\int_{M\times M}\eta\left(\frac{\|\iota(x)-\iota(y)\|}{\epsilon}\right) (f(x)-f(y))^2\rho(x)\rho(y)\,d x\,d y\\
\leq& \int_{M\times M}\eta\left(\frac{d(x,y)}{\epsilon}\right) (f(x)-f(y))^2\rho(x)\rho(y)\,d x\,d y\\
&+\alpha^2 L_\eta\Lip(f)^2\int_M\int_{B_\epsilon(x)}\frac{d(x,y)-\|\iota(x)-\iota(y)\|}{\epsilon}d(x,y)^2\,d y\,d x\\
&+\alpha^2\eta(0)\Lip(f)^2\int_M\int_{B_\epsilon^{\R^d}(x)\setminus B_\epsilon(x)}d(x,y)^2\,d y\,d x.
\end{split}
\end{equation}
We estimate each error terms.

Let us estimate the second term.
For each $u\in UM$, put
$$
S_u:=\int_0^\epsilon |II|\circ \gamma_u(t)\,d t.
$$
For any $u\in U M$, we have
$$
t-\|\iota(\gamma_u(t))-\iota(\gamma_u(0))\|\leq \frac{1}{\sqrt{2}}S_u t
$$
for any $t\in[0,\epsilon]$ by Lemma \ref{fund} (ii).
Thus, we get
\begin{equation}\label{er2}
\begin{split}
&\int_M\int_{B_\epsilon(x)}\frac{d(x,y)-\|\iota(x)-\iota(y)\|}{\epsilon}d(x,y)^2\,d y\,d x\\
\leq& (1+C \epsilon^2)\int_M \int_{U_x M}\int_0^\epsilon \frac{t^{m+2}}{\epsilon} \frac{1}{\sqrt{2}}S_u\, d t\,d u \,d x\leq C\epsilon^{m+3} \int_M|II|\, d\Vol_g\leq C \epsilon^{m+3}
\end{split}
\end{equation}
by Theorem \ref{BishopGromov} (ii) and (\ref{GeodesicFlow}).

We estimate the third term.
For each $u\in U M$, put
$$
\widetilde S_u:=\int_0^{(L+1) \epsilon} |II|\circ \gamma_u(t)\,d t.
$$
Then, we have
$$
B_\epsilon^{\R^d}(x) \setminus B_{\epsilon}(x)\subset\Big\{\gamma_u (t): x\in M,\, \epsilon\leq t\leq \min\big\{(L+1)\epsilon,(1+C_L \widetilde S_u)\epsilon+\tau\big\}\Big\}
$$
for any $x\in M$ by Lemma \ref{fund} (iii) and the assumption $\tau\leq \epsilon$.
Thus, we get
\begin{equation}\label{er3}
\begin{split}
&\int_M\int_{B_\epsilon^{\R^d}(x)\setminus B_\epsilon(x)}d(x,y)^2\,d y\,d x\\
\leq &(1+C \epsilon^2)\int_M\int_{U_x M}\int_r^{\min\big\{(L+1)\epsilon,(1+C_L \widetilde S_u)\epsilon+\tau\big\}} t^{m+1}\, d t \, d u \, d x\\
\leq& C\epsilon^{m+2}\int_M\int_{U_x M} \left(C_L \widetilde S_u +\frac{\tau}{\epsilon}\right)\, d u \, d x
\leq C \epsilon^{m+3} \left(1+\frac{\tau}{\epsilon^2}\right)
\end{split}
\end{equation}
similarly to (\ref{er2}).

By (\ref{er1}), (\ref{er2}) and (\ref{er3}), we get the lemma.
\end{proof}


The following lemma corresponds to \cite[Lemma 5]{TGHS}, but it holds for any Riemannian manifold $(M,g)$ in $\M_2(m,K)$, and this assumption is weaker than the one in \cite[Lemma 5]{TGHS}.

\begin{Lem}\label{nlcvslc}
Suppose that we are given $(M,g)\in \M_2(m,K)$, functions $\eta,\rho$ satisfying Assumption \ref{Asu3}.
Then, there exists a constant $C=C(m,K,R,\alpha,L_\rho)>0$ such that for any $\epsilon\in (0,\pi/(2\sqrt K)]$ and $f\in W^{1,2}(M)$, we have
\begin{align*}
&\frac{1}{\epsilon^{m+2}}\int_{M\times M} \eta\left(\frac{d(x,y)}{\epsilon}\right)(f(x)-f(y))^2\rho(x)\rho(y)\,d x\,d y\\
\leq& (1+C \epsilon)\sigma_{\eta}\int_M|\nabla f|^2 \rho^2\,d \Vol_g.
\end{align*}
\end{Lem}
\begin{proof}
For $x,y\in M$, we have
$$
f(y)-f(x)=\int_0^{d(x,y)}\langle\nabla f,\dot\gamma_{x,y}(s)\rangle\,d s.
$$
Thus, for any $x\in M$,
\begin{align*}
&\int_{M}\eta\left(\frac{d(x,y)}{\epsilon}\right)(f(x)-f(y))^2\rho(x)\rho(y)\,d y\\
\leq &(1+C \epsilon^2)\int_{U_x M}\int_0^\epsilon
\eta\left(\frac{t}{\epsilon}\right) t^{m-1}
\left(\int_0^{t}\langle\nabla f,\dot\gamma_{u}(s)\rangle\,d s\right)^2\rho(x)\rho(\gamma_u(t))\,d t\,d u\\
\leq &(1+C \epsilon)\int_{U_x M}\int_0^\epsilon
\eta\left(\frac{t}{\epsilon}\right) t^{m}
\left(\int_0^{t}\langle\nabla f,\dot\gamma_{u}(s)\rangle^2 \rho(\gamma_u(s))^2\,d s \right)\,d t\,d u.
\end{align*}
Here, we used Theorem \ref{BishopGromov} (ii), the Cauchy-Schwarz inequality and $\rho(x)\rho(\gamma_u(t))\leq (1+C \epsilon)\rho(\gamma_u(s))^2$.
Therefore,
\begin{align*}
&\int_{M\times M} \eta\left(\frac{d(x,y)}{\epsilon}\right)(f(x)-f(y))^2\rho(x)\rho(y)\,d x\,d y\\
\leq &(1+C \epsilon)\int_0^\epsilon  
\eta\left(\frac{t}{\epsilon}\right) t^{m}
\left(\int_0^{t}\int_M \int_{U_x M}\langle\nabla f,\dot\gamma_{u}(s)\rangle^2 \rho(\gamma_u(s))^2\,d u\,d x\,d s \right)\,d t\\
=&(1+C \epsilon)\int_0^\epsilon
\eta\left(\frac{t}{\epsilon}\right) t^{m+1}
\left(\int_M \int_{U_x M}\langle\nabla f,u \rangle^2 \rho(x)^2\,d u\,d x \right)\,d t\\
=&(1+C \epsilon)\sigma_\eta \epsilon^{m+2}\int_M|\nabla f|^2\rho^2\,d \Vol_g.
\end{align*}
Here, we used (\ref{GeodesicFlow}) and $\int_{U_x M}\langle\nabla f,u\rangle^2\, d u=|\nabla f|^2(x) \Vol(S^{m-1})/m$.
This implies the lemma.
\end{proof}

The following Lemma gives an approximation of the integral on the Riemannian manifold by summing over discrete points.
\begin{Lem}\label{etaBer}
Suppose that the pair $((M,g),\iota)$ satisfies Assumption \ref{Asu2} $(b)$, and we are given functions $\eta,\rho$ satisfying Assumption \ref{Asu3}.
Then, there exists a constant $C=C(m,L,\eta(0),\alpha,L_\rho)>0$ such that, for any $\epsilon\in [\tau,\pi/(2\sqrt K)]$, $\gamma\in(1,\infty)$, $\delta\in(0,\gamma^{-1/2}]$, $f\in \LIP(M)$ and i.i.d. sample $x_1,\ldots, x_n\in M$ from $\rho\Vol_g$,
we have
\begin{align*}
&\Bigg|\frac{1}{\epsilon^{m+2}n^2}\sum_{i,j=1}^n\eta\left(\frac{\|\iota(x_i)-\iota(x_j)\|}{\epsilon}\right)(f(x_i)-f(x_j))^2\\
&\,\,-\frac{1}{\epsilon^{m+2}}\int_{M\times M} \eta\left(\frac{\|\iota(x)-\iota(y)\|}{\epsilon}\right)(f(x)-f(y))^2\rho(x)\rho(y)\,d x\, d y\Bigg|\leq C\Lip(f)^2\gamma^{1/2}\delta
\end{align*}
with probability at least
$
1-2(n+1)\exp\left(-n\gamma \epsilon^m\delta^2\right).
$
\end{Lem}
\begin{proof}
For any $x\in M$, define $h_x\colon M\to \R_{\geq 0}$ by
$$
h_x(y):=\eta\left(\frac{\|\iota(x)-\iota(y)\|}{\epsilon} \right)(f(x)-f(y))^2,
$$ 
and define $H\colon M\to \R_{\geq 0}$ by
$$
H(x):=\int_M h_x(y)\rho(y)\,d y.
$$
Then, there exist constant $C_1=C_1(m,L,\eta(0),\alpha)>0$ such that
\begin{align*}
0\leq h_x(y)\leq& \eta(0)\Lip(f)^2(L+1)^2\epsilon^2,\\
\int_M h_x(y)^2\rho(y)\,d y\leq& C_1\Lip(f)^4 \epsilon^{m+4}
\end{align*}
for every $x,y\in M$.
Thus, by the Bernstein inequality and the Fubini theorem, we have
\begin{align*}
&(\rho \Vol_g)^{\otimes n}\left(\left\{(x_1,\ldots,x_n): \left|\frac{1}{n-1}\sum_{j\neq i}h_{x_i}(x_j)-H(x_i)\right|\geq \frac{t}{n-1}\right\}\right)\\
\leq&\int_{M}(\rho \Vol_g)^{\otimes (n-1)}\left(\left\{(x_1,\ldots,\check{x}_i,\ldots ,x_n): \left|\frac{1}{n-1}\sum_{j\neq i}h_{x}(x_j)-H(x)\right|\geq \frac{t}{n-1}\right\}\right)\rho(x)\,d x\\
\leq &2\exp\left(-\frac{t^2}{2n C_1\Lip(f)^4 \epsilon^{m+4}+\frac{2}{3}\eta(0)\Lip(f)^2(L+1)^2\epsilon^2 t}\right).
\end{align*}
for each $i=1,\ldots, n$ and $t>0$.
Here, $(x_1,\ldots,\check{x}_i,\ldots ,x_n):=(x_1,\ldots,x_{i-1},x_{i+1},\ldots ,x_n)$.
Putting $$t=\max\{2\sqrt{C_1},4\eta(0)(L+1)^2/3\}n\Lip(f)^2 \epsilon^{m+2}\gamma^{1/2}\delta,$$ we get
\begin{equation}\label{esh}
\left|\frac{1}{n-1}\sum_{j\neq i}h_{x_i}(x_j)-H(x_i)\right|\leq C\Lip(f)^2 \epsilon^{m+2}\gamma^{1/2}\delta
\end{equation}
for every $i=1,\ldots, n$ with probability at least $1-2 n \exp(-n\gamma \epsilon^m\delta^2)$.
Note that we used the assumption $\gamma^{1/2}\delta\leq 1$ here.

Similarly, by
\begin{align*}
H(x)\leq& C\Lip(f)^2 \epsilon^{m+2},\\
\int_M H(x)^2\rho(x)\,d x\leq & C\Lip(f)^4 \epsilon^{2m+4}
\end{align*}
and the Bernstein inequality, we have
\begin{equation}\label{esH}
\left|\frac{1}{n}\sum_{i=1}^n H(x_i) -\int_M H(x)\rho(x)\,d x\right|\leq C \Lip(f)^2 \epsilon^{\frac{3}{2}m+2}\gamma^{1/2}\delta
\end{equation}
with probability at least $1-2\exp(-n\gamma \epsilon^m\delta^2)$.

If (\ref{esh}) and (\ref{esH}) hold, we have
\begin{align*}
&\left|\frac{1}{n^2}\sum_{i,j= 1}^n h_{x_i}(x_j)-\frac{n-1}{n}\int_M H(x)\rho(x)\,d x\right|\\
\leq &\frac{n-1}{n^2}\sum_{i=1}^n\left|\frac{1}{n-1}\sum_{j= 1}^n h_{x_i}(x_j)-H(x_i)\right|+\left|\frac{1}{n}\sum_{i=1}^n H(x_i)-\int_M H(x)\rho(x)\,d x\right|\\
\leq & C\Lip(f)^2\epsilon^{m+2}\gamma^{1/2}\delta.
\end{align*}
Since we can assume $1/n\leq C\gamma\epsilon^m \delta^2\leq C\gamma^{1/2}\delta$ (otherwise, our probability is zero), this implies the lemma.
\end{proof}
\begin{Cor}\label{etakdim}
Suppose that the pair $((M,g),\iota)$ satisfies Assumption \ref{Asu2} $(b)$, and we are given functions $\eta,\rho$ satisfying Assumption \ref{Asu3}.
Then, there exists a constant $C=C(m,L,\eta(0),\alpha,L_\rho)>0$ such that, for any $\epsilon\in[\tau,\pi/(2\sqrt K)]$, $\gamma\in(1,\infty)$, $\delta\in(0,\gamma^{-1/2}]$, $f_1,\ldots,f_k\in \LIP(M)$ $($$k\in \Z_{>0}$$)$ and i.i.d. sample $x_1,\ldots, x_n\in M$ from $\rho\Vol_g$,
we have
\begin{align*}
&\Bigg|\frac{1}{\epsilon^{m+2}n^2}\sum_{i,j=1}^n\eta\left(\frac{\|\iota(x_i)-\iota(x_j)\|}{\epsilon}\right)(f(x_i)-f(x_j))^2\\
&\,\,-\frac{1}{\epsilon^{m+2}}\int_{M\times M} \eta\left(\frac{\|\iota(x)-\iota(y)\|}{\epsilon}\right)(f(x)-f(y))^2\rho(x)\rho(y)\,d x\, d y\Bigg|\\
\leq& C k\max_{s}\{\Lip(f_s)^2\} \gamma^{1/2}\delta
\end{align*}
for every $f=\sum_{s=1}^k a_s f_s$ $($$a_s\in \R$ with $\sum_{s=1}^k a_s^2=1$$)$
with probability at least
$
1-k(k+1)(n+1)\exp\left(-n\gamma \epsilon^m\delta^2\right).
$
\end{Cor}
\begin{proof}
Define symmetric forms $T,\widetilde T\colon \R^k\times \R^k\to \R$ by
\begin{align*}
&\widetilde T(a,b)\\
=&\frac{1}{\epsilon^{m+2} n^2}\sum_{s,t=1}^k\sum_{i,j=1}^n\eta\left(\frac{\|\iota(x_i)-\iota(x_j)\|}{\epsilon}\right)a_s b_t \left(f_s(x_i)-f_s(x_j)\right)\left(f_t(x_i)-f_t(x_j)\right)\\
&\widetilde T(a,b)\\
=&\frac{1}{\epsilon^{m+2}}\sum_{s,t=1}^k\int_{M\times M}\eta\left(\frac{\|\iota(x)-\iota(y)\|}{\epsilon}\right)a_s b_t \left(f_s(x)-f_s(y)\right)\left(f_t(x)-f_t(y)\right)\rho(x)\rho(y)\,d x\, d y.
\end{align*}
Take a standard basis $\{e_1,\ldots,e_k\}$ of $\R^k$.
Applying Lemma \ref{etaBer} to $f_s$ and $f_s+f_t$ ($s< t$), we get
\begin{equation}
\begin{split}
\left|T(e_s,e_t)-\widetilde T(e_s,e_t)\right|\leq C\left(\Lip(f_s)^2+\Lip(f_t)^2\right)\gamma^{1/2}\delta
\end{split}
\end{equation}
for every $s,t\in\{1,\ldots, k\}$ with probability at least $1-k(k+1)(n+1)\exp\left(-n\gamma \epsilon^m\delta^2\right).$
For any $a\in \R^k$ with $\sum_{s=1}^k a_s^2=1$, we have
\begin{align*}
\left|T(a,a)-\widetilde T(a,a)\right|\leq&\sum_{s,t=1}^k |a_s a_t||T(e_s,e_t)-\widetilde T(e_s,e_t)|\\
\leq& C \max_{s}\{\Lip(f_s)^2\} \gamma^{1/2}\delta\sum_{s,t=1}^k |a_s a_t|\leq C k\max_{s}\{\Lip(f_s)^2\} \gamma^{1/2}\delta.
\end{align*}
This implies the corollary.
\end{proof}
The following lemma is the goal of this subsection.
\begin{Lem}\label{UPP}
Suppose that the pair $((M,g),\iota)$ satisfies Assumption \ref{Asu2} $(b)$, and we are given functions $\eta,\rho$ satisfying Assumption \ref{Asu3}.
Then, there exists a constant $C=C(m,S,K,L,\eta,\alpha,L_\rho)>0$ such that, for any $\epsilon\in[\tau,\pi/(2\sqrt K)]$, $\gamma\in(1,\infty)$, $\delta\in(0,\gamma^{-1/2}]$, $f_1,\ldots,f_k\in \LIP(M)$ $($$k\in\Z_{>0}$$)$ and i.i.d. sample $x_1,\ldots, x_n\in M$ from $\rho\Vol_g$,
we have
\begin{align*}
&\frac{1}{\epsilon^{m+2}n^2}\sum_{i,j=1}^n\eta\left(\frac{\|\iota(x_i)-\iota(x_j)\|}{\epsilon}\right)(f(x_i)-f(x_j))^2\\
\leq&
\sigma_\eta \int_M|\nabla f|^2\rho^2\,d \Vol_g+C k\max_{s}\{\Lip(f_s)^2\}\left(\epsilon+\frac{\tau}{\epsilon}+\gamma^{1/2}\delta\right)
\end{align*}
for every $f=\sum_{s=1}^k a_s f_s$ $($$a_s\in \R$ with $\sum_{s=1}^k a_s^2=1$$)$ with probability at least
$
1-k(k+1)(n+1)\exp\left(-n\gamma \epsilon^m\delta^2\right).
$
\end{Lem}
\begin{proof}
By Lemma \ref{extvsint}, \ref{nlcvslc} and Corollary \ref{etakdim}, we get
\begin{align*}
&\frac{1}{\epsilon^{m+2}n^2}\sum_{i,j=1}^n\eta\left(\frac{\|\iota(x_i)-\iota(x_j)\|}{\epsilon}\right)(f(x_i)-f(x_j))^2\\
\leq&\frac{1}{\epsilon^{m+2}}\int_{M\times M} \eta\left(\frac{\|\iota(x)-\iota(y)\|}{\epsilon}\right)(f(x)-f(y))^2\rho(x)\rho(y)\,d x\,d y+Ck\max_{s}\{\Lip(f_s)^2\} \gamma^{1/2}\delta\\
\leq& \sigma_\eta\int_{M}|\nabla f|^2 \rho^2\,d \Vol_g+Ck\max_{s}\{\Lip(f_s)^2\} \left(\epsilon+\frac{\tau}{\epsilon}+\gamma^{1/2}\delta\right)
\end{align*}
for every $f=\sum_{s=1}^k a_s f_s$ $($$a_s\in \R$ with $\sum_{s=1}^k a_s^2=1$$)$ with probability at least
$
1-k(k+1)(n+1)\exp\left(-n\gamma \epsilon^m\delta^2\right).
$
\end{proof}

\subsection{From Discrete to Continuous}

As a direct consequence of Definition \ref{defDtC}, we get the following.
\begin{Lem}\label{q2q1f-f}
Suppose that we are given a closed Riemannian manifold $(M,g)$, points
$\X=\{x_1,\ldots,x_n\}\subset M$, a positive real number $\epsilon\in(0,\infty)$
and functions $\eta,\rho$ as Assumption \ref{Asu3}.
If $\theta_\epsilon(x)>0$ holds for any $x\in M$, then we have
$$
\left|\Lambda_\epsilon (f|_{\X})(x)-f(x)\right|\leq \Lip(f)\epsilon
$$
for any $f\in \LIP(M)$ and $x\in M$.
\end{Lem}
\begin{proof}
Since we have $\sum_{i=1}^n \widetilde \psi(x,x_i)/n=1$ for any $x\in M$, and $\widetilde\psi(x,y)=0$ for any $x,y\in M$ with $d(x,y)\geq \epsilon$, we get
\begin{align*}
\left|\Lambda_\epsilon (f|_{\X})(x)-f(x)\right|=&\left|\frac{1}{n}\sum_{i=1}^n \widetilde \psi(x, x_i)(f(x_i)-f(x))\right|\\
\leq& \frac{1}{n}\sum_{i=1}^n \widetilde \psi(x, x_i)\Lip(f)\epsilon =\Lip(f)\epsilon.
\end{align*}
Thus, we obtain the lemma.
\end{proof}

\begin{Lem}\label{theta}
Suppose that we are given $(M,g)\in \M_1(m,K,i_0)$, functions $\eta,\rho$ as Assumption \ref{Asu3}.
Then, there exists a constant $C=C(m,K,\eta,\alpha,L_\rho)>0$ such that for any $\epsilon \in\left(0,\min\{i_0,\pi/(2\sqrt K)\}\right)$ and $x\in M$, we have
\begin{itemize}
\item[(i)] $|\overline\theta_\epsilon(x)-\rho(x)\epsilon^m\sigma_\eta|\leq C\epsilon^{m+1}$,
\item[(ii)] $|\nabla \overline \theta_\epsilon(x)|\leq C \epsilon^m$.
\end{itemize}
\end{Lem}
\begin{proof}
We first prove (i).
For any $x\in M$, by the comparison theorem for the volume element (Theorem \ref{BishopGromov} (ii) and \ref{VolLow}), we have
\begin{align*}
&(1-C \epsilon)\rho(x) \Vol(S^{m-1})\int_0^{\epsilon}\psi\left(\frac{t}{\epsilon}\right)t^{m-1}\, d t\\
\leq &\overline\theta_\epsilon(x)
\leq (1+C \epsilon)\rho(x) \Vol(S^{m-1})\int_0^{\epsilon}\psi\left(\frac{t}{\epsilon}\right)t^{m-1}\, d t.
\end{align*}
Since we have
$$
\Vol(S^{m-1})\int_0^{\epsilon}\psi\left(\frac{t}{\epsilon}\right)t^{m-1}\, d t=\epsilon^m\sigma_{\eta},
$$
we get (i).

We next prove (ii).
For any $x\in M$ and $v\in T_x M$, we have
\begin{align*}
&|\langle\nabla \overline\theta_r(x),v\rangle|\\
=& \left| \int_{B_{\epsilon}(x)}\eta\left(\frac{d(x,y)}{\epsilon}\right)\frac{d(x,y)}{\epsilon^2}\langle\dot\gamma_{x,y}(0),v \rangle\rho(y)\, d y\right|\\
\leq &\left| \rho(x)\int_{U_x M}\int_0^\epsilon\eta\left(\frac{t}{\epsilon}\right)\frac{t^m}{\epsilon^2}\langle u,v \rangle\, d t \, d u\right|
+C \epsilon\left|\int_{U_x M}\int_0^\epsilon\eta\left(\frac{t}{\epsilon}\right)\frac{t^m}{\epsilon^2}|v|\, d t \, d u\right|\\
\leq & C \epsilon^m |v|
\end{align*}
by $\int_{U_x M} \langle u, v\rangle\,d u=0$.
Thus, we get (ii).
\end{proof}
\begin{Lem}\label{probt}
Suppose that we are given $(M,g)\in \M_1(m,K,i_0)$ and functions $\eta,\rho$ satisfying Assumption \ref{Asu3}.
Then, there exist constants $C_1=C_1(m,\alpha)>0$ and $C_2=C_2(m,K,\eta,\alpha,L_\rho)>0$ such that, for any $\epsilon\in(0,\min\{1,2i_0/3,\pi/(2\sqrt K)\})$, $\gamma\in(1,\infty)$, $\delta\in(0,\gamma^{-1/2}]$ and i.i.d. sample $x_1,\ldots, x_n\in M$ from $\rho\Vol_g$, we have
the following properties with probability at least $1-C_1(\epsilon^{-2m-1}+n^2)\exp\left(-n\gamma \epsilon^m\delta^2\right)$.
\begin{itemize}
\item[(i)] For any $x\in M$, we have $|\theta_\epsilon(x)-\rho(x)\sigma_\eta \epsilon^m|\leq C_2\epsilon^m(\epsilon +\gamma^{1/2}\delta)$.
\item[(ii)] For a.e. $x\in M$, we have $|\nabla \theta_\epsilon(x)|\leq C_2\epsilon^m(1+\gamma^{1/2}\delta/\epsilon)$.
\item[(iii)] For any $i$, we have
$$
\frac{1}{n-1}\sum_{j=1}^n\eta\left(\frac{d(x_i,x_j)}{\epsilon}\right)\leq C_2 \epsilon^m.
$$
\item[(iv)] For any $i$ and $w\in T_{x_i} M$, we have
\begin{align*}
&\frac{1}{n-1}\sum_{j=1}^n\eta\left(\frac{d(x_i,x_j)}{\epsilon}\right)\frac{d(x_i,x_j)^2}{\epsilon^2}\langle \dot \gamma_{x_i,x_j}(0),w\rangle^2\\
\leq &\epsilon^m\rho(x)\sigma_\eta \big(1+C_2(\epsilon+\gamma^{1/2}\delta)\big)|w|^2.
\end{align*}
\item[(v)] For any $i,j$, we have
\begin{align*}
&\frac{1}{n}\sum_{l=1}^n\langle\Psi(x_l,x_i),\Psi(x_l,x_j)\rangle\frac{1}{\rho(x_l)}\\
\leq &\int_M\langle\Psi(x,x_i),\Psi(x,x_j)\rangle \,d x+C_2\epsilon^{m-2}\gamma^{1/2}\delta.
\end{align*}
\end{itemize}
\end{Lem}
\begin{proof}

There exist a constant $C_3=C_3(m,\alpha)>0$ and an integer $N_1\in \Z_{>0}$ with $N_1\leq C_3 \epsilon^{-2m-1}$ such that there exist points $y_1,\ldots, y_{N_1}\in M$ with
\begin{equation}\label{cover}
M=\bigcup_{s=1}^{N_1}B(y_s,\epsilon^{2+\frac{1}{m}}),
\end{equation}
since we have $\Vol_g\left(B(y,\epsilon^{2+\frac{1}{m}}/2)\right)\geq (\sqrt 2/\pi)^{m-1}\epsilon^{2m+1}\Vol(S^{m-1})/(2m)$ for any $y\in M$ by Theorem \ref{VolLow} and $\epsilon^{2+\frac{1}{m}}/2\leq \pi/(4\sqrt{K})$.

We first show (i) and (ii) for fixed $y=y_s$ ($s\in\{1,\ldots,N_1\}$).
We have
\begin{align*}
\psi\left(\frac{d(y,z)}{\epsilon}\right)\leq &\frac{1}{2}\eta(0) \quad (z\in M),\\
\int_M \psi\left(\frac{d(y,z)}{\epsilon}\right)^2\rho(z)\,d z\leq& C \epsilon^m,\\
|\Psi(y,z)|\leq& \frac{\eta(0)}{\epsilon}\quad (z\in M),\\
\int_M|\Psi(y,z)|^2\rho(z)\,d z\leq& C\epsilon^{m-2}.
\end{align*}
Thus, by the Bernstein inequality, we have
\begin{align*}
|\theta_\epsilon(y)-\overline\theta_\epsilon(y)|\leq& C \gamma^{1/2}\epsilon^m \delta,\\
\left|\frac{1}{n}\sum_{i=1}^n\Psi(y,x_i)-\nabla \overline\theta_\epsilon(y)\right|\leq& C\gamma^{1/2}\epsilon^{m-1}\delta
\end{align*}
with probability at least $1-(2m+2)\exp(-n\gamma\epsilon^m\delta^2)$, and so
\begin{align}
\label{ty}|\theta_\epsilon(y)-\rho(y)\epsilon^m \sigma_\eta|\leq& C \epsilon^m(\epsilon+\gamma^{1/2} \delta),\\
\label{nty}\left|\frac{1}{n}\sum_{i=1}^n\Psi(y,x_i)\right|\leq& C\epsilon^{m-1}(\epsilon+\gamma^{1/2} \delta)
\end{align}
by Lemma \ref{theta}.
Here, we used an identification of $T_y M\cong\R^m$ through an orthonormal basis, and applied the Bernstein inequality to each component of $\Psi(y,\cdot)\colon M\to \R^m$.
 
We next show (i) and (ii) for arbitrary $x\in M$ with $d(x,y)<\epsilon^{2+\frac{1}{m}}$, where $y=y_s$ for some $s$.
Since we have $|\psi'|\leq \eta(0)$, we get
\begin{equation*}
\left|\psi\left(\frac{d(x,z)}{\epsilon}\right)-\psi\left(\frac{d(y,z)}{\epsilon}\right)\right|
\leq \eta(0)\frac{d(x,y)}{\epsilon}\leq \eta(0) \epsilon^{1+\frac{1}{m}}
\end{equation*}
for any $z\in M$, and so
\begin{equation}\label{tx}
\begin{split}
&|\theta_\epsilon(x)-\rho(x)\epsilon^m\sigma_\eta|\\
\leq& |\theta_\epsilon(x)-\theta_\epsilon(y)|
+|\rho(x)-\rho(y)|\epsilon^m\sigma_\eta+|\theta_\epsilon(y)-\rho(y)\epsilon^m \sigma_\eta|\\
\leq&\eta(0)\frac{\epsilon^{1+\frac{1}{m}}}{n} \Card\big\{i:x_i\in B_\epsilon(x)\cup B_\epsilon(y)\big\}+
C\epsilon^{m}(\epsilon+\gamma^{1/2} \delta)
\end{split}
\end{equation}
by (\ref{ty}).
Put $c:=c_{x,y}\colon[0,1]\to M$ and suppose that $j$ satisfies $x_j\in B_{\epsilon}(x)\cap B_{\epsilon}(y)$ and $d(x_j,y)\geq \epsilon^{1+\frac{1}{m}}+\epsilon^{2+\frac{1}{m}}$. 
Then, $\epsilon^{1+\frac{1}{m}}\leq d(x_j,c(t))< i_0$ for any $t\in [0,1]$.
Let $w\in\Gamma(c^\ast TM)$ be a parallel vector field along $c$ with $|w|\equiv1$.
By the Hessian comparison theorem (Theorem \ref{HessComp}), we have
$$
\left|\frac{d}{d t}\langle\dot \gamma_{c(t),x_j}(0),w(t)\rangle\right|=\left|(\Hess d_{x_j})_{c(t)}(\dot c(t),w(t))\right|\leq C \frac{|\dot c(t)|}{d(x_j,c(t))}\leq C \epsilon,
$$
and so
$
|\langle\dot\gamma_{x,x_j}(0),w(0)\rangle-\langle\dot\gamma_{y,x_j}(0),w(1)\rangle|\leq C \epsilon.
$
From this and the Lipschitz continuity of $\eta|_{[0,1]}$, we get that
\begin{equation*}
|\langle\Psi(x,x_j),w(0)\rangle
-\langle\Psi(y,x_j),w(1)\rangle
|\leq C.
\end{equation*}
Thus, we get
\begin{equation}\label{ntx}
\begin{split}
&\left|\frac{1}{n}\sum_{i=1}^n
\langle\Psi(x,x_i),w(0)\rangle\right|\\
\leq&
C\epsilon^{m-1}(\epsilon+\gamma^{1/2} \delta)+C\frac{1}{n}\Card\big\{i:x_i\in B_\epsilon(x)\cap B_\epsilon(y)\big\}\\
+&C\frac{1}{\epsilon n}\Card\big\{i:x_i\in \left(B_\epsilon(x)\cup B_\epsilon(y)\right)\setminus \left(B_\epsilon(x)\cap B_\epsilon(y)\right)\big\}\\
+&C\frac{1}{\epsilon n}\Card\big\{i:x_i\in B(y,\epsilon^{1+\frac{1}{m}}+\epsilon^{2+\frac{1}{m}})\big\}.
\end{split}
\end{equation}
by (\ref{nty}).
We have
\begin{align*}
B_{\epsilon}(x)\cup B_{\epsilon}(y)\subset B\left(y,\epsilon+\epsilon^{2+\frac{1}{m}}\right),\quad
B_{\epsilon}(x)\cap B_{\epsilon}(y)\supset B\left(y,\epsilon-\epsilon^{2+\frac{1}{m}}\right).
\end{align*}
Put
\begin{align*}
D(y):=&B\left(y,\epsilon+\epsilon^{2+\frac{1}{m}}\right),\\
E(y):=&B\left(y,\epsilon^{1+\frac{1}{m}}+ \epsilon^{2+\frac{1}{m}}\right)\cup\left( B\left(y,\epsilon+\epsilon^{2+\frac{1}{m}}\right)\setminus B\left(y,\epsilon-\epsilon^{2+\frac{1}{m}}\right)\right).
\end{align*}
Then, we have
\begin{align*}
(\rho \Vol_g)\left(D(y)\right)\leq C\epsilon^m,\quad
(\rho \Vol_g)\left(E(y)\right)\leq C\epsilon^{m+1},
\end{align*}
and so
\begin{equation}\label{ycard}
\begin{split}
\Card\{i: x_i\in D(y)\}\leq C n\epsilon^m,\quad
\Card\{i: x_i\in E(y)\}\leq C n\epsilon^m(\epsilon+\gamma^{1/2}\delta)
\end{split}
\end{equation}
holds with probability at least $1-2\exp(-\gamma n\epsilon^m\delta^2)$ by the Bernstein inequality.
Thus, we have shown that for fixed $s=1,\ldots, N_1$ and arbitrary $x\in B(y_s ,\epsilon^{2+\frac{1}{m}})$,
\begin{equation*}
\begin{split}
|\theta_\epsilon(x)-\rho(x)\epsilon^m\sigma_\eta|\leq C\epsilon^m(\epsilon+\gamma^{1/2} \delta),\\
\left|\frac{1}{n}\sum_{i=1}^n \Psi(x,x_i)\right|\leq C\epsilon^{m-1}(\epsilon+\gamma^{1/2} \delta)
\end{split}
\end{equation*}
hold with probability at least $1-(2m+4)\exp(-\gamma n\epsilon^m\delta^2)$
by (\ref{tx}), (\ref{ntx}) and (\ref{ycard}).
This and (\ref{cover}) imply that (i) and (ii) hold for every $x\in M$ with probability at least $1-C_3 \epsilon^{-2m-1}(2m+4)\exp(-\gamma n\epsilon^m\delta^2)$, since for a.e. $x\in M$ we have that
$\theta_\epsilon$ is differentiable at $x$ and $\nabla\theta_\epsilon(x)=\sum_{i=1}^n\Psi(x,x_i)/n$.

We next prove (iii).
Take arbitrary $i=1,\ldots,n$.
Since we have
\begin{align*}
\eta\left(\frac{d(x_i,x)}{\epsilon}\right)\leq& \eta(0)\quad(x\in M),\\
\int_M\eta\left(\frac{d(x_i,x)}{\epsilon}\right)^2 \rho(x)\,d x\leq &C\epsilon^m,
\end{align*}
we have
$$
\frac{1}{n-1}\sum_{j\neq i}\eta\left(\frac{d(x_i,x_j)}{\epsilon}\right) \leq \int_M\eta\left(\frac{d(x_i,x)}{\epsilon}\right) \rho(x)\,d x+ C\epsilon^m\gamma^{1/2}\delta 
$$
with probability at least $1-\exp(-n\gamma \epsilon^m\delta^2)$.
Since we can assume $1/n\leq C\gamma\epsilon^m\delta\leq \epsilon^m$, this and
$$
\int_M\eta\left(\frac{d(x_i,x)}{\epsilon}\right) \rho(x)\,d x\leq (1+C \epsilon)\epsilon^m\rho(x_i)\Vol(S^{m-1})\int_0^1\eta(t)t^{m-1}\,d t\leq C\epsilon^m
$$
imply that (iii) holds for every $i=1,\ldots,n$ with probability at least $1-n\exp(-n\gamma \epsilon^m\delta^2)$.

We next prove (iv).
Take arbitrary $i=1,\ldots,n$ and an orthonormal basis $e_1,\ldots,e_m\in T_{x_i} M$.
Define a map 
$
h_i\colon M\times \{1,\ldots,m\}^2 \to \R
$
by
$$
h(y,p,q):=\eta\left(\frac{d(x_i,y)}{\epsilon}\right)\frac{d(x_i,y)^2}{\epsilon^2}
\langle \dot \gamma_{x_i,y}(0),e_p\rangle\langle \dot \gamma_{x_i,y}(0),e_q \rangle.
$$
For each $p,q\in \{1,\ldots,m\}$, we define
$$
s(a,b):=\begin{cases}
1& (p\geq q),\\
-1& (p<q).
\end{cases}
$$
Since we have
\begin{align*}
|h_i(y,p,q)|\leq &\eta(0)\quad (y\in M,\, p,q\in \{1,\ldots,m\}),\\
\int_M|h_i(y,p,q)|^2\rho(y)\,d y\leq &C\epsilon^m\quad(p,q\in\{1,\ldots,m\}),
\end{align*}
we get
\begin{equation}\label{spq}
s(p,q)\left(\frac{1}{n-1}\sum_{j\neq i} h_i(x_j,p,q)-\int_M h_i(y,p,q)\rho(y)\,d y\right) \leq C\gamma^{1/2}\epsilon^m\delta
\end{equation}
for every $p,q\in\{1,\ldots,m\}$ with probability at least $1-m^2\exp(-n\gamma \epsilon^m\delta^2)$ by the Bernstein inequality.
Note that we have $h_i(y,p,q)=h_i(y,q,p)$ and $h_i(y,p,p)\geq 0$ for each $y,p,q$.
Thus, for any $w\in T_{x_i} M$, (\ref{spq}) implies
\begin{align*}
&\frac{1}{n-1}\sum_{j\neq i}\eta\left(\frac{d(x_i,x_j)}{\epsilon}\right)\frac{d(x_i,x_j)^2}{\epsilon^2}\langle \dot \gamma_{x_i,x_j}(0),w\rangle^2\\
=&\frac{1}{n-1}\sum_{j\neq i}\sum_{p,q=1}^m h_i(x_j,p,q)\langle w,e_p\rangle\langle w, e_q\rangle\\
\leq &\sum_{p,q=1}^m \int_M h_i(y,p,q)\langle w,e_p\rangle\langle w, e_q\rangle\rho(y)\,d y+C\gamma^{1/2}\epsilon^{m}\delta\sum_{p,q=1}^m |\langle w,e_p\rangle\langle w, e_q\rangle|\\
\leq &\int_M \eta\left(\frac{d(x_i,y)}{\epsilon}\right)\frac{d(x_i,y)^2}{\epsilon^2}\langle \dot \gamma_{x_i,y}(0),w\rangle^2\rho(y)\,d y+C\gamma^{1/2}\epsilon^{m}\delta|w|^2\\
\leq &\epsilon^m\rho(x_i)\sigma_\eta\left(1+C(\epsilon+\gamma^{1/2}\delta)\right) |w|^2,
\end{align*}
where we used the following inequality:
\begin{align*}
&\int_M \eta\left(\frac{d(x_i,y)}{\epsilon} \right)\frac{d(x_i,x_j)^2}{\epsilon^2}\langle\dot\gamma_{x_i,y}(0),w\rangle^2\rho(y)\,d y\\
\leq& (1+C\epsilon)\frac{\rho(x_i)}{\epsilon^2}\int_0^{\epsilon}\eta\left(\frac{t}{\epsilon}\right)t^{m+1}\,d t \int_{U_{x_i} M}\langle u, w\rangle^2\,d u
= (1+C\epsilon)\epsilon^m\rho(x_i) \sigma_\eta |w|^2.
\end{align*}
Since the content in the sum is $0$ when $j=i$, we get (iv) for every $i$ with probability at least $1-n m^2\exp(-n\gamma \epsilon^m\delta^2)$.

Finally, we prove (v). Take arbitrary $i,j=1,\ldots,n$.
Since we have
\begin{align*}
\left|\left\langle\Psi(x,x_i),\Psi(x,x_j)\right\rangle\frac{1}{\rho(x)}\right|\leq &\alpha\eta(0)^2\frac{1}{\epsilon^2}\quad(x\in M),\\
\int_M\left|\left\langle\Psi(x,x_i),\Psi(x,x_j)\right\rangle\frac{1}{\rho(x)}\right|^2\rho(x)\,d x\leq &C\epsilon^{m-4},
\end{align*}
we get
\begin{align*}
&\frac{1}{n-2}\sum_{l\neq i,j}\left\langle\Psi(x_l,x_i),\Psi(x_l,x_j)\right\rangle\frac{1}{\rho(x_l)}\\
 \leq &\int_M\left\langle\Psi(x,x_i),\Psi(x,x_j)\right\rangle\,d x+C\epsilon^{m-2}\gamma^{1/2}\delta
\end{align*}
with probability at least $1-\exp(-n\gamma \epsilon^m\delta^2)$ by the Bernstein inequality.
Since we can assume $1/n\leq C\gamma\epsilon^m\delta^2\leq C\gamma^{1/2}\delta$, combining this with $\Psi(x_i,x_i)=0$
and $|\int_M\left\langle\Psi(x,x_i),\Psi(x,x_j)\right\rangle\,d x|\leq C\epsilon^{m-2}$, we get (v) with probability at least $1-n^2\exp(-n\gamma \epsilon^m\delta^2)$.
\end{proof}
\begin{Rem}
If $\epsilon+\gamma^{1/2}\delta< \sigma_\eta/(C_2 \alpha)$ holds in Lemma \ref{probt}, we have $\theta_\epsilon(x)>0$ holds for every $x\in M$.
Therefore, $\Lambda_\epsilon$ and $\widetilde\psi$ can be defined according to Definition \ref{defDtC} (v) and (vi). 
\end{Rem}
Let us estimate $\int_M|\nabla \Lambda_\epsilon u|^2 \rho^2\,d\Vol_g$ for $u\in \R^n$.
Since we have $\eta(d(x_i,x_j)/\epsilon)\leq \eta(\|\iota(x_i)-\iota(x_j)\|/\epsilon)$, it is enough to consider the Riemannian distance.
\begin{Lem}\label{UDcomp}
Suppose that we are given $(M,g)\in \M_1(m,K,i_0)$ and functions $\eta,\rho$ satisfying Assumption \ref{Asu3}.
Then, there exists constants $C_1=C_1(m,\alpha)>0$ and $C_2=C_2(m,K,\eta,\alpha,L_\rho)>0$ such that the following properties holds.
For any $\epsilon\in(0,2i_0/3)$, $\gamma\in(1,\infty)$ and $\delta\in(0,\infty)$ with
$$
\epsilon+\gamma^{1/2} \delta\leq C_2^{-1},
$$
and i.i.d. sample $x_1,\ldots, x_n\in M$ from $\rho\Vol_g$,
we have
\begin{align*}
\int_M|\nabla \Lambda_\epsilon u|^2 \rho^2\,d\Vol_g
\leq\frac{1+C_2(\epsilon+\gamma^{1/2}\delta)}{n^2 \epsilon^{m+2}\sigma_\eta}\sum_{i,j=1}^n \eta\left(\frac{d(x_i,x_j)}{\epsilon}\right)(u_i-u_j )^2
\end{align*} 
for every $u\in \R^n$ with probability at least $1-C_1(\epsilon^{-2m-1}+n^2)\exp\left(-n\gamma \epsilon^m\delta^2\right)$.
\end{Lem}
\begin{proof}
We can assume that Lemma \ref{probt} (i)--(v) and
$$
\left|\frac{1}{\theta_\epsilon(x)}-\frac{1}{\rho(x)\sigma_\eta\epsilon^m}\right|\leq C\epsilon^{-m}(\epsilon+\gamma^{1/2}\delta)
$$
hold for any $x\in M$.
Since we have
$$
\nabla_x \widetilde\psi(x,y)=-\frac{\psi\left(d(x,y)/\epsilon\right)}{\theta_\epsilon(x)^2}\nabla\theta_\epsilon(x)+\frac{1}{\theta_\epsilon(x)}\Psi(x,y)
$$
for any $y\in M$ and a.e. $x\in M$,
we get
\begin{equation}\label{pP}
\begin{split}
&\left|\langle\nabla_x \widetilde\psi(x,x_i), \nabla_x \widetilde\psi(x,x_j)\rangle-\frac{1}{\rho(x)^2\sigma_\eta^2 \epsilon^{2m}}\left\langle\Psi(x,x_i),\Psi(x,x_j)\right\rangle\right|\\
\leq& \psi\left(d(x,x_i)/\epsilon\right)\psi\left(d(x,x_j)/\epsilon\right) \frac{|\nabla \theta_\epsilon(x)|^2}{\theta_\epsilon(x)^4}
+\psi\left(d(x,x_i)/\epsilon\right) \frac{|\langle\nabla \theta_\epsilon(x),\Psi(x,x_j)\rangle|}{\theta_\epsilon(x)^3}\\
+&\psi\left(d(x,x_j)/\epsilon\right) \frac{|\langle\nabla \theta_\epsilon(x),\Psi(x,x_i)\rangle|}{\theta_\epsilon(x)^3}
+\left|\frac{1}{\theta_\epsilon(x)^2}-\frac{1}{\rho(x)^2\sigma_\eta^2 \epsilon^{2m}}\right| \left|\left\langle\Psi(x,x_i),\Psi(x,x_j)\right\rangle\right|\\
\leq& C\frac{\epsilon+\gamma^{1/2}\delta}{\epsilon^{2m+2}}\eta\left(\frac{d(x,x_i)}{\epsilon}\right)\eta\left(\frac{d(x,x_j)}{\epsilon}\right)
\end{split}
\end{equation}
for any $i,j=1,\ldots,n$ and a.e. $x\in M$.
Since we have $\sum_{i=1}^n\widetilde\psi(x,x_i)/n=1$,
we get
\begin{equation}\label{Lu}
\begin{split}
|\nabla \Lambda_\epsilon u|^2(x)=&\frac{1}{n^2}\sum_{i,j=1}^n\langle\nabla_x \widetilde\psi(x,x_i),\nabla_x\widetilde\psi(x,x_j)\rangle u_i u_j\\
=&-\frac{1}{2n^2}\sum_{i,j=1}^n\langle\nabla_x \widetilde\psi(x,x_i),\nabla_x\widetilde\psi(x,x_j)\rangle (u_i- u_j)^2
\end{split}
\end{equation}
for a.e. $x\in M$ and any $u\in \R^n$.
We can assume that (\ref{pP}) and (\ref{Lu}) hold for every $x=x_l$ since the probability that $x_l\notin J_{x_i}\setminus \{x_i\}$ holds for some $i,l$ is $0$.
For any $u\in \R^n$, we have
\begin{equation}
\begin{split}
&\int_M |\nabla \Lambda_\epsilon u|^2\rho^2\, d\Vol_g\\
\leq &-\frac{1}{2n^2\sigma_\eta^2 \epsilon^{2m}}\sum_{i,j=1}^n\int_M\langle\Psi(x,x_i),\Psi(x,x_j)\rangle (u_i- u_j)^2\,d x\\
+& C\frac{\epsilon+\gamma^{1/2}\delta}{n^2\epsilon^{2m+2}}\sum_{i,j=1}^n\int_M \eta\left(\frac{d(x,x_i)}{\epsilon}\right)\eta\left(\frac{d(x,x_j)}{\epsilon}\right) (u_i- u_j)^2\,d x\\
\leq &-\frac{1}{2n^3\sigma_\eta^2 \epsilon^{2m}}\sum_{i,j,l=1}^n\langle\Psi(x_l,x_i),\Psi(x_l,x_j)\rangle (u_i- u_j)^2\frac{1}{\rho(x_l)}\\
+&C\frac{\epsilon+\gamma^{1/2}\delta}{n^2\epsilon^{m+2}}\sum_{d(x_i,x_j)\leq 2\epsilon}(u_i- u_j)^2
\end{split}
\end{equation}
by Lemma \ref{probt} (v).
Similarly, we have
\begin{align*}
&\frac{1}{n}\sum_{l=1}^n |\nabla \Lambda_\epsilon u|^2(x_l)\rho(x_l)\\
\geq&- \frac{1}{2n^3\sigma_\eta^2 \epsilon^{2m}}\sum_{i,j,l=1}^n\langle\Psi(x_l,x_i),\Psi(x_l,x_j)\rangle (u_i- u_j)^2\frac{1}{\rho(x_l)}\\
-&C\frac{\epsilon+\gamma^{1/2}\delta}{n^3 \epsilon^{2m+2}}\sum_{i,j,l=1}^n\eta\left(\frac{d(x_l,x_i)}{\epsilon}\right)\eta\left(\frac{d(x_l,x_j)}{\epsilon}\right)(u_i- u_j)^2\\
\geq&- \frac{1}{2n^3\sigma_\eta^2 \epsilon^{2m}}\sum_{i,j,l=1}^n\langle\Psi(x_l,x_i),\Psi(x_l,x_j)\rangle (u_i- u_j)^2\frac{1}{\rho(x_l)}\\
-&C\frac{\epsilon+\gamma^{1/2}\delta}{n^2\epsilon^{m+2}}\sum_{d(x_i,x_j)\leq 2\epsilon}(u_i- u_j)^2
\end{align*}
by Lemma \ref{probt} (iii).
Thus, we get
\begin{equation}\label{nablalambdau}
\begin{split}
&\int_M |\nabla \Lambda_\epsilon u|^2\rho^2\, d\Vol_g\\
\leq &\frac{1}{n}\sum_{l=1}^n |\nabla \Lambda_\epsilon u|^2(x_l)\rho(x_l)+C\frac{\epsilon+\gamma^{1/2}\delta}{n^2\epsilon^{m+2}}\sum_{d(x_i,x_j)\leq 2\epsilon}(u_i- u_j)^2.
\end{split}
\end{equation}

Let us estimate the first term of (\ref{nablalambdau}).
Similarly to (\ref{Lu}), we have
\begin{align*}
&\frac{1}{n}\sum_{l=1}^n |\nabla \Lambda_\epsilon u|^2(x_l)\rho(x_l)\\
=&\frac{1}{n^3}\sum_{i,j,l=1}^n\langle\nabla_x \widetilde\psi(x_l,x_i),\nabla_x\widetilde\psi(x_l,x_j)\rangle (u_i- u_l)(u_j-u_l)\rho(x_l)\\
\leq &\frac{1}{n\sigma_\eta^2 \epsilon^{2m}}\sum_{l=1}^n\left| \frac{1}{n}\sum_{i=1}^n(u_i- u_l)\Psi(x_l,x_i)\right|^2\frac{1}{\rho(x_l)}\\
+&C\frac{\epsilon+\gamma^{1/2}\delta}{n\epsilon^{2m+2}}\sum_{l=1}^n\left|\frac{1}{n}\sum_{i=1}^n\eta\left(\frac{d(x_l,x_i)}{\epsilon}\right)|u_i- u_l|\right|^2\\
\end{align*}
by (\ref{pP}).
For each $l$ and $w\in T_{x_l}M$, we have
\begin{align*}
&\left\langle\frac{1}{n}\sum_{i=1}^n(u_i-u_l)\Psi(x_l,x_i),w\right\rangle^2\\
\leq&\frac{1}{n^2 \epsilon^4}\sum_{i=1}^n\eta\left(\frac{d(x_l,x_i)}{\epsilon}\right)d(x_i,x_l)^2\langle\dot \gamma_{x_l,x_i}(0),w\rangle^2 \sum_{i=1}^n\eta\left(\frac{d(x_l,x_i)}{\epsilon}\right)(u_l-u_i)^2\\
\leq&\left(1+C(\epsilon+\gamma^{1/2}\delta)\right)\frac{\epsilon^m \rho(x_l)\sigma_\eta |w|^2}{n \epsilon^2}\sum_{i=1}^n\eta\left(\frac{d(x_l,x_i)}{\epsilon}\right)(u_l-u_i)^2
\end{align*}
by the Cauchy-Schwarz inequality, the definition of $\Psi$ and Lemma \ref{probt} (iv).
For each $l$, we have
\begin{align*}
\left|\frac{1}{n}\sum_{i=1}^n\eta\left(\frac{d(x_l,x_i)}{\epsilon}\right)|u_i- u_l|\right|^2
\leq& \frac{1}{n^2}\sum_{i=1}^n\eta\left(\frac{d(x_l,x_i)}{\epsilon}\right)\sum_{i=1}^n\eta\left(\frac{d(x_l,x_i)}{\epsilon}\right)(u_i- u_l)^2\\
\leq &C\frac{\epsilon^m}{n}\sum_{i=1}^n\eta\left(\frac{d(x_l,x_i)}{\epsilon}\right)(u_i- u_l)^2
\end{align*}
by Lemma \ref{probt} (iii).
Thus, we get
\begin{equation}\label{FirstTerm}
\frac{1}{n}\sum_{l=1}^n |\nabla \Lambda_\epsilon u|^2(x_l)\rho(x_l)
\leq \frac{1+C(\epsilon+\gamma^{1/2}\delta)}{n^2\sigma_\eta \epsilon^{m+2}}\sum_{i,j=1}^n\eta\left(\frac{d(x_i,x_j)}{\epsilon}\right)(u_i-u_j)^2.
\end{equation}

We next estimate the second term  of (\ref{nablalambdau}).
Applying Lemma \ref{transmap} putting $\tilde\epsilon=\epsilon/10$ and $\tilde\delta=10^{m/2} \delta$,
we can find a Borel map $T\colon M\to \X$ such that
$d(x,T(x))\leq \epsilon/10$ for any $x\in M$ and
$$
\left|\frac{1}{n} -(\rho\Vol_g)(T^{-1}(\{x_i\}))\right|\leq  C(m,\alpha)\gamma^{1/2} \delta(\rho\Vol_g)(T^{-1}(\{x_i\}))
$$
holds for every $i\in\{1,\ldots,n\}$ with probability at least $1-C(m,\alpha)\epsilon^{-m}\exp(-n\gamma\epsilon^m\delta^2)$.
Then, we have
\begin{equation}\label{SecondTerm}
\begin{split}
&\frac{1}{n^2}\sum_{d(x_i,x_j)\leq2\epsilon}(u_i- u_j)^2\\
\leq&(1+C \gamma^{1/2} \delta)\int_M\int_{B(x,2\epsilon+2\tilde\epsilon)}|u(T(x))-u(T(y))|^2\rho(x)\rho(y)\,d y\,d x\\
\leq &4(1+C \gamma^{1/2} \delta)\int_M\int_{B(x,2\epsilon+2\tilde\epsilon)}|u(T(x))-u(T(c_{x,y}(1/2)))|^2\rho(x)\rho(y)\,d x\,d y\\
\leq &C\int_M\int_{B(x,\epsilon+\tilde\epsilon)}|u(T(x))-u(T(y))|^2\rho(x)\rho(y)\,d y\,d x\\
\leq &C\int_M\int_{B(x,(\epsilon+\tilde\epsilon)/2)}|u(T(x))-u(T(y))|^2\rho(x)\rho(y)\,d y\,d x\\
\leq &\frac{C}{n^2}\sum_{d(x_i,x_j)\leq (\epsilon+5\tilde\epsilon)/2}(u_i-u_j)^2
\leq \frac{C}{n^2}\sum_{i,j=1}^n\eta\left(\frac{d(x_i,x_j)}{\epsilon}\right)(u_i-u_j)^2
\end{split}
\end{equation}
by Theorem \ref{BishopGromov} (iv) and $(\epsilon+5\tilde\epsilon)/2=3\epsilon/4$.

By (\ref{nablalambdau}), (\ref{FirstTerm}) and (\ref{SecondTerm}), we get
\begin{align*}
\int_M |\nabla \Lambda_\epsilon u|^2\rho^2\, d\Vol_g
\leq \frac{1+C(\epsilon+\gamma^{1/2}\delta)}{n^2\sigma_\eta \epsilon^{m+2}}\sum_{i,j=1}^n\eta\left(\frac{d(x_i,x_j)}{\epsilon}\right)(u_i-u_j)^2.
\end{align*}
This implies the lemma.
\end{proof}

\subsection{Main Results: the Case of Unnormalized Graph Laplacian}
In this subsection, applying the results of Appendix \ref{LinearAlgebra}, we show our main results for the unnormalized case.
To do this, we prepare some symbols according to Appendix \ref{LinearAlgebra}.
Suppose that we are given an $m$-dimensional closed Riamannian manifold, functions $\eta,\rho$ as Assumption \ref{Asu3}, points $\X=\{x_1,\ldots, x_n\}\in M$ and a positive real number $\epsilon\in(0,\infty)$.
Let $H_1$ denotes an appropriate subspace of $W^{1,2}(M)$ defined in the proofs below, and $H_2:=\R^n$.
Let $\langle\cdot,\cdot\rangle_1$ be the inner product on $L^2(M)$ defined by
$$
\langle f,h\rangle_1:=\int_M f h\rho\,d\Vol_g
$$
for each $f,h\in L^2(M)$.
Let $\{f_i\}_{i=0}^\infty$ denotes the complete orthonormal system of $(L^2(M),\langle\cdot,\cdot\rangle_1)$ consisting of the eigenfunctions of $\Delta_\rho$ corresponding to the eigenvalues $\{\lambda_i(\Delta_\rho)\}_{i=0}^\infty$.
Let $\langle\cdot,\cdot\rangle_2$ be the inner product on $H_2$ defined by
$$
\langle u, v\rangle_2:=\frac{1}{n}\sum_{i=1}^n u_i v_i
$$
for each $u,v\in H_2$.
Let $\{u^i\}_{i=0}^{n-1}$ denotes the orthonormal basis of $(H_2,\langle\cdot,\cdot\rangle_2)$ consisting of the eigenvectors of $\La$  corresponding to the eigenvalues $\{\lambda_i(\La)\}_{i=0}^{n-1}$.
Define
\begin{align*}
D_1(f,h):=&\int_M\langle\nabla f,\nabla h\rangle\rho^2\,d\Vol_g \quad \text{for $f,h\in H_1$},\\
D_2(u,v):=&\frac{1}{\sigma_\eta n^2\epsilon^{m+2}}\sum_{i,j=1}^n\eta\left(\frac{\|\iota(x_i)-\iota(x_j)\|}{\epsilon}\right)(u_i-u_j)(v_i-v_j)\quad \text{for $u,v\in H_2$}
\end{align*}
and $Q_1\colon H_1\to H_2,\, f\mapsto f|_{\X}$.
If $\theta_\epsilon(x)>0$ holds for every $x\in M$ and $\Lambda_\epsilon (H_2)\subset H_1$, we define $Q_2\colon H_2\to H_1,\, u\mapsto \Lambda_\epsilon u$.
Let $\lambda_1(D_i)\leq \lambda_2(D_i)\leq \cdots\leq \lambda_{\dim H_i}(D_i)$ denotes the  eigenvalue of $D_i$ with respect to the inner product $\langle\cdot,\cdot\rangle_1$ for each $i=1,2$.
If $f_0,\ldots, f_{j-1}\in H_1$, we have
$\lambda_j(D_1)=\lambda_{j-1}(\Delta_\rho)$.
We have
\begin{align*}
\lambda_j(D_2)=\frac{2}{\sigma_\eta n\epsilon^{m+2}}\lambda_{j-1}(\La).
\end{align*}

To compare the Rayleigh quotients, we estimate $\left|\|\Lambda_\epsilon u\|^2_1-\|u\|_2^2\right|$ for $u\in \R^n$ using a quantity $b_\epsilon(u)$ defined by
$$
b_\epsilon(u):=\frac{1}{\sigma_\eta n^2\epsilon^{m+2}}\eta\left(\frac{d(x_i,x_j)}{\epsilon}\right)(u_i-u_j)^2.
$$
Note that we have $b_\epsilon(u)\leq D_2(u,u)$.
\begin{Lem}\label{UNl2comp}
Suppose that we are given $(M,g)\in \M_1(m,K,i_0)$ and functions $\eta,\rho$ satisfying Assumption \ref{Asu3}.
Then, there exists constants $C_1=C_1(m,\alpha)>0$ and $C_2=C_2(m,K,\eta,\alpha,L_\rho)>0$ such that the following properties holds.
For any $\epsilon\in(0,2i_0/3)$, $\gamma\in(1,\infty)$ and $\delta\in(0,\infty)$ with
$$
\epsilon+\gamma^{1/2} \delta\leq C_2^{-1},
$$
and i.i.d. sample $x_1,\ldots, x_n\in M$ from $\rho\Vol_g$,
we have
\begin{align*}
\left|\|\Lambda_\epsilon u\|_1^2-\|u\|_2^2\right|\leq C_2\left((\epsilon+\gamma^{1/2}\delta)\|u\|_2+\epsilon b_\epsilon(u)^{1/2}\right)(\|u\|_2+\epsilon b_\epsilon(u)^{1/2} )
\end{align*}
for every $u\in \R^n$ with probability at least $1-C_1(\epsilon^{-2m-1}+n^2)\exp\left(-n\gamma \epsilon^m\delta^2\right)$.
\end{Lem}
\begin{proof}
We can make the same assumptions as at the beginning of the proof of Lemma \ref{UDcomp}.
Then, for each $i$, since we have $\sum_{j=1}^n\widetilde\psi(x_i,x_j)/n=1$,
we get
\begin{align*}
\left(\Lambda_\epsilon u(x_i)-u_i\right)^2
=&\left(\frac{1}{n}\sum_{j=1}^n\widetilde \psi(x_i,x_j)(u_i-u_j)\right)^2\\
\leq &\frac{1}{n}\sum_{j=1}^n\widetilde \psi(x_i,x_j)(u_i-u_j)^2\leq \frac{C}{n\epsilon^m}\sum_{j=1}^n\eta\left(\frac{d(x_i,x_j)}{\epsilon}\right)(u_i-u_j)^2.
\end{align*}
Thus, putting $\X=\{x_1,\ldots,x_n\}$, we have
\begin{equation}\label{Luvsu}
\left\|\Lambda_\epsilon u|_{\X}-u\right\|_2\leq C\epsilon b_\epsilon(u)^{1/2}.
\end{equation}
For each $x\in M$, put
$$
A(x):=\frac{1}{n\rho(x)\epsilon^m\sigma_\eta}\sum_{i=1}^n\psi\left(\frac{d(x,x_i)}{\epsilon}\right)u_i.
$$
Then, we have
\begin{align*}
&\left|\Lambda_\epsilon u(x)-A(x)\right|\\
\leq&\frac{1}{n}\left|\frac{1}{\theta_\epsilon(x)}-\frac{1}{\rho(x)\epsilon^m\sigma_\eta}\right|\sum_{i=1}^n\psi\left(\frac{d(x,x_i)}{\epsilon}\right)|u_i|\\
\leq& C\frac{\epsilon+\gamma^{1/2}\delta}{\epsilon^m}\left(\frac{1}{n}\sum_{i=1}^n\psi\left(\frac{d(x,x_i)}{\epsilon}\right)\right)^{1/2}\left(\frac{1}{n}\sum_{i=1}^n\psi\left(\frac{d(x,x_i)}{\epsilon}\right)u_i^2\right)^{1/2}\\
\leq &C(\epsilon+\gamma^{1/2}\delta)\left(\frac{1}{n\epsilon^m}\sum_{i=1}^n\psi\left(\frac{d(x,x_i)}{\epsilon}\right)u_i^2\right)^{1/2}.
\end{align*}
Thus, we get
\begin{align}
\label{LuvsA}\left\|\Lambda_\epsilon u-A\right\|_1\leq& C(\epsilon+\gamma^{1/2}\delta)\|u\|_2,\\
\label{LuvsAX}\left\|\Lambda_\epsilon u|_{\X}-A|_{\X}\right\|_2\leq& C(\epsilon+\gamma^{1/2}\delta)\|u\|_2
\end{align}
by Lemma \ref{theta} (i) and Lemma \ref{probt} (i), respectively.
Since we have
\begin{align*}
\left|\psi\left(\frac{d(x,x_i)}{\epsilon}\right)\psi\left(\frac{d(x,x_j)}{\epsilon}\right)\frac{1}{\rho(x)^2}\right|\leq \frac{\alpha^2}{4}\eta(0)^2,\\
\int_M\left(\psi\left(\frac{d(x,x_i)}{\epsilon}\right)\psi\left(\frac{d(x,x_j)}{\epsilon}\right)\frac{1}{\rho(x)^2}\right)^2\rho(x)\,d x\leq C\epsilon^m,
\end{align*}
we get
\begin{equation}\label{psicomp}
\begin{split}
\Big|\frac{1}{n-2}\sum_{l\neq i, j}&\psi\left(\frac{d(x_l,x_i)}{\epsilon}\right)\psi\left(\frac{d(x_l,x_j)}{\epsilon}\right)\frac{1}{\rho(x_l)^2}\\
&-\int_M \psi\left(\frac{d(x,x_i)}{\epsilon}\right)\psi\left(\frac{d(x,x_j)}{\epsilon}\right)\frac{1}{\rho(x)}\,d x\Big|\leq C\gamma^{1/2}\epsilon^m\delta
\end{split}
\end{equation}
for every $i,j$ with probability at least  $1-2 n^2\exp\left(-n\gamma \epsilon^m\delta^2\right)$ by the Bernstein inequality.
Since  we have $(\rho\Vol_g)(B_{2\epsilon}(x_i))\leq C\epsilon^m$, we get
\begin{equation}\label{card2e}
\max_i\Card\{j: x_j\in B_{2\epsilon}(x_i)\}\leq  n C(1+\gamma^{1/2}\delta)\epsilon^m+1\leq n C\epsilon^m+1
\end{equation}
with probability at least  $1-n\exp\left(-n\gamma \epsilon^m\delta^2\right)$ by the Bernstein inequality.
Since we can assume $1/(n\epsilon^m)\leq C\gamma\delta^2\leq C\gamma^{1/2}\delta$,
(\ref{psicomp}) and (\ref{card2e}) imply
\begin{equation}\label{AXvsA}
\begin{split}
&\left|\frac{n}{n-2}\|A|_{\X}\|_2^2- \|A\|_1^2\right|\\
\leq &C\frac{\gamma^{1/2}\epsilon^m\delta}{n^2\epsilon^m}\sum_{d(x_i,x_j)\leq 2\epsilon}|u_i u_j|\\
\leq & C\gamma^{1/2}\delta\frac{\max_i\Card\{j: x_j\in B_{2\epsilon}(x_i)\}}{n\epsilon^m}\|u\|^2_2\leq  C\gamma^{1/2}\delta\|u\|^2_2.
\end{split}
\end{equation}
Here, we used $|u_i u_j|\leq |u_i|^2/2+|u_j|^2/2$.

Combining (\ref{Luvsu}), (\ref{LuvsA}), (\ref{LuvsAX}) and (\ref{AXvsA}), we get
\begin{equation*}\label{LuvsuX}
\left|\|\Lambda_\epsilon u\|_1^2-\frac{n}{n-2}\|u\|_2^2\right|\leq C\left((\epsilon+\gamma^{1/2}\delta)\|u\|_2+\epsilon b_\epsilon(u)^{1/2}\right)(\|u\|_2+\epsilon b_\epsilon(u)^{1/2}).
\end{equation*}
Since we can assume that $\left|1-n/(n-2)\right|=2/(n-2)\leq C\gamma\epsilon^m\delta^2\leq C\epsilon$,
we get the lemma.
\end{proof}

\begin{Lem}\label{mainu}
Suppose that the pair $((M,g),\iota)$ satisfies Assumption \ref{Asu2} $(a)$, and we are given functions $\eta,\rho$ satisfying Assumption \ref{Asu3}.
Take arbitrary $k\in \Z_{>0}$.
Then, there exist  constants $C_1=C_1(m,K,\alpha,k)>0$ and $C_2=C_2(m,S,K,i_0,L,\eta,\alpha,L_\rho,k)>0$ such that, for any $\epsilon\in [\tau,\pi/(2\sqrt K)]$, $\gamma\in(1,\infty)$, $\delta\in(0,\gamma^{-1/2}]$ with
$\gamma^{1/2}\epsilon^{m/2}\delta\leq C_1^{-1}$ and i.i.d. sample $x_1,\ldots, x_n\in M$ from $\rho\Vol_g$,
we have
\begin{align*}
\frac{2}{\sigma_\eta n \epsilon^{m+2}}\lambda_k(\La)\leq\lambda_k(\Delta_\rho)+C_2\left(\epsilon+\frac{\tau}{\epsilon}+\gamma^{1/2}\delta\right)
\end{align*}
with probability at least
$
1-(k+1)(k+2)(n+2)\exp\left(-n\gamma \epsilon^m\delta^2\right).
$
\end{Lem}
\begin{proof}
To apply Lemma \ref{evalcomp0}, we define $H_1:=\Span\{f_0,f_1,\ldots,f_k\}$.
We have that $\diam_g(M)\leq C(m,K,i_0,\alpha)$. See Remark \ref{RemA} (e).
Since we have $\lambda_k(\Delta_\rho)\leq\alpha^3\lambda_k(\Delta)\leq C(m,K,\alpha,k)$ (see Remark \ref{RemA} (f)), we get
\begin{align*}
\sup_{s\in\{0,\ldots, k\}}\|f_s\|_{L^\infty}\leq& C(m,K,\alpha,k),\\
\sup_{s\in\{0,\ldots, k\}}\Lip(f_s)\leq& C(m,K,i_0,\alpha,k)
\end{align*}
by Lemma \ref{LinfEF} and \ref{GradEF}.
Thus, we get the lemma by Lemma \ref{LinfBer} putting $\widetilde \delta=\epsilon^{m/2}\delta$, Lemma \ref{UPP} and Lemma \ref{evalcomp0}.
\end{proof}

\begin{Lem}\label{mainl}
Suppose that the pair $((M,g),\iota)$ satisfies Assumption \ref{Asu2} $(i)$, and we are given functions $\eta,\rho$ satisfying Assumption \ref{Asu3}.
Take arbitrary $k\in \Z_{>0}$.
Then, there exists constants $C_1=C_1(m,\alpha,k)>0$ and $C_2=C_2(m,K,\eta,\alpha,L_\rho,k)>0$ such that, for any
$\epsilon\in(0,2i_0/3)$, $\gamma\in(1,\infty)$ and $\delta\in(0,\infty)$ with
$$
\epsilon+\frac{\tau}{\epsilon}+\gamma^{1/2} \delta\leq C_2^{-1}
$$
and i.i.d. sample $x_1,\ldots,x_n\in M$ from $\rho\Vol_g$, we have 
$$
\lambda_k(\Delta_\rho)\leq \frac{2}{\sigma_\eta n\epsilon^{m+2}}\lambda_k(\La)+C_2(\epsilon+\gamma^{1/2}\delta)
$$
with probability at least $1-C_1(\epsilon^{-2m-1}+n^2)\exp\left(-n\gamma \epsilon^m\delta^2\right)$.
\end{Lem}
\begin{proof}
If
$$
\lambda_k(\Delta_\rho)\leq \frac{2}{\sigma_\eta n\epsilon^{m+2}}\lambda_k(\La)
$$
holds, then we get the lemma. Thus, we can assume
$$
\frac{2}{\sigma_\eta n\epsilon^{m+2}}\lambda_k(\La)\leq \lambda_k(\Delta_\rho)\leq C(m,K,\alpha,k).
$$
This implies that we have $b_\epsilon(u)\leq C\|u\|_2^2$ for any $u\in \Span\{u^0,\ldots,u^k\}$.
Defining $H_1:=\Span\{f_0,\ldots,f_k\}+ \Lambda_\epsilon\left(\R^n\right)$, we get the lemma by  Lemma \ref{UDcomp}, \ref{UNl2comp} and \ref{evalcomp} (i).
\end{proof}
Putting $\delta:=\epsilon^{-m/2}(\log n/n)^{1/2}$, we get the following by Lemma \ref{mainu} and \ref{mainl}.
\begin{Thm}\label{eigmain}
Suppose that the pair $((M,g),\iota)$ satisfies Assumption \ref{Asu2} $(a)$, and we are given functions $\eta,\rho$ satisfying Assumption \ref{Asu3}.
Take arbitrary $k\in \Z_{>0}$.
Then, there exist constants $C_1=C_1(m,\alpha,k)>0$ and $C_2=C_2(m,S,K,i_0,L,\eta,\alpha,L_\rho,k)>0$ such that, for any
$n\in\Z_{>0}$, $\epsilon\in[\tau,\infty)$, $\gamma\in(1,\infty)$ with
$$
\epsilon+\frac{\tau}{\epsilon}+\gamma^{1/2} \epsilon^{-m/2}\left(\frac{\log n}{n}\right)^{1/2}\leq C_2^{-1}
$$
and i.i.d. sample $x_1,\ldots,x_n\in M$ from $\rho\Vol_g$, we have 
$$
\left|\lambda_k(\Delta_{\rho})-\frac{2}{\sigma_\eta n \epsilon^{m+2}}\lambda_k(\La)\right|\leq C_2\left(\epsilon+\frac{\tau}{\epsilon}+\gamma^{1/2} \epsilon^{-m/2}\left(\frac{\log n}{n}\right)^{1/2}\right)
$$
with probability at least $1-C_1(\epsilon^{-2m-1}+n^2)n^{-\gamma}.$
\end{Thm}

Now let us show the approximation result for the eigenfunctions.

\begin{Thm}\label{evecmain}
Suppose that the pair $((M,g),\iota)$ satisfies Assumption \ref{Asu2} (a), and we are given functions $\eta,\rho$ satisfying Assumption \ref{Asu3}.
Take arbitrary $k,l\in \Z_{>0}$ with $l\geq k$.
Put $s:=\lambda_l(\Delta_\rho)-\lambda_k(\Delta_\rho)$ and
$$
G:=\min\left\{|\lambda_k(\Delta_\rho)-\lambda_{k-1}(\Delta_\rho)|,|\lambda_{l+1}(\Delta_\rho)-\lambda_l(\Delta_\rho)|\right\}.
$$
Then, there exist constants $C_1=C_1(m,\alpha,l)>0$ and $C_2=C_2(m,S,K,i_0,L,\eta,\alpha,L_\rho,l)>0$ such that, for any
$n\in\Z_{>0}$, $\epsilon\in[\tau,\infty)$, $\gamma\in(1,\infty)$ with
$$
C_2\left(\epsilon+\frac{\tau}{\epsilon}+\gamma^{1/2} \epsilon^{-m/2}\left(\frac{\log n}{n}\right)^{1/2}\right)+2 G s < G^2
$$
and i.i.d. sample $x_1,\ldots,x_n\in M$ from $\rho\Vol_g$, we have the following with probability at least $1-C_1(\epsilon^{-2m-1}+n^2)n^{-\gamma}.$
Let $\Pj\colon \R^n\to \Span\{u^k,\ldots,u^l\}$ be the orthogonal projection.
Then, the map $\Span\{f_k,\ldots,f_l\}\to \Span\{u^k,\ldots,u^l\},\, f\mapsto \Pj(f|_{\X})$ is an isomorphism and for any $f\in\Span\{f_k,\ldots,f_l\}$, we have
\begin{align*}
\left(1-\frac{C_2}{G^2}\left(\epsilon+\frac{\tau}{\epsilon}+\gamma^{1/2}\epsilon^{-m/2}\left(\frac{\log n}{n}\right)^{1/2}\right)-\frac{2s}{G}\right)&\frac{1}{n}\sum_{i=1}^n f(x_i)^2\leq\frac{1}{n}\sum_{i=1}^n\Pj(f|_{\X})_i^2,\\
\left|\frac{1}{n}\sum_{i=1}^n f(x_i)^2-\int_M f^2\rho\,d\Vol_g\right|\leq& C_2\gamma^{1/2}\left(\frac{\log n}{n}\right)^{1/2}\int_M f^2\rho\,d\Vol_g
\end{align*}
and
$$
\frac{1}{n}\sum_{i=1}^n (f(x_i)-\Pj(f|_{\X})_i)^2\leq \left(\frac{C_2}{G^2}\left(\epsilon+\frac{\tau}{\epsilon}+\gamma^{1/2} \epsilon^{-m/2}\left(\frac{\log n}{n}\right)^{1/2}\right)+\frac{2s}{G}\right)\frac{1}{n} \sum_{i=1}^n f(x_i)^2.
$$
\end{Thm}
\begin{proof}
Put $\delta:=\epsilon^{-m/2}(\log n/n)^{1/2}$ and $H_1:=\Span\{f_0,\ldots,f_{l+1}\}+ \Lambda_\epsilon\left(\R^n\right)$.
By Lemma \ref{LinfBer} and \ref{q2q1f-f}, for any $f\in \Span\{f_k,\ldots,f_l\}\setminus\{0\}$, we have
\begin{align*}
\left|\|Q_2 Q_1 f\|_1^2-\|Q_1 f\|_2^2\right|\leq C(\epsilon+\gamma^{1/2}\delta)\|Q_1 f\|_2,
\end{align*}
and so
\begin{align*}
\frac{D_1(Q_2 Q_1 f,Q_2 Q_1 f)}{\|Q_2 Q_1 f\|_1^2}\leq& C\left(\epsilon+\frac{\tau}{r}+\gamma^{1/2} \delta\right)\frac{D_2(Q_1 f,Q_1 f)}{\|Q_1 f\|_2^2}\\
\leq& \frac{D_2(Q_1 f,Q_1 f)}{\|Q_1 f\|_2^2}+C\left(\epsilon+\frac{\tau}{r}+\gamma^{1/2} \delta\right).
\end{align*}
by Lemma \ref{LinfBer}, \ref{UPP}, \ref{UDcomp} and \ref{UNl2comp}.
By Lemma \ref{LinfBer}, for any $f\in \Span\{f_k,\ldots,f_l\}\setminus\{0\}$, we have
$$
\frac{\left|\|Q_1 f\|_2-\|f\|_1\right|}{\|f\|_1}\leq \frac{\left|\|Q_1 f\|_2^2-\|f\|_1^2\right|}{\|f\|_1^2}\leq C\gamma^{1/2}\left(\frac{\log n}{n}\right)^{1/2}.
$$
Thus, we get the theorem by Lemma \ref{LinfBer}, \ref{UPP}, \ref{q2q1f-f}, \ref{UDcomp}, \ref{UNl2comp} and \ref{eveccomp}.
\end{proof}

Approximating the pair $(M,\iota)$ satisfying Assumption \ref{Asu}, we obtain the following theorems.
See Appendix \ref{ApproxSmooth} for why such an approximation works.
\begin{Thm}\label{UNMainEval}
Suppose that the pair $((M,g),\iota)$ satisfies Assumption \ref{Asu}, and we are given functions $\eta, \rho$ satisfying Assumption \ref{Asu3}.
Take arbitrary $k\in \Z_{>0}$.
Then, there exist constants $C_1=C_1(m,\alpha,k)>0$ and $C_2=C_2(m,S,K,i_0,L,\eta,\alpha,L_\rho,k)>0$ such that, for any
$n\in\Z_{>0}$, $\epsilon\in(0,\infty)$, $\gamma\in(1,\infty)$ with
\begin{equation*}\label{AR}
\epsilon+\gamma^{1/2} \epsilon^{-m/2}\left(\frac{\log n}{n}\right)^{1/2}\leq C_2^{-1}
\end{equation*}
and i.i.d. sample $x_1,\ldots,x_n\in X$ from $\rho\Vol_g$, we have 
$$
\left|\lambda_k(\Delta_{\rho})-\frac{2}{\sigma_\eta n \epsilon^{m+2}}\lambda_k(\La)\right|\leq C_2\left(\epsilon+\gamma^{1/2} \epsilon^{-m/2}\left(\frac{\log n}{n}\right)^{1/2}\right)
$$
with probability at least $1-C_1(\epsilon^{-2m-1}+n^2)n^{-\gamma}.$
\end{Thm}
Putting $\epsilon=(\log n/n)^{1/(m+2)}$, we get Theorem \ref{MainResult1}.

\begin{Thm}\label{UNMainEvec}
Suppose that the pair $((M,g),\iota)$ satisfies Assumption \ref{Asu}, and we are given functions $\eta, \rho$ satisfying Assumption \ref{Asu3}. Take arbitrary $k,l\in \Z_{>0}$ with $l\geq k$.
Put $s:=\lambda_l(\Delta_\rho)-\lambda_k(\Delta_\rho)$ and
$$
G:=\min\left\{|\lambda_k(\Delta_\rho)-\lambda_{k-1}(\Delta_\rho)|,|\lambda_{l+1}(\Delta_\rho)-\lambda_l(\Delta_\rho)|\right\}.
$$
Then, there exist constants $C_1=C_1(m,\alpha,l)>0$ and $C_2=C_2(m,S,K,i_0,L,\eta,\alpha,L_\rho,l)>0$ such that, for any
$n\in\Z_{>0}$, $\epsilon\in(0,\infty)$, $\gamma\in(1,\infty)$ with
$$
C_2\left(\epsilon+\gamma^{1/2} \epsilon^{-m/2}\left(\frac{\log n}{n}\right)^{1/2}\right)+2 G s < G^2
$$
and i.i.d. sample $x_1,\ldots,x_n\in M$ from $\rho\Vol_g$, we have the following with probability at least $1-C_1(\epsilon^{-2m-1}+n^2)n^{-\gamma}.$
Let $\Pj\colon \R^n\to \Span\{u^k,\ldots,u^l\}$ be the orthogonal projection.
Then, the map $\Span\{f_k,\ldots,f_l\}\to \Span\{u^k,\ldots,u^l\},\, f\mapsto \Pj(f|_{\X})$ is an isomorphism and for any $f\in\Span\{f_k,\ldots,f_l\}$, we have
\begin{align*}
\left(1-\frac{C_2}{G^2}\left(\epsilon+\gamma^{1/2}\epsilon^{-m/2}\left(\frac{\log n}{n}\right)^{1/2}\right)-\frac{2s}{G}\right)&\frac{1}{n}\sum_{i=1}^n f(x_i)^2\leq\frac{1}{n}\sum_{i=1}^n\Pj(f|_{\X})_i^2,\\
\left|\frac{1}{n}\sum_{i=1}^n f(x_i)^2-\int_M f^2\rho\,d\Vol_g\right|\leq& C_2\gamma^{1/2}\left(\frac{\log n}{n}\right)^{1/2}\int_M f^2\rho\,d\Vol_g
\end{align*}
and
$$
\frac{1}{n}\sum_{i=1}^n (f(x_i)-\Pj(f|_{\X})_i)^2\leq \left(\frac{C_2}{G^2}\left(\epsilon+\gamma^{1/2} \epsilon^{-m/2}\left(\frac{\log n}{n}\right)^{1/2}\right)+\frac{2s}{G}\right)\frac{1}{n} \sum_{i=1}^n f(x_i)^2.
$$
\end{Thm}
We can get the results corresponding to Lemma \ref{eveccomp} (iii) and (iv).
However, these are easy consequences of (i) and (ii), so we do not state here.
For the case  when $k=l$, we have $s=0$ and get Theorem \ref{MainResult2} putting $\epsilon=(\log n/n)^{1/(m+2)}$.

\subsection{Main Results: the Case of Normalized Graph Laplacian}
In this subsection, we approximate the eigenvalue of the Laplacian for the normalized case.

\begin{notation}
Let $m,d\in\Z_{>0}$ be integers with $m<d$.
Suppose that we are given an $m$-dimensional closed Riemannian manifold $(M,g)$, an isometric immersion $\iota\colon M\to \R^d$, functions $\eta,\rho$ as Assumption \ref{Asu3} and a constant $\epsilon>0$.
We define
\begin{align*}
\widetilde\sigma_\eta:=&\Vol(S^{m-1})\int_0^1\eta(t)t^{m-1}\,d t,\\
\rho_{\eta}(x):=&\frac{1}{\epsilon^m\widetilde\sigma_\eta}\int_M\eta\left(\frac{d(x,y)}{\epsilon}\right)\rho(y)\,d y,\\
\widetilde\rho_\eta(x):=&\frac{1}{\epsilon^m\widetilde\sigma_\eta}\int_M\eta\left(\frac{\|\iota(x)-\iota(y)\|}{\epsilon}\right)\rho(y)\,d y.
\end{align*}
Suppose that we are given points $x_1,\ldots,x_n\in M$. Then, we define $\widetilde \D_i\in \R_{\geq 0}$ ($i=1,\ldots, n$) by
$$
\widetilde \D_i:=\frac{1}{n \epsilon^m\widetilde\sigma_\eta}\sum_{j=1}^n\eta\left(\frac{\|\iota(x_i)-\iota(x_j)\|}{\epsilon}\right)=\frac{1}{n \epsilon^m\widetilde\sigma_\eta}\D_{i i},
$$
where $\D\in \R^{n\times n}$ is a matrix defined in Notation \ref{Not} (vii).
Put $\X=\{x_1,\ldots,x_n\}$.
For each $u,v\in \R^n$, we define
\begin{align*}
\langle u,v\rangle_{\widetilde D}:=&\frac{1}{n}\sum_{i=1}^n u_i v_i\widetilde D_i,\\
\langle u, v\rangle_{\widetilde\rho_\eta|_{\X}}:=&\frac{1}{n}\sum_{i=1}^n u_i v_i\widetilde\rho_\eta(x_i).
\end{align*}
Let $\{u^{N,i}\}_{i=0}^{n-1}$ denotes the orthonormal basis of $(\R^n,\langle\cdot,\cdot\rangle_{\widetilde D})$ consisting of the eigenvectors of  the eigenvalue problem $\La u=\lambda \D u$  corresponding to the eigenvalues $\{\lambda_i(\La,\D)\}_{i=0}^{n-1}$.
For each $f,h\in L^2(M)$, we define
\begin{align*}
\langle f, h\rangle_{\rho^2}:=&\int_M f h\rho^2\,d \Vol_g,\\
\langle f, h\rangle_{\rho\widetilde\rho_\eta}:=&\int_M f h\rho\widetilde\rho_\eta\,d \Vol_g.
\end{align*}
Let $\{f_i^N\}_{i=0}^\infty$ denotes the complete orthonormal system of $(L^2(M),\langle\cdot,\cdot\rangle_\rho^2)$ consisting of the eigenfunctions of $\Delta_\rho^N$ corresponding to the eigenvalues $\{\lambda_i(\Delta_\rho^N)\}_{i=0}^\infty$.

For each $u\in \R^n$ and $\epsilon>0$, we define
$$
B_\epsilon (u):=\frac{1}{\sigma_\eta \epsilon^{m+2}n^2}\sum_{i,j=1}^n \eta\left(\frac{\|\iota(x_i)-\iota(x_j)\|}{\epsilon}\right)(u_i-u_j)^2.
$$
For each $k\in \{0,\ldots,n-1\}$ and $\epsilon$,
we define
$$
\lambda_k(B_\epsilon^N):=\inf\left\{\sup_{u\in V\setminus\{0\}}\frac{B_\epsilon(u)}{\frac{1}{n}\sum_{i=1}^n u_i^2 \widetilde D_i}: V\subset \R^n \text{ is a $(k+1)$-dimensional subspace}\right\}.
$$
By a straightforward calculus, we have
$$
\lambda_k(B_\epsilon^N)=\frac{2\widetilde\sigma_\eta}{\sigma_\eta \epsilon^2}\lambda_k(\La,\D)
$$
for any $k\in \{0,\ldots,n-1\}$.
\end{notation}
By the volume comparison theorem, we get the following.
\begin{Lem}\label{retavsr}
Suppose that we are given $(M,g)\in \M_1(m,K,i_0)$ and functions $\eta,\rho$ satisfying Assumption \ref{Asu3}.
Then, there exists a constant $C=C(m,K,\alpha,L_\rho)>0$ such that for any $\epsilon\in(0,\min\{i_0,\pi/(2\sqrt K)\})$ and $x\in M$, we have
$|\rho_\eta(x)-\rho(x)|\leq C\epsilon\rho(x)$.
\end{Lem}

Similarly to the proof of Lemma \ref{extvsint}, in particular (\ref{er2}) and (\ref{er3}), we get the following.
\begin{Lem}\label{retavstreta}
Suppose that the pair $((M,g),\iota)$ satisfies Assumption \ref{Asu2} $(b)$, and we are given functions $\eta,\rho$ satisfying Assumption \ref{Asu3}.
Then, there exists a constant $C=C(m,S,K,L,\eta,\alpha)>0$  such that for any $\epsilon\in[\tau,\pi/(2\sqrt K)]$, we have
$$
\int_M |\rho_\eta-\widetilde \rho_\eta|\,d\Vol_g\leq C\left(\epsilon+\frac{\tau}{\epsilon}\right).
$$
\end{Lem}

By applying the Bernstein inequality to $\eta(-\|\iota(x_i)-\iota(\cdot)\|/\epsilon)$ for each $i\in\{1,\ldots,n\}$, we get the following.
\begin{Lem}\label{tDvstr}
Suppose that the pair $((M,g),\iota)$ satisfies Assumption \ref{Asu2} $(b)$, and we are given functions $\eta,\rho$ satisfying Assumption \ref{Asu3}.
Then, there exists constant $C=C(m,K,L,\eta(0),\alpha)>0$  such that for any $\epsilon\in[\tau,\pi/(2\sqrt K)]$, $\gamma\in(1,\infty)$ and $\delta\in(0,\gamma^{-1/2}]$ and i.i.d. sample $x_1,\ldots, x_n\in M$ from $\rho\Vol_g$,
we have
$$
\left|\widetilde D_i -\widetilde\rho_\eta(x_i)\right|\leq C\gamma^{1/2}\delta
$$
for every $i\in \{1,\ldots,n\}$ with probability at least $1-2n\exp\left(-n\gamma \epsilon^m\delta^2\right)$.
\end{Lem}
By Corollary \ref{LinfBer}, Lemma \ref{retavsr}, \ref{retavstreta} and \ref{tDvstr}, we get the following.
Note that we apply Corollary \ref{LinfBer} to $f\widetilde\rho_\eta^{1/2}$.
\begin{Lem}\label{Nl2tDf}
Suppose that the pair $((M,g),\iota)$ satisfies Assumption \ref{Asu2} $(a)$, and we are given functions $\eta,\rho$ satisfying Assumption \ref{Asu3}.
Then, there exists constant $C=C(m,K,L,\eta,\alpha,L_\rho)>0$  such that for
any $f_1,\ldots, f_k\in L^\infty(M)$ $($$k\in\Z_{>0}$$)$, $\epsilon\in[\tau,\min\{i_0,\pi/(2\sqrt K)\})$, $\gamma\in(1,\infty)$ and $\delta\in(0,\gamma^{-1/2}]$ and i.i.d. sample $x_1,\ldots,x_n\in M$ from $\rho\Vol_g$, we have
$$
\left|\frac{1}{n}\sum_{i=1}^n f^2(x_i)\widetilde\D_i-\int_M f^2\rho^2 \,d\Vol_g\right|\leq Ck\max_{s}\{\|f_s\|^2_{L^\infty}\}\left(\epsilon+\frac{\tau}{\epsilon}+\gamma^{1/2}\delta\right)
$$
for every $f=\sum_{s=1}^k a_s f_s$ $($$a_s\in \R$ with $\sum_{s=1}^k a_s^2=1$$)$
with probability at least $1-\left(k(k+1)+n\right)\exp\left(-n \gamma\epsilon^m\delta^2\right)$.
\end{Lem}

By Lemma \ref{UPP}, Lemma \ref{Nl2tDf}, Proposition \ref{LinfEF} and \ref{GradEF}, we can apply Lemma \ref{evalcomp0}, and get the following similarly to Theorem \ref{mainu}.
\begin{Lem}\label{NUPP}
Suppose that the pair $((M,g),\iota)$ satisfies Assumption \ref{Asu2} $(a)$, and we are given functions $\eta,\rho$ satisfying Assumption \ref{Asu3}.
Take arbitrary $k\in \Z_{>0}$.
Then, there exist constants $C_1=C_1(m,K,L,\eta,\alpha,L_\rho,k)>0$ and $C_2=C_2(m,S,K,i_0,L,\eta,\alpha,L_\rho,k)>0$ such that, for any $\epsilon\in[\tau,i_0)$, $\gamma\in(1,\infty)$, $\delta\in(0,\infty)$ with
$$\epsilon+\frac{\tau}{\epsilon}+\gamma^{1/2}\delta\leq C_1^{-1}$$
and i.i.d. sample $x_1,\ldots, x_n\in M$ from $\rho\Vol_g$,
we have
\begin{align*}
\lambda_k(B_\epsilon^N)\leq \lambda_k(\Delta_\rho^N)+C_2\left(\epsilon+\frac{\tau}{\epsilon}+\gamma^{1/2}\delta\right)
\end{align*}
with probability at least
$
1-\left((k+1)(k+2)(n+2)+n\right)\exp\left(-n\gamma \epsilon^m\delta^2\right).
$
\end{Lem}
\begin{Lem}\label{preLuvsu}
Suppose that the pair $((M,g),\iota)$ satisfies Assumption \ref{Asu2}, and we are given functions $\eta,\rho$ satisfying Assumption \ref{Asu3}.
Then, there exist constants $C_1=C_1(m,\alpha)>0$ and $C_2=C_2(m,K,L,\eta,\alpha,L_\rho)>0$ such that the following properties holds.
For any $\epsilon\in[\tau,2i_0/3)$, $\gamma\in(1,\infty)$ and $\delta\in(0,\infty)$ with
$$
\epsilon+\gamma^{1/2} \delta\leq C_2^{-1},
$$
and i.i.d. sample $x_1,\ldots, x_n\in M$ from $\rho\Vol_g$,
we have
\begin{align*}
&\left|\|\Lambda_\epsilon u\|_{\rho\widetilde\rho_\eta}^2-\|u\|_{\widetilde\rho_\eta|_{\X}}^2\right|\\
\leq& C_2\left((\epsilon+\gamma^{1/2}\delta)\|u\|_{\widetilde\rho_\eta|_{\X}}+\epsilon b_\epsilon(u)^{1/2}\right)(\|u\|_{\widetilde\rho_\eta|_{\X}}+\epsilon b_\epsilon(u)^{1/2} )
\end{align*}
for every $u\in \R^n$ with probability at least $1-C_1(\epsilon^{-2m-1}+n^2)\exp\left(-n\gamma \epsilon^m\delta^2\right)$.
\end{Lem}
\begin{proof}
Put $\X=\{x_1,\ldots,x_n\}$.
Since we have
$\|\iota(x)-\iota(y)\|\leq d(x,y)\leq L\|\iota(x)-\iota(y)\|+\epsilon$ for any $x,y \in M$, we get
$\rho_\eta\leq \widetilde\rho_\eta\leq C$.
Moreover, by Lemma \ref{retavsr}, we can assume that $\rho_\eta\geq \rho/2$.
Thus, we have
$$
C^{-1}\|v\|_2\leq\|v\|_{\widetilde\rho_\eta|_{\X}}\leq C \|v\|_2
$$
for any $v\in \R^n$, and
$$
C^{-1}\|f\|_1\leq\|f\|_{\rho\widetilde\rho_\eta}\leq C\|f\|_1
$$
for any $f\in L^2(M)$.
Therefore, we immediately get
\begin{equation*}
\begin{split}
\left\|\Lambda_\epsilon u|_{\X}-u\right\|_{\widetilde\rho_\eta|_{\X}}\leq& C\epsilon b_\epsilon(u)^{1/2},\\
\left\|\Lambda_\epsilon u-A\right\|_{\rho\widetilde\rho_\eta}\leq& C(\epsilon+\gamma^{1/2}\delta)\|u\|_{\widetilde\rho_\eta|_{\X}},\\
\left\|\Lambda_\epsilon u|_{\X}-A|_{\X}\right\|_{\widetilde\rho_\eta|_{\X}}\leq& C(\epsilon+\gamma^{1/2}\delta)\|u\|_{\widetilde\rho_\eta|_{\X}}.
\end{split}
\end{equation*}
by (\ref{Luvsu}), (\ref{LuvsA}) and (\ref{LuvsAX}),
where $A$ is a function defined in the proof of Lemma \ref{UNl2comp}.
Applying the Bernstein inequality to the function
\begin{equation*}
\psi\left(\frac{d(\cdot,x_i)}{\epsilon}\right)\psi\left(\frac{d(\cdot,x_j)}{\epsilon}\right)\frac{\widetilde\rho_\eta(\cdot)}{\rho(\cdot)^2}
\end{equation*}
for each $i,j$,
we get the inequality corresponding to (\ref{AXvsA}).
The rest of the proof is the same as Lemma \ref{UNl2comp}. 
\end{proof}

\begin{Lem}\label{NLOW}
Suppose that the pair $((M,g),\iota)$ satisfies Assumption \ref{Asu2} $(a)$, and we are given functions $\eta,\rho$ satisfying Assumption \ref{Asu3}.
Take arbitrary $k\in \Z_{>0}$.
Then, there exist constants $C_1=C_1(m,\alpha)>0$ and $C_2=C_2(m,S,K,i_0,L,\eta,\alpha,L_\rho,k)>0$ such that, for any
$\epsilon\in(\tau,2i_0 /3)$, $\gamma\in(1,\infty)$ and $\delta\in(0,\infty)$ with
$$
\frac{\tau}{\epsilon}+\epsilon+\gamma^{1/2} \delta\leq C_2^{-1}
$$
and i.i.d. sample $x_1,\ldots,x_n\in M$ from $\rho\Vol_g$, we have the following with probability at least $1-C_1(\epsilon^{-2m-1}+n^2)\exp\left(-n\gamma \epsilon^m\delta^2\right)$.
We have
$$
\lambda_k(\Delta_\rho^N)\leq \lambda_k(B_\epsilon^N)+C_2\left(\epsilon+\frac{\tau}{\epsilon}+\gamma^{1/2}\delta\right).
$$
and
$$
\left|\int_M|\Lambda_\epsilon u|^2\rho^2\, d\Vol_g-\frac{1}{n}\sum_{i=1}^n u_i^2 \widetilde D_i\right|\leq C_2\left(\epsilon+\frac{\tau}{\epsilon}+\gamma^{1/2}\delta\right)\frac{1}{n}\sum_{i=1}^n u_i^2 \widetilde D_i
$$
for any $u\in\Span\{u^{N,0},u^{N,1},\ldots,u^{N,k}\}$.
\end{Lem}
\begin{proof}
Since we have $\diam_g(M)\leq C(m,K,i_0,\alpha)$, there exist constants $\mu=\mu(m)>1$ and $C=C(m,K,i_0,\alpha,\mu)>0$ such that
\begin{equation}\label{Sobolev}
\begin{split}
&\left(\int_M h^{2\mu}\,d \Vol_g/\Vol_g(M)\right)^{1/(2\mu)}\\
\leq& C\left(\int_M |\nabla h|^2\,d \Vol_g/\Vol_g(M)\right)^{1/2}+\left(\int_M h^2\,d \Vol_g/\Vol_g(M)\right)^{1/2}
\end{split}
\end{equation}
holds for any $h\in W^{1,2}(M)$ by the Sobolev inequality (see Proposition 7.1.13 and Proposition 7.1.17 in \cite{Pe3}).
For example, if $m\geq 3$, we can take $\mu(m)=m/(m-2)$.

For every $h\in L^{2\mu}(M)$, we have
\begin{equation}\label{holder}
\begin{split}
&\left|\int_M h^2\rho \rho_\eta\,d \Vol_g-\int_M h^2\rho \widetilde\rho_\eta\,d\Vol_g\right|\\
\leq& \left(\int_M h^{2\mu}\rho\, d\Vol_g\right)^{1/\mu}\left(\int_M |\rho_\eta-\widetilde\rho_\eta|^{\frac{\mu}{\mu-1}}\rho\,d\Vol_g\right)^{\frac{\mu-1}{\mu}}\\
\leq &C \left(\int_M h^{2\mu}\rho\, d\Vol_g\right)^{1/\mu}\left(\int_M |\rho_\eta-\widetilde\rho_\eta|\,d\Vol_g \right)^{\frac{\mu-1}{\mu}}\\
\leq &C\left(\epsilon+\frac{\tau}{\epsilon}\right)^{\frac{\mu-1}{\mu}}\left(\int_M h^{2\mu}\rho\, d\Vol_g\right)^{1/\mu}
\end{split}
\end{equation}
by the H\"{o}lder inequality and Lemma \ref{retavstreta}.

Define
$$
\xi:=\sup\left\{\frac{\left|\int_M|\Lambda_\epsilon u|^2\rho^2\, d\Vol_g-\frac{1}{n}\sum_{i=1}^n u_i^2 \widetilde D_i\right|}{\frac{1}{n}\sum_{i=1}^n u_i^2 \widetilde D_i}:u\in\Span\{u^0,u^1,\ldots,u^k\}\setminus\{0\}\right\}.
$$
We can assume that
$
\lambda_k(B_\epsilon^N)\leq \lambda_k(\Delta_\rho^N)+1\leq C(m,K,\alpha,k)
$
by Lemma \ref{NUPP}, and so for any $u\in\Span\{u^{N,0},\ldots,u^{N,k}\}$,
$$
B_\epsilon^N(u)\leq C\frac{1}{n}\sum_{i=1}^n u_i^2 \widetilde D_i.
$$
For any $u\in\Span\{u^{N,0},\ldots,u^{N,k}\}$,
we have 
$$
\int_M|\nabla \Lambda_\epsilon u|^2\rho^2\, d\Vol_g\leq \left(1+C(\epsilon+\gamma^{1/2}\delta)\right)\lambda_k(B_\epsilon^N)\frac{1}{n}\sum_{i=1}^n u_i^2 \widetilde D_i.
$$
by Lemma \ref{UDcomp}, 
and
\begin{equation}\label{onesideNLOW}
\int_M|\Lambda_\epsilon u|^2\rho^2\,d\Vol_g\leq (1+C(\epsilon+\gamma^{1/2}\delta))\frac{1}{n}\sum_{i=1}^n u_i^2 \widetilde D_i
\end{equation}
by Lemma \ref{retavsr}, \ref{tDvstr}, \ref{preLuvsu} and $\rho_\eta\leq \widetilde\rho_\eta$.
Thus, for any $u\in\Span\{u^{N,0},\ldots,u^{N,k}\}$,
we get
$$
\left(\int_M |\Lambda_\epsilon u|^{2\mu}\,d \Vol_g\right)^{1/\mu}\leq \frac{C}{n}\sum_{i=1}^n u_i^2 \widetilde D_i
$$
by (\ref{Sobolev}), and so
\begin{align*}
&\left|\int_M|\Lambda_\epsilon u|^2\rho^2\, d\Vol_g-\frac{1}{n}\sum_{i=1}^n u_i^2 \widetilde D_i\right|\\
\leq& C\epsilon \int_M|\Lambda_\epsilon u|^2\rho^2\, d\Vol_g+C\left(\epsilon+\frac{\tau}{\epsilon}\right)^{\frac{\mu-1}{\mu}}\left(\int_M |\Lambda_\epsilon u|^{2\mu}\rho\, d\Vol_g\right)^{1/\mu}\\
&\qquad\qquad\qquad+C(\epsilon+\gamma^{1/2}\delta)\frac{1}{n}\sum_{i=1}^n u_i^2 \widetilde D_i\\
\leq &C\left(\left(\epsilon+\frac{\tau}{\epsilon}\right)^{\frac{\mu-1}{\mu}}+\epsilon+\gamma^{1/2}\delta\right)\frac{1}{n}\sum_{i=1}^n u_i^2 \widetilde D_i
\end{align*}
by Lemma \ref{retavsr}, \ref{tDvstr}, \ref{preLuvsu} and (\ref{holder}).
This implies that
$$
\xi\leq C\left(\left(\epsilon+\frac{\tau}{\epsilon}\right)^{\frac{\mu-1}{\mu}}+\epsilon+\gamma^{1/2}\delta\right).
$$
In particular, we can assume that $\xi\leq 1/2$.

Let $\Pj\colon L^2(M,\langle\cdot,\cdot\rangle_{\rho^2})\to L^2(M,\langle\cdot,\cdot\rangle_{\rho^2})$ denotes the orthogonal projection to the subspace of $L^2(M,\langle\cdot,\cdot\rangle_{\rho^2})$ spanned by the eigenfunctions of $\Delta_\rho^N$ corresponding to the eigenvalues in $[0,\lambda_k(\Delta_\rho^N)+1]$.
Then, we have
\begin{equation}\label{preNEig}
\left|\lambda_j(\Delta_\rho^N)- \lambda_j(B_\epsilon^N)\right|\leq C\left(\xi+\epsilon+\frac{\tau}{\epsilon}+\gamma^{1/2}\delta\right)
\end{equation}
for any $j\in\{1,\ldots,k\}$ and
\begin{align*}
\left\|\nabla \left((1-\mathbb{P})\Lambda_\epsilon u\right)\right\|_{\rho^2}^2\leq & C\left(\xi+\epsilon+\frac{\tau}{\epsilon}+\gamma^{1/2}\delta\right) \left\|\Lambda_\epsilon u\right\|_{\rho^2}^2,\\
\left\|(1-\mathbb{P})\Lambda_\epsilon u\right\|_{\rho^2}^2\leq & C\left(\xi+\epsilon+\frac{\tau}{\epsilon}+\gamma^{1/2}\delta\right) \left\|\Lambda_\epsilon u\right\|_{\rho^2}^2
\end{align*}
by the definition of $\xi$, Lemma \ref{UPP}, \ref{UDcomp}, \ref{Nl2tDf} and \ref{evalcomp} putting $a=1$.
Since we have $\|\mathbb P\Lambda_\epsilon u\|_{L^\infty}\leq C \|\Lambda_\epsilon u\|_{\rho^2}$ by Proposition \ref{LinfEF}, we get
\begin{align*}
&\left|\int_M |\Lambda_\epsilon u|^2\rho(\rho_\eta-\widetilde \rho_\eta)\,d \Vol_g\right|^{1/2}\\
\leq & \left|\int_M |(1-\mathbb P)\Lambda_\epsilon u|^2\rho(\rho_\eta-\widetilde \rho_\eta)\,d \Vol_g\right|^{1/2}+ \left|\int_M |\mathbb P \Lambda_\epsilon u|^2\rho(\rho_\eta-\widetilde \rho_\eta)\,d \Vol_g\right|^{1/2}\\
\leq &C\left(\epsilon+\frac{\tau}{\epsilon}\right)^{\frac{\mu-1}{2 \mu}}\left(\int_M |(1-\mathbb P)\Lambda_\epsilon u|^{2\mu}\rho\, d\Vol_g\right)^{1/2\mu}+C \left(\epsilon+\frac{\tau}{\epsilon}\right)^{1/2}\|\Lambda_\epsilon u\|_{\rho^2}\\
\leq& C\left(\left(\epsilon+\frac{\tau}{\epsilon}\right)^{\frac{\mu-1}{\mu}}\xi+\epsilon+\frac{\tau}{\epsilon}+\gamma^{1/2}\delta\right)^{1/2}\left(\frac{1}{n}\sum_{i=1}^n u_i^2 \widetilde D_i\right)^{1/2}
\end{align*}
by (\ref{Sobolev}), (\ref{holder}), (\ref{onesideNLOW}) and Lemma \ref{retavstreta}.
Thus, we have
\begin{align*}
&\left|\int_M|\Lambda_\epsilon u|^2\rho^2\, d\Vol_g-\frac{1}{n}\sum_{i=1}^n u_i^2 \widetilde D_i\right|\\
\leq & C\left(\left(\epsilon+\frac{\tau}{\epsilon}\right)^{\frac{\mu-1}{\mu}}\xi+\epsilon+\frac{\tau}{\epsilon}+\gamma^{1/2}\delta\right)\frac{1}{n}\sum_{i=1}^n u_i^2 \widetilde D_i
\end{align*}
by Lemma \ref{retavsr}, \ref{tDvstr} and \ref{preLuvsu},
and so
$$
\xi\leq C\left(\left(\epsilon+\frac{\tau}{\epsilon}\right)^{\frac{\mu-1}{\mu}}\xi+\epsilon+\frac{\tau}{\epsilon}+\gamma^{1/2}\delta\right).
$$
This implies $\xi \leq C(\epsilon+\tau/\epsilon+\gamma^{1/2}\delta)$.
Therefore, we get the lemma by the definition of $\xi$ and (\ref{preNEig}).
\end{proof}
Putting $\delta:=\epsilon^{-m/2}(\log n/n)^{1/2}$, we get the result corresponding to Theorem \ref{eigmain}  by Lemma \ref{NUPP} and \ref{NLOW}.
%
By and Lemma \ref{UPP}, \ref{UDcomp}, \ref{q2q1f-f}, \ref{Nl2tDf}, \ref{NLOW} and \ref{eveccomp}, we get the result corresponding to Theorem \ref{evecmain}.
Approximating the pair $(M,\iota)$ satisfying Assumption \ref{Asu}, we obtain the following theorems.
\begin{Thm}\label{NMainEval}
Suppose that the pair $((M,g),\iota)$ satisfies Assumption \ref{Asu}, and we are given functions $\eta, \rho$ satisfying Assumption \ref{Asu3}.
Take arbitrary $k\in \Z_{>0}$.
Then, there exist constants $C_1=C_1(m,\alpha,k)>0$ and $C_2=C_2(m,S,K,i_0,L,\eta,\alpha,L_\rho,k)>0$ such that, for any
$n\in\Z_{>0}$, $\epsilon\in(0,\infty)$, $\gamma\in(1,\infty)$ with
\begin{equation*}
\epsilon+\gamma^{1/2} \epsilon^{-m/2}\left(\frac{\log n}{n}\right)^{1/2}\leq C_2^{-1}
\end{equation*}
and i.i.d. sample $x_1,\ldots,x_n\in X$ from $\rho\Vol_g$, we have 
$$
\left|\lambda_k(\Delta_{\rho}^N)-\frac{2\widetilde\sigma_\eta}{\sigma_\eta n \epsilon^{2}}\lambda_k(\La,\D)\right|\leq C_2\left(\epsilon+\gamma^{1/2} \epsilon^{-m/2}\left(\frac{\log n}{n}\right)^{1/2}\right)
$$
with probability at least $1-C_1(\epsilon^{-2m-1}+n^2)n^{-\gamma}.$
\end{Thm}
\begin{Thm}\label{NMainEvec}
Suppose that the pair $((M,g),\iota)$ satisfies Assumption \ref{Asu}, and we are given functions $\eta, \rho$ satisfying Assumption \ref{Asu3}.
Put $s:=\lambda_l(\Delta_\rho^N)-\lambda_k(\Delta_\rho^N)$ and
$$
G:=\min\left\{|\lambda_k(\Delta_\rho^N)-\lambda_{k-1}(\Delta_\rho^N)|,|\lambda_{l+1}(\Delta_\rho^N)-\lambda_l(\Delta_\rho^N)|\right\}.
$$
Then, there exist constants $C_1=C_1(m,\alpha,l)>0$ and $C_2=C_2(m,S,K,i_0,L,\eta,\alpha,L_\rho,l)>0$ such that, for any
$n\in\Z_{>0}$, $\epsilon\in(0,\infty)$, $\gamma\in(1,\infty)$ with
$$
C_2\left(\epsilon+\gamma^{1/2} \epsilon^{-m/2}\left(\frac{\log n}{n}\right)^{1/2}\right)+2 G s < G^2
$$
and i.i.d. sample $x_1,\ldots,x_n\in M$ from $\rho\Vol_g$, we have the following with probability at least $1-C_1(\epsilon^{-2m-1}+n^2)n^{-\gamma}.$
Let $\Pj\colon \R^n \to \Span\{u^{N,k},\ldots,u^{N,l}\}$ be the orthogonal projection.
Then, the map $\Span\{f_k^N,\ldots,f_l^N\}\to \Span\{u^{N,k},\ldots,u^{N,l}\},\, f\mapsto \Pj(f|_{\X})$ is an isomorphism and for any $f\in\Span\{f_k^N,\ldots,f_l^N\}$, we have
\begin{align*}
\left(1-\frac{C_2}{G^2}\left(\epsilon+\gamma^{1/2}\epsilon^{-m/2}\left(\frac{\log n}{n}\right)^{1/2}\right)-\frac{2s}{G}\right)&\frac{1}{n}\sum_{i=1}^n f(x_i)^2\widetilde\D_{i}\leq\frac{1}{n}\sum_{i=1}^n\Pj(f|_{\X})_i^2\widetilde\D_i,\\
\left|\frac{1}{n}\sum_{i=1}^n f(x_i)^2\widetilde\D_{i}-\int_M f^2\rho^2\,d\Vol_g\right|\leq& C_2\gamma^{1/2}\left(\frac{\log n}{n}\right)^{1/2}\int_M f^2\rho^2\,d\Vol_g
\end{align*}
and
$$
\frac{1}{n}\sum_{i=1}^n (f(x_i)-\Pj(f|_{\X})_i)^2\widetilde\D_{i}\leq \left(\frac{C_2}{G^2}\left(\epsilon+\gamma^{1/2} \epsilon^{-m/2}\left(\frac{\log n}{n}\right)^{1/2}\right)+\frac{2s}{G}\right)\frac{1}{n} \sum_{i=1}^n f(x_i)^2\widetilde\D_{i}.
$$
\end{Thm}
\appendix
\section{$L^\infty$ and Gradient Estimate for Eigenfunctions}\label{GradientEstimate}
In this appendix, we show the $L^\infty$ and the gradient estimate for eigenfunctions of $\Delta_\rho$ and $\Delta_\rho^N$.
For any $p\in[1,\infty)$ measurable function $f$ on a closed Riemannian manifold $(M,g)$, define
$$
\|f\|_{L^p}^p:=\frac{1}{\Vol_g(M)}\int_M |f|^p\,d\Vol_g,
$$
and
$\|f\|_{L^\infty}=\esssup_{x\in M}|f|(x)$.

We first give the $L^\infty$ estimate.
The proof is similar to \cite[Theorem 9.2.7]{Pe3}.
\begin{Prop}\label{LinfEF}
Let $m\in \Z_{>0}$ be an integer and take constants $K,D,\alpha,\Lambda>0$.
Then, there exists a constant $C=C(m,K,D,\alpha,\Lambda)>0$ such that for any $(M,g)\in \M_2(m,K)$ with $\diam_g(M)\leq D$, any $\rho\in \LIP(M)$ with $1/\alpha\leq \rho\leq \alpha$, any $\lambda\in[0,\Lambda]$ and any solution $f\in W^{2,2}(M)$ to the equation $\Delta_\rho f=\lambda f$ or $\Delta_\rho^N f=\lambda f$,
we have
$
\|f\|_{L^\infty}\leq C\|f\|_{L^2}.
$
\end{Prop}
\begin{proof}
We first note that there exists constants $\mu=\mu(m)>1$ and $C_1=C_1(m,K,D,\mu)>0$ such that
$$
\|h\|_{L^{2\mu}}\leq C_1\|\nabla h\|_{L^2}+\|h\|_{L^2}
$$
holds for any $h\in W^{1,2}(M)$ by Proposition 7.1.13 and Proposition 7.1.17 in \cite{Pe3}.

For each $\epsilon>0$, put
\begin{align*}
u_\epsilon:=(f^2+\epsilon)^{1/2}.
\end{align*}
Then, we have
\begin{align*}
\Delta_\rho u_\epsilon=\frac{f}{u_\epsilon} \Delta_\rho f-\frac{\epsilon\rho}{u_\epsilon^3}|\nabla f|^2.
\end{align*}
For any $k\in \Z_{\geq 0}$, we have
\begin{align*}
\int_M|\nabla u_\epsilon^{\mu^k}|^2\rho^2\,d\Vol_g=&\mu^{2 k}\int_M u_\epsilon^{2\mu^k-2}|\nabla u_\epsilon|^2\rho^2\,d\Vol_g\\
=&\frac{\mu^{2 k}}{2\mu^k-1}\int_M u_\epsilon^{2\mu^k-1}\Delta_\rho u_\epsilon\rho\,d\Vol_g\\
\leq & \frac{\lambda\mu^{2 k}}{2\mu^k-1}  \int_M u_\epsilon^{2\mu^k-1}|f|\rho^i \,d\Vol_g\\
\to&  \frac{\lambda \mu^{2 k}}{2\mu^k-1} \int_M |f|^{2\mu^k}\rho^i\,d\Vol_g\quad(\epsilon\to 0),
\end{align*}
where $i=1$ if $\Delta_\rho f=\lambda f$ and $i=2$ if $\Delta_\rho^N f=\lambda f$.
Thus, we get
\begin{align*}
\|f\|_{L^{2\mu^{k+1}}}=\|f^{\mu^k}\|_{L^{2\mu}}^{1/\mu^k}\leq& \liminf_{\epsilon\to 0}\left(C_1 \|\nabla (u_\epsilon^{\mu^k})\|_{L^2}+\|u_\epsilon^{\mu^k}\|_{L^2}\right)^{1/\mu^k}\\
\leq& \left(C_1\mu^{ k}\left(\frac{\lambda \alpha^{i+2}}{2\mu^k-1}\right)^{1/2}+1\right)^{1/\mu^k} \|f\|_{L^{2^{\mu^k}}}.
\end{align*}
Therefore, we have
\begin{align*}
\|f\|_{L^\infty}=\lim_{k\to \infty}\|f\|_{L^{2\mu^k}}
\leq \prod_{k=0}^\infty\left(C_1\mu^{ k}\left(\frac{\lambda \alpha^{i+2}}{2\mu^k-1}\right)^{1/2}+1\right)^{1/\mu^k}\|f\|_{L^2}.
\end{align*}
By $\log(1+x)\leq x$ for $x\in [0,\infty)$, we have
\begin{align*}
\log \prod_{k=0}^\infty\left(C_1\mu^{ k}\left(\frac{\lambda \alpha^{i+2}}{2\mu^k-1}\right)^{1/2}+1\right)^{1/\mu^k}
\leq& \sum_{k=0}^\infty C_1\left(\frac{\lambda \alpha^{i+2}}{2\mu^k-1}\right)^{1/2}\\
\leq &\sum_{k=0}^\infty C_1\lambda^{1/2}\alpha^{1+i/2}\left(\frac{1}{\sqrt{\mu}}\right)^{k}\\
=&C_1 \lambda^{1/2}\alpha^{1+i/2}\frac{\sqrt{\mu}}{\sqrt{\mu}-1}.
\end{align*}
Thus, we get
$$
\|f\|_{L^\infty}\leq \exp\left(C_1 \lambda^{1/2}\alpha^{1+i/2}\frac{\sqrt{\mu}}{\sqrt{\mu}-1}\right)\|f\|_{L^2}.
$$
This is what we wanted to show.
\end{proof}
We next give the gradient estimate.
\begin{Prop}\label{GradEF}
Let $m\in \Z_{>0}$ be an integer and take constants $K,D,i_0,\alpha,L_\rho,\Lambda>0$.
Then, there exists a constant $C=C(m,K,D,i_0,\alpha,L_\rho,\Lambda)>0$ such that for any $(M,g)\in \M_2(m,K)$ with $\diam_g(M)\leq D$ and $\Inj_g\geq i_0$, any $\rho\in \LIP(M)$ with $1/\alpha\leq \rho\leq \alpha$ and $\Lip(\rho)\leq L_\rho$, any $\lambda\in[0,\Lambda]$ and any solution $f\in W^{2,2}(M)$ to the equation $\Delta_\rho f=\lambda f$ or $\Delta_\rho^N f=\lambda f$,
we have
$
\Lip(f)\leq C\|f\|_{L^2}.
$
\end{Prop}
\begin{proof}
Take arbitrary $Q\in (1,\infty)$ and $p\in(m,\infty)$ (e.g. we can take $Q=2$ and $p=2m$).
Then, there exists a constant $r_H=r_H(m,K,i_0,Q,p)>0$ such that for any $x\in M$, there exists a coordinate chart $\psi_x=(x_1,\ldots,x_n)\colon B_{r_H}(x)\to \R^m$ satisfying $\psi_x(x)=0$ and the following properties by \cite[Theorem 0.3]{AC} (see also \cite[Appendix A]{Por}):
\begin{itemize}
\item[(i)] We have $\Delta \psi_x=0$ on $B_{r_H}(x)$.
\item[(ii)] We have $Q^{-1}\psi_x^\ast g_0\leq g\leq Q\psi_x^\ast g_0$, where $g_0$ denotes the standard Euclidean metric on $\R^m$.
\item[(iii)] Putting $g_{i j}:=g(\partial/\partial x_i,\partial /\partial x_j)$, we have
$$
r_H^{1-m/p}\left(\int_{\psi_x(B_{r_H}(x))}|\partial g_{i j}|^p(y)\, d y\right)^{1/p}\leq Q-1.
$$
\end{itemize}
We have that $\Delta=-\sum_{i,j=1}^m  g^{i j}\partial^2/\partial x_i \partial x_j$ by (i), where $\{g^{i j}\}$ denotes the inverse matrix of $\{g_{i j}\}$, and
$\|g^{i j}\|_{C^\alpha}\leq C$ by (ii), (iii) and the Sobolev inequality, where $\alpha:=1-n/p$.
By (ii), we have $B^{\R^m}(0,r_H/\sqrt{Q})\subset \psi_x(B_{r_H}(x))$.
Through the coordinate chart $\psi_x$, $\Delta_\rho$ and $\Delta_\rho^N$ have the following representation:
\begin{align*}
\Delta_\rho=&-\rho\sum_{i,j=1}^n g^{i j}\frac{\partial^2}{\partial x_i \partial x_j} -2\sum_{i,j=1}^n g^{i j}\frac{\partial \rho}{\partial x_i}\frac{\partial}{\partial x_j},\\
\Delta_\rho^N=&-\sum_{i,j=1}^n g^{i j}\frac{\partial^2}{\partial x_i \partial x_j} -2\sum_{i,j=1}^n g^{i j}\frac{1}{\rho}\frac{\partial \rho}{\partial x_i}\frac{\partial}{\partial x_j}.
\end{align*}
Thus, we can apply the $L^p$ elliptic estimate \cite[Theorem 9.11]{GT} to $\Delta_\rho$ and $\Delta_\rho^N$, and get
$$
\|f\circ \psi_x^{-1}\|_{2,p;\Omega'}\leq C(\lambda+1)\|f\circ \psi_x^{-1}\|_{p; \Omega}\leq C\|f\|_{L^2}
$$
by Proposition \ref{LinfEF}, where $\Omega=B^{\R^m}(0,r_H/\sqrt Q)$, $\Omega'=B^{\R^m}(0,r_H/(2\sqrt Q))$ and
\begin{align*}
\|h\|_{2,p;\Omega'}^p:=&\int_{\Omega'}\left(\sum_{i,j=1}^m\left|\frac{\partial^2 h}{\partial x_i\partial x_j}\right|^p+\sum_{i=1}^m \left|\frac{\partial h}{\partial x_i}\right|^p+|h|^p\right)(x)\,d x,\\
\|h\|_{p:\Omega}^p:=&\int_{\Omega} |h|^p(x)\,d x
\end{align*}
for any $h\in W^{2,p}(\Omega)$.
Combining this, the Sobolev embedding $W^{2,p}\to C^{1,\alpha}$ and (ii), we get the proposition.
\end{proof}

\section{Linear Algebraic Arguments for Eigenvalue Problems}\label{LinearAlgebra}
In this section, we give some linear algebraic assertions required for our spectral convergence.
The discussion in this section is essentially written in section 7 of \cite{BIK}, but we extract the linear algebraic arguments from there in the form we use.
For a symmetric form $D\colon H\times H\to \R$ on a finite dimensional inner product space $(H,\langle\cdot,\cdot\rangle)$, let
$$
\lambda_1(D)\leq \lambda_2(D)\leq \cdots \leq \lambda_{\dim H}(D)
$$
denote its eigenvalues counted with multiplicities.
By the minimax principle, for any $k\in\{1,2,\ldots,\dim H\}$, $\lambda_k(D)$ is expressed as
$$
\lambda_k(D)=\min \left\{\max_{v\in V\setminus\{0\}}\frac{D(v,v)}{\|v\|_{H}^2}: V\subset H\text{ is a $k$-dimensional subspace}\right\}.
$$
In this section, we promise that $a/0=\infty$ for any $a\in\R$.
\begin{Lem}\label{dtilde}
Let $(H,\langle\cdot,\cdot\rangle)$ be a finite dimensional inner product space, and $D\colon H\times H\to \R$ be a symmetric form with $D\geq 0$.
Take arbitrary $\lambda,\lambda'>0$ with $\lambda'\geq \lambda$, and let $\Pj\colon H\to H$ be the orthogonal projection to the subspace of $H$ spanned by the eigenvectors of $D$ corresponding to the eigenvalues in $[0,\lambda']$.
Define a symmetric form $\widetilde D\colon H\times H\to \R$ by
$$
\widetilde D(u,v):=D(\Pj u,\Pj v)+\lambda\left\langle(1-\Pj)u ,(1-\Pj)v \right\rangle.
$$
Then, we have
$D\geq \widetilde D$, and for any subspace $L\subset H$ and $j\in\{1,\ldots,\dim L\}$,
$$
\lambda_j\left(\widetilde D|_{L}\right)\geq\min\{\lambda,\lambda_j(D)\}.
$$
\end{Lem}
\begin{proof}
The first assertion is trivial.
Let us show the second assertion.
Let $\{u_1,\ldots, u_{\dim H}\}$ be the orthonormal bases of $H$ consisting of the eigenvectors of $D$ corresponding to the eigenvalues $\{\lambda_1(D),\ldots,\lambda_{\dim H}(D)\}$.
Take arbitrary $j\in\{1,\ldots,\dim L\}$ and $j$-dimensional subspace $V\subset H_2$.
We can take $v\in V$ such that $v\perp u_1,\ldots,u_{j-1}$ and $v\neq 0$.
Then, we have $\Pj v\perp u_1,\ldots,u_{j-1}$, and so
\begin{align*}
\widetilde D(v,v)=&D(\Pj v,\Pj v)+\lambda\|(1-\Pj)v\|^2\\
\geq &\lambda_j(D)\|\Pj v\|^2 +\lambda\|(1-\Pj)v\|^2\geq \min\{\lambda,\lambda_j(D)\}\|v\|^2.
\end{align*}
This implies the lemma.
\end{proof}

\begin{Lem}\label{evalcomp0}
Let $(H_i,\langle\cdot,\cdot\rangle_i)$ $(i=1,2)$ be finite dimensional inner product spaces, and $D_i\colon H_i\times H_i\to \R$ $(i=1,2)$ be symmetric forms.
Let $\{f_1,... ,f_{\dim H_1}\}$ be the orthonormal basis of $H_1$ consisting of the eigenvectors of $D_1$ corresponding to the eigenvalues $\{\lambda_1(D_1),\ldots,\lambda_{\dim H_1}(D_1)\}$.
Take arbitrary $k\in \Z_{>0}$ with $k\leq \min\{\dim H_1,\dim H_2\}$.
Suppose that we are given a linear maps
$Q_1\colon H_1\to H_2$.
Define
\begin{align*}
E:=&\sup_{f\in\Span\{f_1,\ldots,f_k\}\setminus\{0\}}\left(\frac{D_2(Q_1 f,Q_1 f)}{\|Q_1 f\|_{H_2}^2}-\frac{D_1(f,f)}{\|f\|_{H_1}^2}\right).
\end{align*}
Then, we have $\lambda_j(D_2)\leq \lambda_j(D_1)+E$ for any $j\in\{1,\ldots,k\}$.
\end{Lem}
\begin{proof}
The assertion is a direct consequence of the minimax principle and the definition of $E$.
\end{proof}

\begin{Lem}\label{evalcomp}
Let $(H_i,\langle\cdot,\cdot\rangle_i)$ $(i=1,2)$ be finite dimensional inner product spaces, and $D_i\colon H_i\times H_i\to \R$ $(i=1,2)$ be symmetric forms with $D_i\geq 0$.
Let $\{f_1,... ,f_{\dim H_1}\}$ and $\{u_1,... ,u_{\dim H_2}\}$ be the orthonormal bases of $H_1$ and $H_2$ consisting of the eigenvectors of $D_1$ and $D_2$ corresponding to the eigenvalues $\{\lambda_1(D_i),\ldots,\lambda_{\dim H_i}(D_i)\}$ $(i=1,2)$, respectively.
Take arbitrary $l\in \Z_{>0}$ with $l\leq \min\{\dim H_1,\dim H_2\}$.
Suppose that we are given linear maps
$Q_1\colon H_1\to H_2$ and
$Q_2\colon H_2\to H_1$.
Define
\begin{align*}
E_1:=&\sup_{u\in\Span\{u_1,\ldots,u_l\}\setminus\{0\}}\left(\frac{D_1(Q_2 u,Q_2 u)}{\|Q_2 u\|_{H_1}^2}-\frac{D_2(u,u)}{\|u\|_{H_2}^2}\right),\\
E_2:=&\sup_{f\in\Span\{f_1,\ldots,f_l\}\setminus\{0\}}\left(\frac{D_2(Q_1 f,Q_1 f)}{\|Q_1 f\|_{H_2}^2}-\frac{D_1(f,f)}{\|f\|_{H_1}^2}\right).
\end{align*}
Then, we have the following:
\begin{itemize}
\item[(i)] For any $j\in\{1,\ldots,l\}$,
$\lambda_j(D_1)\leq \lambda_j(D_2)+E_1$ and
$\lambda_j(D_2)\leq \lambda_j(D_1)+E_2$.
\item[(ii)] Take arbitrary $a>0$, and let $\Pj\colon H_2\to H_2$ be the orthogonal projection to the subspace of $H_2$ spanned by the eigenvectors of $D_2$ corresponding to the eigenvalues in $[0,\lambda_l(D_1)+a]$.
Then, for any $f\in\Span \{f_1,\ldots,f_l\}$, we have
\begin{align*}
\|(1-\Pj)Q_1 f\|_{H_2}^2\leq& \frac{l}{a}\left((E_1)_{+}+E_2\right)\|Q_1f\|_{H_2}^2,\\
D_2\left((1-\Pj)Q_1 f,(1-\Pj)Q_1 f\right)\leq& \frac{\lambda_l(D_1)+a}{a}l\left((E_1)_{+}+E_2\right)\|Q_1f\|_{H_2}^2,
\end{align*}
where $(E_1)_{+}:=\max\{E_1,0\}$.
\end{itemize}
\end{Lem}
\begin{proof}
We immediately get (i) by Lemma \ref{evalcomp0}

Let us prove (ii).
Define $\widetilde D_2$ as in Lemma \ref{dtilde} for $\left((H_2,\langle\cdot,\cdot\rangle),D_2\right)$ putting $\lambda=\lambda_l(D_1)$ and $\lambda'=\lambda_l(D_1)+a$.
Defining $L:=Q_1\left(\Span\{f_1,\ldots,f_l\}\right)$, we get
\begin{equation}\label{lE1}
\lambda_j(\widetilde D_2|_L)\geq \min\{\lambda_l(D_1),\lambda_j(D_2)\}\geq \lambda_j(D_1)-(E_1)_{+}
\end{equation}
for any 
$j\in \{1,\ldots, l\}$ by Lemma \ref{dtilde} and (i).
By the definition of $E_2$, we have
\begin{equation}\label{lE2}
\lambda_j(D_2|_L)\leq\lambda_j(D_1)+E_2
\end{equation}
for any $j\in \{1,\ldots, l\}$.
By (\ref{lE1}) and (\ref{lE2}), we get
$$
0\leq \lambda_j(D_2|_{L})-\lambda_j(\widetilde D_2|_{L})\leq (E_1)_{+}+E_2
$$
for any $j\in \{1,\ldots, l\}$, and so
\begin{equation*}
\begin{split}
0\leq D_2(u,u)-\widetilde D_2(u,u)\leq \tr \left(D_2|_L-\widetilde D_2|_L\right)\|u\|^2_{H_2}\leq l\left((E_1)_{+}+E_2\right)\|u\|^2_{H_2}
\end{split}
\end{equation*}
for any $u\in L$.
This and
$$
D_2\left((1-\Pj)u,(1-\Pj)u\right)\geq (\lambda_l(D_1)+a)\|(1-\Pj)u\|_{H_2}^2
$$
imply
\begin{equation*}
\begin{split}
l\left((E_1)_{+}+E_2\right)\|u\|^2_{H_2}\geq& D_2(u,u)-\widetilde D_2(u,u)\\
\geq &D_2\left((1-\Pj)u,(1-\Pj)u\right)-\lambda_l(D_1)\|(1-\Pj)u\|_{H^2}^2\\
\geq &\frac{a}{\lambda_l(D_1)+a}D_2\left((1-\Pj)u,(1-\Pj)u\right)\geq a\|(1-\Pj)u\|_{H_2}^2
\end{split}
\end{equation*}
for any $u\in L$.
Thus, we get (ii).
\end{proof}

\begin{Lem}\label{eveccomp}
Let $(H_i,\langle\cdot,\cdot\rangle_i)$ $(i=1,2)$ be finite dimensional inner product spaces, and $D_i\colon H_i\times H_i\to \R$ $(i=1,2)$ be symmetric forms with $D_i\geq 0$.
Let $\{f_1,... ,f_{\dim H_1}\}$ and $\{u_1,... ,u_{\dim H_2}\}$ be the orthonormal bases of $H_1$ and $H_2$ consisting of the eigenvectors of $D_1$ and $D_2$ corresponding to the eigenvalues $\{\lambda_1(D_i),...\lambda_{\dim H_i}(D_i)\}$ $(i=1,2)$, respectively.
Take arbitrary $k,l\in \Z_{>0}$ with $2\leq k\leq l\leq \min\{\dim H_1,\dim H_2\}-1$.
Suppose that we are given linear maps
$Q_1\colon H_1\to H_2$ and
$Q_2\colon H_2\to H_1$.
Define
\begin{align*}
S:=&\Span\{f_k,\ldots,f_l\}\subset H_1,\\
\widetilde S:=&\Span\{u_k.\ldots, u_l\}\subset H_2,\\
E_1:=&\sup_{u\in\left(\Span\{u_1,\ldots,u_{l+1}\}\cup Q_1( S)\right)\setminus\{0\}}\left(\frac{D_1(Q_2 u,Q_2 u)}{\|Q_2 u\|_{H_1}^2}-\frac{D_2(u,u)}{\|u\|_{H_2}^2}\right),\\
E_2:=&\sup_{f\in\Span\{f_1,\ldots,f_{l+1}\}\setminus\{0\}}\left(\frac{D_2(Q_1 f,Q_1 f)}{\|Q_1 f\|_{H_2}^2}-\frac{D_1(f,f)}{\|f\|_{H_1}^2}\right),\\
E_3:=&\sup_{f\in S\setminus \{0\}}\frac{\|f-Q_2 Q_1 f\|_{H_1}}{\|f\|_{H_1}},\\
E_4:=&\sup_{f\in S\setminus \{0\}}\frac{\left|\|Q_1 f\|_{H_2}-\|f\|_{H_1}\right|}{\|f\|_{H_1}}.
\end{align*}
Let $\Pj_{\widetilde S}\colon H_2\to H_2$ be the orthogonal projection to $\widetilde S$. 
Put $s:=\lambda_l(D_1)-\lambda_k(D_1)$, and suppose that
$$
\gamma:=
\frac{1}{2}\min\{|\lambda_k(D_1)-\lambda_{k-1}(D_1)|,|\lambda_{l+1}(D_1)-\lambda_l(D_1)||\}>\max\{E_1,E_2\}.
$$
Then, putting
$$
F:=\frac{1}{\gamma}\left(\left(\left(\frac{\lambda_l(D_1)}{\gamma}+2\right)l+1\right)((E_1)_{+}+(E_2)_+)+4\lambda_l(D_1)E_3+s\right),
$$
and supposing $F<1$,
we have the following:
\begin{itemize}
\item[(i)] For any $f\in S$, we have
$$
\|(1-\Pj_{\widetilde S})Q_1 f\|_{H_2}^2\leq F\|Q_1 f\|_{H_2}^2.
$$
In particular, $\Pj_{\widetilde S}\circ Q_1|_{S}\colon S\to \widetilde S$ is an isomorphism.
\item[(ii)] For any $f\in S$, we have
\begin{align*}
(1-F)^{1/2}\|Q_1 f\|_{H_2}&\leq \|\Pj_{\widetilde S}Q_1 f\|_{H_2}\leq \|Q_1 f\|_{H_2},\\
(1-E_4)\|f\|_{H_1}&\leq \|Q_1 f\|_{H_2}\leq (1+E_4)\|f\|_{H_1}.
\end{align*}
\item[(iii)] For any $v\in \widetilde S$, there exists $f\in S$ such that
$$
\|v-Q_1 f\|_{H_2}\leq\left(\frac{F}{1-F}\right)^{1/2}\|v\|_{H_2}. 
$$
\item[(iv)] For any $v\in \widetilde S$, there exists $f\in S$ with $\|f\|_{H_1}=\|v\|_{H_2}$ such that
$$
\|v-Q_1 f\|_{H_2}\leq\left(\frac{1}{1-F}\right)^{1/2}\left(F^{1/2}(2+E_4)+E_4\right)\|v\|_{H_2}. 
$$
\end{itemize}
\end{Lem}
\begin{proof}
By Lemma \ref{evalcomp} (i) and the assumption for $\gamma$, we have
\begin{align*}
\lambda_{k-1}(D_2)<&\lambda_{k-1}(D_1)+\gamma\leq \lambda_k(D_1)-\gamma<\lambda_k(D_2),\\
\lambda_l(D_2)\leq \lambda_l(D_1)+E_2<&\lambda_l(D_1)+\gamma\leq \lambda_{l+1}(D_1)-\gamma<\lambda_{l+1}(D_2).
\end{align*}
Thus, we have
\begin{equation}
\Pj_{\widetilde S}=\Pj_{(\lambda_k(D_1)-\gamma,\lambda_l(D_1)+\gamma]}=\Pj_{(\lambda_k(D_1)-\gamma,\lambda_l(D_1)+E_2]},
\end{equation}
where $\Pj_J\colon H_2\to H_2$ denotes the orthogonal projection to the subspace of $H_2$ spanned by the eigenvectors of $D_2$ corresponding to the eigenvalues in $J$ for any interval $J\subset \R$.

Take arbitrary $f\in S$ and put $u:=Q_1 f\in Q_1(S)$.
Decompose $u$ as
$
u=u_0+u_-+u_+,
$
where we defined
\begin{align*}
u_0:=\Pj_{\widetilde S} u,\quad
u_-:=\Pj_{[0,\lambda_k(D_1)-\gamma]} u,\quad
u_+:=\Pj_{(\lambda_l(D_1)+\gamma,\infty)}u.
\end{align*}
Then, we have
\begin{align}
\label{u+bou} \|u_+\|_{H_2}^2\leq& \frac{l}{\gamma}\left((E_1)_{+}+E_2\right)\|u\|_{H_2}^2,\\
\label{du+bou} D_2\left(u_+,u_+\right)\leq& \frac{\lambda_l(D_1)+\gamma}{\gamma}l\left((E_1)_{+}+E_2\right)\|u\|_{H_2}^2
\end{align}
by Lemma \ref{evalcomp} (ii).
We next estimate $\|u_-\|_{H_2}$.
Let $p_S \colon H_1\to H_1$ denotes the orthogonal projection to $S$.
Then,
$$
f=p_S Q_2 u+(f-p_S Q_2 u)=p_S Q_2 u+p_S (f-Q_2 Q_1 f),
$$
and so
\begin{align*}
D_1(Q_2 u,Q_2 u)^{1/2}\geq &D_1(p_S Q_2 u, p_S Q_2 u)^{1/2}\\
\geq& D_1(f,f)^{1/2}-D_1\left(p_S (f-Q_2 Q_1 f),p_S (f-Q_2 Q_1 f)\right)^{1/2}\\
\geq& \left((\lambda_l(D_1)-s)^{1/2}-\lambda_l(D_1)^{1/2}E_3\right)\|f\|_{H_1}.
\end{align*}
By this and the definition of $E_1$, putting
$$
G:=\frac{\|f\|_{H_1}^2}{\|Q_2 Q_1 f\|_{H_1}^2}\left((\lambda_l(D_1)-s)^{1/2}-\lambda_l(D_1)^{1/2}E_3\right)^2,
$$
we get
\begin{equation}\label{estbyG}
D_2(u,u)\geq \left(G-E_1\right)\|u\|_{H_2}^2.
\end{equation}
By (\ref{du+bou}), we have
\begin{align*}
&D_2(u,u)\\
=&D_2(u_0,u_0)+D_2(u_-,u_-)+D_2(u_+,u_+)\\
\leq &(\lambda_l(D_1)+E_2)\|u_0\|_{H_2}^2+(\lambda_l(D_1)-s-\gamma)\|u_-\|_{H_2}^2+ \frac{\lambda_l(D_1)+\gamma}{\gamma}l\left((E_1)_{+}+E_2\right)\|u\|_{H_2}^2\\
\leq &\left(\lambda_l(D_1)+(E_2)_+ +\frac{\lambda_l(D_1)+\gamma}{\gamma}l\left((E_1)_{+}+E_2\right)\right)\|u\|_{H_2}^2-\gamma \|u_-\|_{H_2}^2.
\end{align*}
Combining this with (\ref{estbyG}), we get
\begin{equation}\label{u-bou}
\gamma\|u_-\|_{H_2}^2\leq \left(\left(\frac{\lambda_l(D_1)+\gamma}{\gamma}l+1\right)\left((E_1)_{+}+(E_2)_+\right)+\lambda_l(D_1)-G\right)\|u\|_{H_2}^2.
\end{equation}
By the definition of $G$ and $E_3$, we have
\begin{align*}
\lambda_l(D_1)-G\leq& \lambda_l(D_1)-\left(\frac{(\lambda_l(D_1)-s)^{1/2}-\lambda_l(D_1)^{1/2}E_3}{1+E_3}\right)^2\\
=&\frac{2\lambda_l(D_1)E_3+2\lambda_l(D_1)^{1/2}E_3(\lambda_l(D_1)-s)^{1/2}+s}{1+2 E_3+E_3^2}\\
\leq &4\lambda_l(D_1) E_3 + s.
\end{align*}
Thus, (\ref{u+bou}) and (\ref{u-bou}) imply (i).

We immediately get the first assertion of (ii) by (i), and the second assertion by the definition of $E_4$.


Combining (i) and (ii), we immediately get (iii).

Let us prove (iv).
Take arbitrary $v\in \widetilde S$.
Then, we can take $f\in S$ such that $v=\Pj_{\widetilde S} Q_1 f$.
By (i) and (ii), we have
\begin{align*}
&\left\|v-Q_1\left( \frac{\|v\|_{H_2}}{\|f\|_{H_1}}f\right)\right\|_{H_2}\\
\leq &\|v-Q_1 f\|_{H_2}+\frac{\|Q_1 f\|_{H_2}}{\|f\|_{H_1}}\big|\|f\|_{H_1}-\|v\|_{H_2}\big|\\
\leq & F^{1/2}\|Q_1 f\|_{H_2}+\frac{\|Q_1 f\|_{H_2}}{\|f\|_{H_1}}\left(\big|\|f\|_{H_1}-\|Q_1 f\|_{H_2}\big|+\|v-Q_1 f\|_{H_2}\right)\\
\leq & \left(F^{1/2}+E_4\right)\|Q_1 f \|_{H_2}+F^{1/2}\frac{\|Q_1 f\|_{H_2}^2}{\|f\|_{H_1}}\\
\leq &\left(\frac{1}{1-F}\right)^{1/2}\left(F^{1/2}(2+E_4)+E_4\right)\|v\|_{H_2}.
\end{align*}
Since $\left\|(\|v\|_{H_2}/\|f\|_{H_1})f\right\|_{H_1}=\|v\|_{H_2}$, we get (iv).
\end{proof}

\section{Map from the Manifold to Random Points}\label{MapStD}

In this section, we construct a map from given Riemannian manifold to random points on it used in the proof of Lemma \ref{UDcomp}.
The following lemma corresponds to \cite[Proposition 2.11]{CT}.
\begin{Lem}\label{transmap}
Given an integer $m\in \Z_{>0}$ and positive real numbers $K,i_0,\alpha>0$, there exist constants $C_1=C_1(m,\alpha)>0$ and $C_2=C_2(m,\alpha)>0$ such that for any $n\in \Z_{>0}$, $\tilde \epsilon\in(0,\min\{4i_0,\pi/\sqrt{K}\})$, $\gamma\in(1,\infty)$, $\tilde\delta\in (0,\gamma^{-1/2}]$ with 
$
\gamma^{1/2}\delta\leq C_2^{-1},
$
$m$-dimensional closed Riemannian manifold $(M,g)$ with $|\Sect_g|\leq K$ and $\Inj_g(M)\geq i_0$,
probability density function $\rho\colon M\to \R_{>0}$ with $1/\alpha\leq \rho\leq \alpha$, and i.i.d. sample $\X=\{x_1,\ldots,x_n\}\subset M$ from $\rho\Vol_g$, there exists a Borel map $T\colon M\to \X$ such that
$d_g(x,T(x))\leq \tilde \epsilon$ for any $x\in M$, and
$$
\left|\frac{1}{n} -(\rho\Vol_g)(T^{-1}(\{x_i\}))\right|\leq  C_2\gamma^{1/2}\tilde \delta(\rho\Vol_g)(T^{-1}(\{x_i\}))
$$
holds for every $i\in\{1,\ldots, n\}$ with probability at least $1-C_1\tilde\epsilon^{-m}\exp(-n\gamma \tilde \epsilon^m\tilde \delta^2)$.
\end{Lem}
\begin{proof}
Define
$$
N:=\max\left\{k\in \Z_{>0}:\begin{array}{l}\text{there exist points $y_1,\ldots,y_k\in M$ such that}\\
B_{\tilde\epsilon/4}(y_s)\cap B_{\tilde\epsilon/4}(y_t)=\emptyset \text{ for any $s,t$ with $s\neq t$}\end{array}\right\}.
$$
By Theorem \ref{VolLow}, we have $\Vol_g\left(B(y,\tilde \epsilon/4)\right)\geq 1/(\sqrt 2\pi)^{m-1}\tilde\epsilon^{m}\Vol(S^{m-1})/(4m)$ for any $y\in M$, and so $N\leq C(m,\alpha)\tilde\epsilon^{-m}$. Fix points $y_1,\ldots, y_N\in M$ that attain $N$.
Then, by the maximality, we have
$
M=\bigcup_{s=1}^N B_{\tilde\epsilon/2}(y_s).
$
Define 
$$
V_1:=B_{\tilde\epsilon/2}(y_1)\setminus \bigcup_{t\neq 1} B_{\tilde\epsilon/4}(y_t),\quad
V_s:=B_{\tilde\epsilon/2}(y_s)\setminus \left(\bigcup_{t\neq s} B_{\tilde\epsilon/4}(y_t)\cup \bigcup_{t=1}^{s-1} V_t\right)
$$
for any $s\in \{2,\ldots, N\}$.
Then, we have $
M=\bigcup_{s=1}^N V_s,
$
$V_s\cap V_t=\emptyset$ for any $s,t$ with $s\neq t$ and
\begin{equation}\label{intep}
B_{\tilde\epsilon/4}(y_s)\subset V_s \subset B_{\tilde\epsilon/2}(y_s).
\end{equation}
Then, by Theorem \ref{BishopGromov} (iii) and \ref{VolLow}, we have
\begin{equation}\label{LandUbound}
C(m,\alpha)^{-1}\tilde\epsilon^m\leq (\rho\Vol_g)(V_s)\leq C(m,\alpha)\tilde\epsilon^m
\end{equation}
holds for every $s\in\{1,\ldots, N\}$.
Put $n_s:=\Card(\X\cap V_s)$ and represent $\X\cap V_s$ as $\X\cap V_s=\{x_s^1,\ldots, x_s^{n_s}\}$ for each $s\in\{1,\ldots, N\}$.
Applying the Bernstein inequality to $\1_{V_s}$, we have
\begin{equation}\label{nsest}
\left|\frac{n_s}{n}-(\rho\Vol_g)(V_s)\right|\leq C(m,\alpha) \gamma^{1/2}\tilde\epsilon^m\tilde\delta\leq C(m,\alpha)\gamma^{1/2}\tilde\delta(\rho\Vol_g)(V_s)
\end{equation}
for every $s\in\{1,\ldots, N\}$  with probability at least $1-C(m,\alpha)\tilde\epsilon^{-m}\exp(-n\gamma \tilde \epsilon^m\tilde \delta^2)$.
In particular, we have $n_s>0$ for all $s$.
For each $s\in\{1,\ldots,N\}$, we can take Borel subsets $W_s^1,\ldots,W_s^{n_s}\subset V_s$ such that
$V_s=\bigcup_{i=1}^{n_s}W_s^{i}$, $W_s^i\cap W_s^j=\emptyset$ for any $i,j$ with $i\neq j$ and
$$
(\rho\Vol_g)(W_s^1)=\cdots=(\rho\Vol_g)(W_s^{n_s})=\frac{(\rho\Vol_g)(V_s)}{n_s}.
$$
We can take such subsets because the function $(0,\infty)\to \R_{\geq 0},\,r\mapsto (\rho\Vol_g)(A\cap B_r(x))$ is continuous for any Borel subset $A\subset M$ and $x\in M$.
Define a map $T\colon M\to \X$ by $T|_{W_s^i}\equiv x_s^i$.
Then, (\ref{intep}) and (\ref{nsest}) imply the lemma.
\end{proof}

We get similar assertion under the assumption on the Ricci curvature.
\begin{Lem}\label{Rictransmap}
Given an integer $m\in \Z_{>0}$ and positive real numbers $K,D,\alpha>0$, there exist constants $C_1=C_1(m,K,D)>0$ and $C_2=C_2(m,K,D,\alpha)>0$ such that for any $n\in \Z_{>0}$, $\tilde \epsilon\in(0,\infty)$, $\gamma\in(1,\infty)$, $\tilde\delta\in (0,\gamma^{-1/2}]$ with 
$
\gamma^{1/2}\delta\leq C_2^{-1},
$
$m$-dimensional closed Riemannian manifold $(M,g)$ with $\Ric_g\geq -(m-1)K g$ and $\diam_g(M)\leq D$,
probability density function $\rho\colon M\to \R_{>0}$ with $1/\alpha\leq \rho\leq \alpha$, and i.i.d. sample $\X=\{x_1,\ldots,x_n\}\subset M$ from $\rho\Vol_g$, there exists a Borel map $T\colon M\to \X$ such that
$d_g(x,T(x))\leq \tilde \epsilon$ for any $x\in M$, and
$$
\left|\frac{1}{n} -(\rho\Vol_g)(T^{-1}(\{x_i\}))\right|\leq  C_2\gamma^{1/2}\tilde \delta(\rho\Vol_g)(T^{-1}(\{x_i\}))
$$
holds for every $i\in\{1,\ldots, n\}$ with probability at least $1-C_1\tilde\epsilon^{-m}\exp(-n\gamma \tilde \epsilon^m\tilde \delta^2)$.
\end{Lem}
\begin{proof}
We first note that $1/\alpha\leq \Vol_g(M)\leq C(m,K,D)$ by the assumption on $\rho$ and the Bishop inequality.
It is enough to prove the lemma for the case $\tilde \epsilon\leq D$.
By the Bishop-Gromov inequality, we have $N\leq C(m,K,D)\tilde\epsilon^{-m}$, where $N$ is defined as in the proof of Lemma \ref{transmap}.
For the estimate corresponding to (\ref{LandUbound}), we use Theorem \ref{BishopGromov} (iii) to get
$$
C(m,K,D,\alpha)^{-1}\tilde\epsilon^m\leq (\rho\Vol_g)(V_s)\leq C(m,K,D,\alpha)\tilde\epsilon^m.
$$
Then, the remaining part of the proof is similar.
\end{proof}

\section{Approximation by a Sequence of Smooth Submanifolds}\label{ApproxSmooth}
In this section we work under Assumption \ref{Asu} and \ref{Asu3}, and show that such an approximation behaves well for our eigenvalue problems.
Suppose that we are given a pair $(M,\iota)$ satisfying Assumption \ref{Asu}, and functions $\eta\colon [0,\infty)\to [0,\infty)$ and $\rho\colon M\to \R$ satisfying Assumption \ref{Asu3}.
Let $\{((M_t,g_t),\iota_t)\}_{t\in\Z_{>0}}$ be the approximation sequence and $\psi_t\colon M\to M_t$ be the approximation map.
We only give rates for convergence as $n\to \infty$ for sample size in the limit space $M$, and do not care about rates for spectral convergence due to convergence of manifolds and random points as $t\to \infty$.

Under our assumptions,
$M$ turn out to be a smooth manifold with a $C^{1,\alpha}$ metric $g$, and we can assume that $\psi_t$ is a $C^{2,\alpha}$ diffeomorphism and $\psi_t^\ast g_t\to g$ in $C^{1,\alpha}$ for some $\alpha\in(0,1)$, as we have seen in Remark \ref{RemA}.
By this and Lemma \ref{UnifConv}, we can take $\{\tau_t\}_{t\in\Z_{>0}}\subset \R_{>0}$ with $\lim_{t\to\infty} \tau_t=0$ such that the following holds for any $t\in\Z_{>0}$:
\begin{itemize}
\item[(i)] $\|\iota(x)-\iota_t(\psi_t(x))\|\leq \tau_t$ for any $x\in M$.
\item[(ii)] $d_{g_t}(x,y)\leq L\|\iota_t(x)-\iota_t(y)\|+\tau_t$ for any $x,y\in M_t$.
\item[(iii)] $(1+\tau_t)^{-1}g\leq \psi_t^\ast g_t\leq (1+\tau_t) g$.
\end{itemize}
The condition (iii) implies
\begin{equation}\label{taugrad}
(1+\tau_t)^{-1}|\nabla f|_{g}^2(x) \leq |\nabla (f\circ \psi_t^{-1})|_{g_t}^2(\psi_t(x))\leq (1+\tau_t) |\nabla f|_{g}^2(x)
\end{equation}
for any $x\in M$ and $f\in W^{1,2}(M)$,
\begin{equation}\label{taud}
(1+\tau_t)^{-1/2}d_g(x,y)\leq d_{g_t}(\psi_t(x),\psi_t(y))\leq (1+\tau_t)^{1/2}d_g(x,y)
\end{equation}
for any $x,y\in M$, and
\begin{equation}\label{tauVol}
(1+\tau_t)^{-m/2}(\psi_{t})_\ast \Vol_g \leq \Vol_{g_t}\leq (1+\tau_t)^{m/2} (\psi_{t})_\ast \Vol_g,
\end{equation}
where $(\psi_{t})_\ast \Vol_g$ denotes the push-forward measure defined by $\left((\psi_{t})_\ast \Vol_g\right)(A):=\Vol_g\left(\psi_t^{-1}(A)\right)$ for any Borel subset $A\subset M_t$.
For each $t\in \Z_{>0}$, define $\rho_t\colon M_t\to\R$ by
\begin{equation}
\rho_t (x):=\frac{\rho(\psi_t^{-1}(x))}{\int_{M_t} \rho\circ\psi_t^{-1}\, d\Vol_{g_t}} \quad(x\in M_t).
\end{equation}
We have 
\begin{equation}\label{taurho}
(1+\tau_t)^{-m/2}\leq \int_{M_t} \rho\circ\psi_t^{-1}\, d\Vol_{g_t}\leq (1+\tau_t)^{m/2}
\end{equation}
by (\ref{tauVol}), and so $(1+\tau_t)^{-m/2}\alpha^{-1}\leq \rho_t\leq (1+\tau_t)^{m/2}\alpha$ and
$\Lip(\rho_t)\leq (1+\tau_t)^{(m+1)/2}\Lip (\rho)$ by (\ref{taud}).
Thus, by (\ref{taugrad}), (\ref{tauVol}) and (\ref{taurho}) imply
\begin{align}
\label{taurhoVol}(1+\tau_t)^{-m}(\psi_{t})_\ast (\rho\Vol_g) \leq& \rho_t\Vol_{g_t}\leq (1+\tau_t)^{m} (\psi_{t})_\ast(\rho \Vol_g),\\
\label{taul2}(1+\tau_t)^{-(1+i)m/2}\int_M f^2\rho^i\,d\Vol_g\leq &\int_{M_t} (f\circ \psi_t^{-1})^2\rho_t^i\,d\Vol_{g_t}\\
\notag \leq& (1+\tau_t)^{(1+i)m/2}\int_M f^2\rho^i\,d\Vol_g
\end{align}
and
\begin{equation}\label{taudir}
\begin{split}
(1+\tau_t)^{-1-\left(2+\frac{i}{2}\right)m}\frac{\int_M|\nabla f|_{g}^2 \rho^2\,d\Vol_g}{\int_M f^2\rho^i\,d\Vol_g}
\leq& \frac{\int_{M_t}|\nabla (f\circ\psi_t^{-1})|_{g_t}^2 \rho_t^2\,d\Vol_{g_t}}{\int_{M_t} (f\circ \psi_t^{-1})^2\rho_t^i\,d\Vol_{g_t}}\\
\leq& (1+\tau_t)^{1+\left(2+\frac{i}{2}\right)m}\frac{\int_M|\nabla f|_{g}^2 \rho^2\,d\Vol_g}{\int_M f^2\rho^i\,d\Vol_g}
\end{split}
\end{equation}
for any $t\in\Z_{>0}$, $f\in W^{1,2}(M)$ and $i\in\{1,2\}$.
Then, (\ref{taudir}) immediately implies the following lemma.
\begin{Lem}\label{tevalmfd}
For any $t\in \Z_{>0}$ and $k\in \Z_{>0}$, we have
\begin{align*}
(1+\tau_t)^{-1-5m/2}\lambda_k(\Delta_{\rho})\leq&\lambda_k(\Delta_{\rho_t})\leq (1+\tau_t)^{1+5m/2}\lambda_k(\Delta_{\rho}),\\
(1+\tau_t)^{-1-3m}\lambda_k(\Delta_{\rho}^N)\leq&\lambda_k(\Delta_{\rho_t}^N)\leq (1+\tau_t)^{1+3m}\lambda_k(\Delta_{\rho}^N).
\end{align*}
\end{Lem}
We next give the convergence lemma for the eigenfunctions.
In this section, we consider the inner product $\langle\cdot,\cdot\rangle_\rho$ on $L^2(M)$ defined by $\langle f,h\rangle_\rho=\int_M f h\rho\,d\Vol_g$ for the unnormalized case, and the inner product $\langle\cdot,\cdot\rangle_{\rho^2}$ defined by $\langle f,h\rangle_{\rho^2}=\int_M f h\rho^2\,d\Vol_g$ for the normalized case.
Using $\rho_t$, we also define the inner products $\langle\cdot,\cdot\rangle_{\rho_t}$ and $\langle \cdot,\cdot\rangle_{\rho_t^2}$ on $L^2(M_t)$ similarly.
Let $\{f_i\}_{i=0}^\infty$ denotes the complete orthonormal system of $(L^2(M),\langle\cdot,\cdot\rangle_\rho)$ consisting of the eigenfunctions of $\Delta_\rho$ corresponding to the eigenvalues $\{\lambda_i(\Delta_\rho)\}_{i=0}^\infty$.
Similarly we define $\{f_i(t)\}_{t=0}^\infty$, $\{f_i^N\}_{i=0}^\infty$ and $\{f_i^N(t)\}_{i=0}^\infty$ for $\Delta_{\rho_t}$, $\Delta_\rho^N$ and $\Delta_{\rho_t}^N$, respectively.
\begin{Lem}\label{tLinfError}
Take arbitrary $k,l\in\Z_{>0}$ and suppose that
$\lambda_k(\Delta_\rho)>\lambda_{k-1}(\Delta_\rho)$ and $\lambda_{l+1}(\Delta_\rho)>\lambda_l(\Delta_\rho)$.
Let $p_t \colon L^2(M_t)\to \Span\{f_k(t),\ldots,f_l(t)\}$ be the orthogonal projection for each $t\in\Z_{>0}$.
Then, for any $\delta_1\in(0,\infty)$, there exists $N_1\in\Z_{>0}$ such that for any $t\in\Z_{>0}$ with $t\geq N_1$,
the map $ \Span\{f_k,\ldots,f_l\}\to \Span\{f_k(t),\ldots,f_l(t)\},\, f\mapsto p_t(f\circ \psi_t^{-1})$ is an isomorphism, and
for any $f\in \Span\{f_k,\ldots,f_l\}$, we have
$$
\sup_{x\in M}|f(x)-p_t(f\circ \psi_t^{-1})(\psi_t(x))|\leq \|f\circ \psi_t^{-1}\|_{\rho_t}\delta_1.
$$
\end{Lem}
\begin{proof}
Put $n_1:=\Card\{\lambda_k(\Delta_\rho),\ldots, \lambda_l(\Delta_\rho)\}$, and take real numbers $\nu_1,\ldots,\nu_{n_1}\in\R$ such that $\nu_1<\cdots<\nu_{n_1}$ and $\{\nu_1,\ldots,\nu_{n_1}\}=\{\lambda_k(\Delta_\rho),\ldots, \lambda_l(\Delta_\rho)\}$.
For each $t\in \Z_{>0}$ and $j\in\{1,\ldots,n_1\}$, let
$p_{t, j}\colon L^2(M_t)\to \Span\{f_i(t):\lambda_i(\Delta_\rho)=\nu_j\}$ denotes the orthogonal projection.
Then, we can find a sequence $\{\delta_t\}_{t\in \Z_{>0}}\subset \R_{>0}$ with $\lim_{t\to\infty}\delta_t=0$ such that for any $j\in \{1,\ldots,n_1\}$ and $h\in\Span\{f_i: \lambda_i(\Delta_\rho)=\nu_j\}$,
$$
\|h\circ \psi_t^{-1}-p_{t}(h\circ \psi_t^{-1})\|_{\rho_t}^2\leq \|h\circ \psi_t^{-1}-p_{t,j}(h\circ \psi_t^{-1})\|_{\rho_t}^2\leq \delta_t\|h\circ \psi_t^{-1}\|_{\rho_t}^2
$$
holds by  (\ref{taudir}) and Lemma \ref{eveccomp} putting $s=0$.
Take arbitrary $f=\sum_{j=1}^{n_1}h_j\in \Span\{f_k,\ldots, f_l\}$, where $h_j\in \Span\{f_i: \lambda_i(\Delta_\rho)=\nu_j\}$.
Then, using (\ref{taul2}), we get
\begin{align*}
\|f\circ \psi_t^{-1}-p_{t}(f\circ \psi_t^{-1})\|_{\rho_t}\leq &\delta_t^{1/2}\sum_{j=1}^{n_1}\|h_j\circ \psi_t^{-1}\|_{\rho_t}\\
\leq& \delta_t^{1/2}n_1^{1/2}(1+\tau_t)^m\|f\circ \psi_t^{-1}\|_{\rho_t}.
\end{align*}
For any $f\in \Span\{f_k,\ldots, f_l\}$ and $t\in\Z_{>0}$ with $\tau_t<1$, we have that $\Lip(f\circ \psi_t^{-1}-p_{t}(f\circ \psi_t^{-1}))$ is uniformly bounded:
$$
\Lip(f\circ \psi_t^{-1}-p_{t}(f\circ \psi_t^{-1}))\leq C(m,K,i_0,\alpha,l) \|f\circ \psi_t^{-1}\|_{\rho_t}
$$
by Lamma \ref{GradEF}, (\ref{taud}) and (\ref{taul2}), and so
we get the lemma by Theorem \ref{BishopGromov} (iii) (or Theorem \ref{VolLow}).
\end{proof}
We also have the corresponding result for the normalized case.
\begin{Lem}
Take arbitrary $k,l\in\Z_{>0}$ and suppose that
$\lambda_k(\Delta_\rho^N)>\lambda_{k-1}(\Delta_\rho^N)$ and $\lambda_{l+1}(\Delta_\rho^N)>\lambda_l(\Delta_\rho^N)$.
Let $p_t \colon L^2(M_t)\to \Span\{f_k^N(t),\ldots,f_l^N(t)\}$ be the orthogonal projection for each $t\in\Z_{>0}$.
Then, for any $\delta_1\in(0,\infty)$, there exists $N_1\in\Z_{>0}$ such that for any $t\in\Z_{>0}$ with $t\geq N_1$,
the map $ \Span\{f_k^N,\ldots,f_l^N\}\to \Span\{f_k^N(t),\ldots,f_l^N(t)\},\, f\mapsto p_t(f\circ \psi_t^{-1})$ is an isomorphism, and
for any $f\in \Span\{f_k,\ldots,f_l\}$, we have
$$
\sup_{x\in M}|f(x)-p_t(f\circ \psi_t^{-1})(\psi_t(x))|\leq \|f\circ \psi_t^{-1}\|_{\rho_t^2}\delta_1.
$$
\end{Lem}
Now move to the discrete setting.
Fix a sample size $n\in\Z_{>0}$.
Suppose that we are given $\epsilon\in(0,\infty)$.
Put
$$
\epsilon_t:=\epsilon-2\tau_t
$$
for each $t\in \Z_{>0}$.
If we are given points $x_1,\ldots,x_n\in M$, we define $\K,\K(t)\in \R^{n\times n}$ for $t\in\Z_{>0}$ with $\epsilon_t>0$ by
$$
\K_{i j}:=\eta\left(\frac{\|\iota(x_i)-\iota(x_j)\|}{\epsilon}\right),\quad \K(t)_{i j}:=\eta\left(\frac{\|\iota_t(\psi_t(x_i))-\iota_t(\psi_t(x_j))\|}{\epsilon_t}\right),
$$
and define $\D,\La,\D(t),\La(t)\in\R^{n\times n}$ as in Notation \ref{Not}.
For each $x_1,\ldots,x_n\in M$, we define
$$
W_\epsilon(x_1,\ldots,x_n):=\min\left\{\epsilon-\|\iota(x_i)-\iota(x_j)\|: i,j\in \{1,\ldots,n\}\text{ with }\|\iota(x_i)-\iota(x_j)\|<\epsilon\right\},
$$
and for each $\zeta\in(0,\epsilon)$,
$$
V_\zeta:=\left\{(x_1,\ldots,x_n)\in M\times \cdots\times M: W_\epsilon(x_1,\ldots,x_n)\geq \zeta\right\}.
$$
Then, the condition (i) immediately implies the following lemmas.
\begin{Lem}
Take arbitrary $t\in \Z_{>0}$ and $\zeta\in(0,\epsilon)$ with $\zeta>4\tau_t$.
Then, for any $(x_1,\ldots,x_n)\in V_{\zeta}$ and $i,j\in\{1\ldots,n\}$, the following two conditions are mutually equivalent.
\begin{itemize}
\item $\|\iota(x_i)-\iota(x_j)\|<\epsilon$,
\item $\|\iota_t(\psi_t(x_i))-\iota_t(\psi_t(x_j))\|<\epsilon_t$.
\end{itemize}
\end{Lem}

\begin{Lem}
Take arbitrary $t\in \Z_{>0}$ and $\zeta\in(0,\epsilon)$ with $\zeta>4\tau_t$.
Then, for any $(x_1,\ldots,x_n)\in V_{\zeta}$ and $i,j\in\{1\ldots,n\}$, we have
$$
\left|\K_{i j}-\K(t)_{i j}\right|\leq \frac{4 L_\eta \tau_t}{\epsilon}.
$$
\end{Lem}
This lemma shows that $\K(t),\D(t),\La(t)$ converge to $\K,\D,\La$ uniformly on $V_{\zeta}$, respectively. Thus, we immediately get the following lemma.
\begin{Lem}\label{tevalgrp}
Take arbitrary $n\in\Z_{>0}$, $\zeta\in(0,\epsilon)$ and $\delta_1\in(0,\infty)$.
Then, there exists $N_1\in\Z_{>0}$ such that for any $(x_1,\ldots,x_n)\in V_{\zeta}$, $t\in \Z_{>0}$ with $t\geq N_1$ and $k\in\Z_{>0}$ with $k\leq n-1$, we have
\begin{align*}
|\lambda_k(\La)-\lambda_k(\La(t))|\leq &\delta_1,\\
|\lambda_k(\La,\D)-\lambda_k(\La(t),\D(t))|\leq &\delta_1.
\end{align*}
\end{Lem}
We next give the convergence lemma for the eigenvectors.
Suppose that we are given points $x_1,\ldots,x_n\in M$ and $t\in\Z_{>0}$ with $\epsilon_t>0$.
In this section, we consider the inner product $\langle\cdot,\cdot\rangle$ on $\R^n$ defined by $\langle u,v\rangle=\sum_{i=1}^n u_i v_i/n$ for the unnormalized case.
For the normalized case, we define $\langle u,v\rangle_{\D}=\sum_{i=1}^n u_i v_i \D_{i i}/n$ and $\langle u, v\rangle_{\D(t)}= \sum_{i=1}^n u_i v_i \D(t)_{i i}/n$.
Let $\{u^i\}_{i=0}^{n-1}$ denotes the orthonormal basis of $(\R^n,\langle\cdot,\cdot\rangle)$ consisting of the eigenvectors of $\La$ corresponding to the eigenvalues $\{\lambda_i(\La)\}_{i=0}^{n-1}$.
Similarly we define $\{u^i(t)\}_{t=0}^{n-1}$, $\{u^{N,i}\}_{i=0}^{n-1}$ and $\{u^{N,i}(t)\}_{i=0}^{n-1}$ for $\La(t)$, $(\La,\D)$ and $(\La(t),\D(t))$, respectively.
\begin{Lem}\label{UNtinf}
Take arbitrary $n,k,l\in\Z_{>0}$ with $k\leq l\leq n-2$, $\Gamma\in (0,\infty)$, $\zeta\in(0,\epsilon)$ and $\delta_1\in(0,\infty)$.
Then, there exists $N_1\in\Z_{>0}$ such that for any $(x_1,\ldots,x_n)\in V_{\zeta}$ with $\lambda_k(\La)-\lambda_{k-1}(\La)\geq \Gamma$ and $\lambda_{l+1}(\La)-\lambda_l(\La)\geq \Gamma$, and $t\in \Z_{>0}$ with $t\geq N_1$, the following properties hold.
\begin{itemize}
\item[(i)] The map $\Span\{u^k(t),\ldots,u^l(t)\}\to\Span\{u^k,\ldots,u^l\},\,v\mapsto \Pj v$ is an isomorphism, where $\Pj\colon \R^n\to\Span\{u^k,\ldots,u^l\}$ denotes the orthogonal projection, and
for any $v\in \{u^k(t),\ldots,u^l(t)\}$, we have
$\| v-\Pj v\|\leq \delta_1\|v\|$.
\item[(ii)] The map $\Span\{u^k,\ldots,u^l\}\to\Span\{u^k(t),\ldots,u^l(t)\},\,u\mapsto \Pj_t u$ is an isomorphism, where $\Pj_t\colon \R^n\to\Span\{u^k(t),\ldots,u^l(t)\}$ denotes the orthogonal projection, and
for any $u\in \{u^k,\ldots,u^l\}$, we have
$\| u-\Pj_t u\|\leq \delta_1\|u\|$.
\item[(iii)] For any $u\in \R^n$, we have $\|(\Pj-\Pj_t)u\|\leq 2 \delta_1\|u\|$.
\end{itemize}
\end{Lem}
\begin{proof}
We first prove (i).
Since $\La(t)$ converges to $\La$ uniformly on $V_{\zeta}$, we can find a sequence $\{\delta_t\}_{t\in\Z_{>0}}\subset \R_{>0}$ with $\lim_{t\to\infty}\delta_t=0$ such that
$\|(\La-\La(t))u\|^2\leq \delta_t \|u\|^2$ holds for any $u\in \R^n$ and $(x_1,\ldots,x_n)\in V_{\zeta}$.
Moreover, there exists $N\in\Z_{>0}$ such that $|\lambda_k(\La)-\lambda_k(\La(t))|\leq \Gamma/2$ and $|\lambda_l(\La)-\lambda_l(\La(t))|\leq \Gamma/2$ hold for any $(x_1,\ldots,x_n)\in V_{\zeta}$ and $t\in\Z_{>0}$ with $t\geq N$.

Take arbitrary $(x_1,\ldots,x_n)\in V_{\zeta}$ with $\lambda_k(\La)-\lambda_{k-1}(\La)\geq \Gamma$ and $\lambda_{l+1}(\La)-\lambda_l(\La)\geq \Gamma$, and $t\in\Z_{>0}$ with $t\geq N$.
Then, we have $\lambda_k(\La(t))\geq \lambda_{k-1}(\La)+\Gamma/2$ and $\lambda_{l+1}(\La)\geq \lambda_l(\La(t))+\Gamma/2$.
Take arbitrary $i\in\{k,\ldots, l\}$ and expand $u^i(t)$ as
$
u^i(t)=\sum_{j=0}^{n-1}a_j(t) u^j.
$
Then, we have
\begin{align*}
\delta_t\geq \|\La u^i(t)-\La(t) u^i(t)\|^2=&\sum_{j=0}^\infty a_j(t)^2 \left(\lambda_j(\La)-\lambda_i(\La(t))\right)^2\\
\geq &\frac{\Gamma^2}{4}\sum_{j\notin\{k,\ldots,l\}}a_j(t)^2. 
\end{align*}
This implies 
$$\|u^i(t)-\Pj u^i(t)\|^2\leq \frac{4\delta_t}{\Gamma^2}.
$$
Take arbitrary $v=\sum_{i=k}^l a_i u^i(t)\in \Span\{u^k(t),\ldots,u^l(t)\}$.
Then, we have
\begin{align*}
\|v-\Pj v\|\leq \sum_{i=k}^l|a_i|\|u^i(t)-\Pj u^j(t)\|\leq (l-k+1)^{1/2}\frac{2\delta_t^{1/2}}{\Gamma}\|v\|.
\end{align*}
Thus, we get (i).

We can prove (ii) similarly to (i).

Finally, we prove (iii).
Take arbitrary $u\in \R^n$.
Then, by (ii), we have
$$\|\Pj u-\Pj_t\Pj u\|\leq \delta_1\|\Pj u\|\leq \delta_1\|u\|.$$
By (i), we have
\begin{align*}
\|\Pj_t\Pj u-\Pj_t u\|^2=&\left|\langle\Pj_t\Pj u,\Pj_t(\Pj u-u)\rangle-\langle\Pj_t u,\Pj_t(\Pj u-u)\rangle\right|\\
=&\left|\langle u, \Pj \Pj_t(\Pj u-u)-\Pj_t(\Pj u-u)\rangle\right|\\
\leq &\|u\|\|\Pj \Pj_t(\Pj u-u)-\Pj_t(\Pj u-u)\|\\
\leq& \delta_1\|u\|\|\Pj_t\Pj u-\Pj_t u\|,
\end{align*}
and so $\|\Pj_t\Pj u-\Pj_t u\|\leq \delta_1\|u\|$.
Thus, we get (iii).
\end{proof}
Combining the fact that $\La u=\lambda \D u$ if and only if $\D^{-1/2} \La\D^{-1/2}\left(\D^{1/2} u\right)=\lambda \left(\D^{1/2}u\right)$, and that $\D(t)^{-1/2} \La(t)\D(t)^{-1/2}$ and $\D(t)$ converge to $\D^{-1/2} \La\D^{-1/2}$ and $\D$ uniformly on $V_{\zeta}$ respectively, we get the following similarly to Lemma \ref{UNtinf}.
Note that we have $\D_{ii}\geq \eta(0)$ for any $i\in\{1,\ldots,n\}$.
\begin{Lem}
Take arbitrary $n,k,l\in\Z_{>0}$ with $k\leq l\leq n-2$, $\Gamma\in (0,\infty)$, $\zeta\in(0,\epsilon)$ and $\delta_1\in(0,\infty)$.
Then, there exists $N_1\in\Z_{>0}$ such that for any $(x_1,\ldots,x_n)\in V_{\zeta}$ with  $\lambda_k(\La,\D)-\lambda_{k-1}(\La,\D)\geq \Gamma$, $\lambda_{l+1}(\La,\D)-\lambda_l(\La,\D)\geq \Gamma$, and $t\in \Z_{>0}$ with $t\geq N_1$, the following properties hold.
\begin{itemize}
\item[(i)] The map $\Span\{u^k(t),\ldots,u^l(t)\}\to\Span\{u^k,\ldots,u^l\},\,v\mapsto \Pj v$ is an isomorphism, where $\Pj\colon \R^n\to\Span\{u^k,\ldots,u^l\}$ denotes the orthogonal projection, and
for any $v\in \{u^k(t),\ldots,u^l(t)\}$, we have
$\left\| v-\Pj v\right\|_{\D}\leq \delta_1\|v\|_{\D}.$
\item[(ii)] The map $\Span\{u^k,\ldots,u^l\}\to\Span\{u^k(t),\ldots,u^l(t)\},\,u\mapsto \Pj_t u$ is an isomorphism, where $\Pj_t\colon \R^n\to\Span\{u^k(t),\ldots,u^l(t)\}$ denotes the orthogonal projection, and
for any $u\in \{u^k,\ldots,u^l\}$, we have
$\left\| u-\Pj_t u\right\|_{\D}\leq \delta_1\|u\|_{\D}.$
\item[(iii)] For any $u\in \R^n$, we have $\left\|\Pj u-\Pj_t u\right\|_{\D}\leq \delta_1\|u\|_{\D}.$
\item[(iv)] For any $u\in \R^n$, we have $\left|\|u\|_{\D}-\|u\|_{\D(t)}\right|\leq \delta_1\|u\|_{\D}$.
\end{itemize}
\end{Lem}

Now let us prove Theorem \ref{UNMainEval}.
\begin{proof}[Proof of Theorem \ref{UNMainEval}]
Let $C_1=C_1(m,\alpha+1,k)$ and $C_2=C_2(m,S,K,i_0,L,\eta,\alpha+1,L_\rho+1,k)$ be constants appearing in Theorem \ref{eigmain}.
Suppose that
$$
\epsilon+\gamma^{1/2} \epsilon^{-m/2}\left(\frac{\log n}{n}\right)^{1/2}< C_2^{-1}.
$$
Then, there exists $N\in \Z_{>0}$ such that for any $t\in\Z_{>0}$ with $t\geq N$ and i.i.d. sample $x_1,\ldots,x_n\in M$ from $\rho\Vol_g$, 
$$
\left|\lambda_k(\Delta_{\rho_t})-\frac{2}{\sigma_\eta n \epsilon_t^{m+2}}\lambda_k(\La(t))\right|\leq C_2\left(\epsilon+\frac{\tau_t}{\epsilon}+\gamma^{1/2} \epsilon^{-m/2}\left(\frac{\log n}{n}\right)^{1/2}\right)
$$
with probability at least $1-C_1(1+\tau_t)^{m n}(\epsilon_t^{-2m-1}+n^2)n^{-\gamma}$ by Theorem \ref{eigmain} and (\ref{taurhoVol}).

Take arbitrary $\delta_1\in(0,\infty)$.
Then, there exists $\zeta\in(0,\epsilon)$ such that
$(\rho\Vol_g)(V_\zeta)\leq \delta_1$.
By Lemma \ref{tevalmfd} and \ref{tevalgrp}, we can find $N_1\in\Z_{>0}$ with $N_1\geq N$ such that for any $t\in\Z_{>0}$ with $t\geq N_1$, we have
$$
\left|\left(\lambda_k(\Delta_{\rho})-\frac{2}{\sigma_\eta n \epsilon^{m+2}}\lambda_k(\La)\right)-\left(\lambda_k(\Delta_{\rho_t})-\frac{2}{\sigma_\eta n \epsilon_t^{m+2}}\lambda_k(\La(t))\right)\right|\leq \delta_1.
$$
Thus, we get that for any $t\in\Z_{>0}$ with $t\geq N_1$, $\delta_1\in (0,\infty)$ and i.i.d. sample $x_1,\ldots,x_n\in M$ from $\rho\Vol_g$,
$$
\left|\lambda_k(\Delta_{\rho})-\frac{2}{\sigma_\eta n \epsilon^{m+2}}\lambda_k(\La)\right|\leq C_2\left(\epsilon+\frac{\tau_t}{\epsilon}+\gamma^{1/2} \epsilon^{-m/2}\left(\frac{\log n}{n}\right)^{1/2}\right)+\delta_1
$$
holds with probability at least $1-C_1(1+\tau_t)^{m n}(\epsilon_t^{-2m-1}+n^2)n^{-\gamma}-\delta_1.$
This implies the theorem.
\end{proof}
Similarly, we get Theorem \ref{NMainEval}.
\begin{proof}[Proof of Theorem \ref{UNMainEvec}]
Let $C_1=C_1(m,\alpha+1,l)$ and $C_2=C_2(m,S,K,i_0,L,\eta,\alpha+1,L_\rho+1,l)$ be constants appearing in Theorem \ref{evecmain}.

For each $t\in \Z_{>0}$, let $p_t\colon L^2(M_t)\to\Span\{f_k(t),\ldots,f_l(t)\}$, $\Pj_t\colon \R^n\to \Span\{u^k(t),\ldots, u^l(t)\}$ and $\Pj \colon \R^n \to\Span \{u^k,\ldots,u^l\}$ be the orthogonal projections.
For each $t\in\Z_{>0}$, put $s_t:=\lambda_l(\Delta_{\rho_t})-\lambda_k(\Delta_{\rho_t})$ and
$
G_t:=\min\{|\lambda_k(\Delta_{\rho_t})-\lambda_{k-1}(\Delta_{\rho_t})|,|\lambda_{l+1}(\Delta_{\rho_t})-\lambda_{l}(\Delta_{\rho_t})|\}.
$
Suppose that $n\in\Z_{>0}$, $\epsilon\in(0,D]$ and $\gamma\in (1,\infty)$ satisfies
$$
C_2\left(\epsilon+\gamma^{1/2} \epsilon^{-m/2}\left((\log n)/n\right)^{1/2}\right)+2 G s < G^2.
$$
Then, there exists $N\in\Z_{>0}$ such that for any $t\in\Z_{>0}$ with $t\geq N$, we have
$$
C_2\left(\epsilon_t+\tau_t/\epsilon_t+\gamma^{1/2} \epsilon^{-m/2}\left((\log n)/n\right)^{1/2}\right)+2 G_t s_t < G_t^2.
$$
Thus, we can apply Theorem \ref{evecmain}, and so
for any $t\geq N$ and i.i.d. sample $x_1,\ldots,x_n\in M$ from $\rho\Vol_g$, 
putting
$$
A_t:=\frac{C_2}{G_t^2}\left(\epsilon_t+\frac{\tau_t}{\epsilon_t}+\gamma^{1/2}\epsilon_t^{-m/2}\left(\frac{\log n}{n}\right)^{1/2}\right)+\frac{2s_t}{G_t},
$$
we have
\begin{equation}\label{evecm1}
\left(1-A_t\right)\|p_t(f\circ\psi_t^{-1})|_{\psi_t(\X)}\|^2
\leq\left\|\Pj_t \left(p_t(f\circ\psi_t^{-1})|_{\psi_t(\X)}\right)\right\|^2,
\end{equation}
\begin{equation}\label{evecm2}
\left|\|p_t(f\circ\psi_t^{-1})|_{\psi_t(\X)}\|^2-\|p_t(f\circ\psi_t^{-1})\|_{\rho_t}^2 \right|\leq C_2\gamma^{1/2}\left(\frac{\log n}{n}\right)^{1/2}\|p_t(f\circ\psi_t^{-1})\|_{\rho_t}^2
\end{equation}
and
\begin{equation}\label{evecm3}
\left\|p_t(f\circ\psi_t^{-1})|_{\psi_t(\X)}-\Pj_t (p_t(f\circ\psi_t^{-1})|_{\psi_t(\X)})\right\|^2
\leq A_t \left\|p_t(f\circ\psi_t^{-1})|_{\psi_t(\X)}\right\|^2.
\end{equation}
holds for every $f\in\Span\{f_k,\ldots,f_l\}$ with probability at least $1-C_1(1+\tau_t)^{m n}(\epsilon_t^{-2m-1}+n^2)n^{-\gamma}$.
In particular, we can assume that
\begin{equation}\label{evecm4}
\frac{1}{2}\|p_t(f\circ\psi_t^{-1})\|_{\rho_t}
\leq
\|p_t(f\circ\psi_t^{-1})|_{\psi_t(\X)}\|\leq \frac{3}{2}\|p_t(f\circ\psi_t^{-1})\|_{\rho_t}.
\end{equation}

Take arbitrary $\delta_1\in(0,1/3)$.
Then, there exists $\zeta\in(0,\epsilon)$ such that
$(\rho\Vol_g)(V_\zeta)\leq \delta_1$.
We can take $N_1\in \Z_{>0}$ with $N_1\geq N$ such that for any $f\in\Span\{f_k,\ldots,f_l\}$, $t\in \Z_{>0}$ with $t\geq N$ and $(x_1,\ldots,x_n)\in V_\zeta$ with $\lambda_{l+1}(\La)-\lambda_l(\La)\geq G/2$ and $\lambda_k(\La)-\lambda_{k-1}(\La)\geq G/2$, we have
\begin{equation}\label{evecm5}
\sup_{x\in M}|f(x)-p_t(f\circ \psi_t^{-1})(\psi_t(x))|\leq \|f\circ\psi_t^{-1}\|_{\rho_t}\delta_1,
\end{equation}
\begin{equation}\label{evecm6}
\begin{split}
&\left\|
\Pj\left(\Pj_t(p_t(f\circ \psi_t^{-1})|_{\psi_t(\X)})\right)-\Pj_t(p_t(f\circ \psi_t^{-1})|_{\psi_t(\X)})
\right\|\\
\leq &\left\|\Pj_t(p_t(f\circ \psi_t^{-1})|_{\psi_t(\X)})
\right\|\delta_1\leq \left\|p_t(f\circ \psi_t^{-1})|_{\psi_t(\X)}\right\|\delta_1.
\end{split}
\end{equation}
and
\begin{equation}\label{evecm7}
\|\Pj u-\Pj_t u\|\leq 2\|u\|\delta_1
\end{equation}
for any $u\in\R^n$ by Lemma \ref{tLinfError}, \ref{UNtinf} (i) and (iii).
Take such $f$, $t$ and $(x_1,\ldots,x_n)$, and assume that (\ref{evecm1})--(\ref{evecm4}) holds.
Note that we can assume that the assumptions $\lambda_{l+1}(\La)-\lambda_l(\La)\geq G/2$ and $\lambda_k(\La)-\lambda_{k-1}(\La)\geq G/2$ hold with probability at least $1-C(m,\alpha,l)(\epsilon^{-2m-1}+n^2)n^{-\gamma}$ by Theorem \ref{UNMainEval}.
Then, (\ref{evecm5}) implies
\begin{equation}\label{evecm8}
(1-\delta_1)\|f\circ\psi_t^{-1}\|_{\rho_t}\leq \|p_t(f\circ \psi_t^{-1})\|_{\rho_t}\leq \|f\circ\psi_t^{-1}\|_{\rho_t}.
\end{equation}
We get
\begin{equation}\label{evecm9}
(1+\tau_t)^{-m/2}\|p_t(f\circ \psi_t^{-1})\|_{\rho_t}\leq \|f\|_{\rho}\leq \frac{(1+\tau_t)^{m/2}}{1-\delta_1 }\|p_t(f\circ \psi_t^{-1})\|_{\rho_t}.
\end{equation}
by (\ref{taul2}) and (\ref{evecm8}). 
We get
\begin{equation}\label{evecm10}
 \left(1-\frac{2\delta_1}{1-\delta_1}\right)\left\|p_t(f\circ \psi_t^{-1})|_{\psi_t(\X)}\right\|\leq \left\|f|_{\X}\right\|\leq \left(1+\frac{2\delta_1}{1-\delta_1}\right)\left\|p_t(f\circ \psi_t^{-1})|_{\psi_t(\X)}\right\|
\end{equation}
by (\ref{evecm4}), (\ref{evecm5}) and (\ref{evecm8}).
We get
\begin{equation}\label{evecm11}
\begin{split}
&\left\|\Pj(f|_{\X})-\Pj_t\left(p_t(f\circ \psi_t^{-1})|_{\psi_t(\X)}\right)\right\|\\
\leq&2\|f|_{\X}\|\delta_1+\left\|f|_{\X}-p_t(f\circ\psi^{-1})|_{\psi_t(\X)}\right\| \delta_1
\leq \frac{4\delta_1}{1-3\delta_1}\|f|_{\X}\|
\end{split}
\end{equation}
by (\ref{evecm4}), (\ref{evecm5}), (\ref{evecm7}), (\ref{evecm8}) and (\ref{evecm10}).
We have
\begin{align*}
&\left\|f|_{\X}-\Pj_t\left(p_t(f\circ \psi_t^{-1})|_{\psi_t(\X)}\right)
\right\|\\
\leq &\left\|f|_{\X}-p_t(f\circ \psi_t^{-1})|_{\psi_t(\X)}\right\|+\left\|p_t(f\circ \psi_t^{-1})|_{\psi_t(\X)}-\Pj_t\left(p_t(f\circ \psi_t^{-1})|_{\psi_t(\X)}\right)\right\|\\
\leq &\|f\circ \psi_t^{-1}\|_{\rho_t}\delta_1+ \left\|p_t(f\circ \psi_t^{-1})|_{\psi_t(\X)}\right\| A_t^{1/2}
\leq \frac{2\delta_1+A_t^{1/2}}{1-3\delta_1}\|f|_{\X}\|
\end{align*}
by (\ref{evecm3}),  (\ref{evecm4}), (\ref{evecm5}), (\ref{evecm8}) and (\ref{evecm10}).
Thus, we get
\begin{equation*}\label{evecm12}
\begin{split}
&\left\|f|_{\X}-\Pj(f|_{\X})\right\|\\
\leq &\left\|(1-\Pj)\left(f|_{\X}-\Pj_t\left(p_t(f\circ \psi_t^{-1})|_{\psi_t(\X)}\right)\right)
\right\|+\left\|(1-\Pj)\left(\Pj_t\left(p_t(f\circ \psi_t^{-1})|_{\psi_t(\X)}\right)\right)
\right\|\\
\leq &\frac{3\delta_1+A_t^{1/2}}{1-3\delta_1}\|f|_{\X}\|
\end{split}
\end{equation*}
by (\ref{evecm6}) and (\ref{evecm10}).
This is a result corresponding to (\ref{evecm3}) with $p_t(f\circ \psi_t^{-1})$ replaced by $f$.
By (\ref{evecm1}), (\ref{evecm10}) and (\ref{evecm11}), we get a result corresponding to (\ref{evecm1}) for $f$.
By (\ref{evecm2}), (\ref{evecm9}) and (\ref{evecm10}), we get a result corresponding to (\ref{evecm2}) for $f$.
Thus, we get the theorem.
\end{proof}
Similarly, we get Theorem \ref{NMainEvec}.

\section{Assumption about Reach}\label{Reach}
In this section, we assume that $M\subset \R^d$ be a closed submanifold with bounded reach, which is defined as follows:
\begin{Def}\label{DR}
We define the $\Reach(M)$ of $M$ by
$$
\Reach(M):=\inf\left\{d_{\R^d}(x, M):\begin{array}{l}x\in \R^d\text{ is a point such that there exist }p,q\in M\\
 \text{ with } p\neq q \text{ and } d_{\R^d}(x,M)=\|x-p\|=\|x-q\| \end{array}\right\},
$$
where $d_{\R^d}(x,M)$ denotes the Euclidean distance between $x$ and $M$.
\end{Def}
\begin{Thm}\label{CR}
If the reach of $M$ satisfies $\Reach(M)\geq R$ for some constant $R>0$, we have
\begin{align*}
\|II\|_{L^\infty}\leq& 1/R,\\
|\Sect_M|\leq& 1/R^2,\\
\Inj_M\geq& \pi R.
\end{align*}
\end{Thm}
The first assertion is shown in \cite[Proposition 6.1]{NSW}.
The second assertion is easily seen from the Gauss equation \cite[Theorem 3.2.4]{Pe3}, and the third from the Klingenberg theorem \cite[Lemma 6.4.7]{Pe3} and the Fenchel theorem.

When the second fundamental form is bounded, we obtain the following sharp comparison of the Riemannian and Euclidean distances under the assumption that the Riemannian distance is not large.
\begin{Prop}[Lemma 3 of \cite{BdSLT}]\label{bdslt}
If $\|II\|_{L^\infty}\leq 1/R$,
we have
$$
\|x-y\|\geq 2 R\sin \frac{d(x,y)}{2R}
$$
for any $x,y\in M$ with $d(x,y)\leq \pi R$.
\end{Prop}
In \cite{BdSLT}, Proposition \ref{bdslt} is proved geometrically, but it can be proved using calculus, as suggested in \cite{BdSLT}.
Indeed, Proposition \ref{bdslt} is an immediate consequence of the following differential inequality.
\begin{Lem}\label{difineq}
Take a real number $a\in(0,\pi/2]$ and smooth function $f\colon (-a,a)\to \R$ satisfying\begin{empheq}[left={\empheqlbrace}]{align*}
f^2+ (f')^2&\leq 1,\\
f(0)&=1,\\
f'(0)&=0.
\end{empheq}
Then, for any $t\in (-a,a)$, we have
\begin{align*}
f(t)\geq &\cos t,\\
|f'(t)|\leq &\left|\sin t\right|.
\end{align*}
\end{Lem}
Lemma \ref{difineq} is an elementary result on differential inequalities, so we skip the proof.

To apply Proposition \ref{bdslt} to points $x,y\in M$, we need to assume $d(x,y)\leq \pi R$, so Proposition \ref{bdslt} does not lead to the assertion that if $\|x-y\|$ is small, then $d(x,y)$ is also small.
Proposition 6.3 in \cite{NSW} leads to such an assertion assuming $\Reach(M)\geq R$, but its proof does not overcome this point and does not seem to be complete (the assertion itself is correct).
In \cite{BdSLT}, the following quantity is introduced to overcome this point:
$$
s(R):= \sup\left\{s\in\R_{>0}: \text{$\|x-y\|<s$ implies $d(x,y)\leq \pi R$ for any $x,y\in M$}\right\}
$$
for each $R\in(0,\infty)$.
By the definition of $s(R)$, the expansion of the sine function and Proposition \ref{bdslt}, we get the following:
\begin{Prop}[Corollary 4 of \cite{BdSLT}]\label{corbdslt}
If $\|II\|_{L^\infty}\leq 1/R$,
we have
$$
\|x-y\|\geq \left(1-\frac{\pi^2}{96 R^2}\|x-y\|^2\right)d(x,y)
$$
for any $x,y\in M$ with $\|x-y\|< s(R)$.
\end{Prop}
The assumptions in \cite{BdSLT} are essentially equivalent to the lower bound of the reach in the following sense.
\begin{Prop}\label{rtosr}
If $\Reach(M)\geq R$, then $\|II\|_{L^\infty}\leq 1/R$ and $s(R)>R/2$.
\end{Prop}
\begin{Prop}\label{srtor}
If $\|II\|_{L^\infty}\leq 1/R$, then $\Reach(M)\geq\min\{s(R)/2,(1-\pi/4)R\}$.
\end{Prop}
We do not give the proofs of these propositions, since they are outside the scope of this work.
Proposition \ref{corbdslt} and \ref{rtosr} imply the following.
\begin{Cor}\label{compreach}
If $\Reach(M)\geq R$, we have
$$
d(x,y)\leq  \left(1+\frac{\pi^2}{48 R^2}\|x-y\|^2\right)\|x-y\|
$$
for any $x,y\in M$ with $\|x-y\|\leq R/2$.
\end{Cor}
It is easy to give a better estimate than this corollary, but we do not know the sharp inequalities in Propositions \ref{rtosr}, \ref{srtor} and Corollary \ref{compreach}.

\section{Hausdorff Measures and its Coincidence}\label{HausdorffMeasure}
Under Assumption \ref{Asu}, $M$ has two different distance functions.
One is a given distance $d_M$, and the other is a distance determined as a subset of $\R^d$.
In this section, we see that the $m$-dimensional Hausdorff measures determined by these two distance functions coincide with each other.
For smooth submanifolds, it is easy to show that this coincidence holds for the Riemannian distance and the distance as a subset.
However, under Assumption \ref{Asu}, $M$ can have even dense singularities, so this coincidence is nontrivial.
The goal of this section is to show that this coincidence holds even under weaker assumptions, as follows.

\begin{Thm}\label{HAP}
Let $m,d\in \Z_{>0}$ be integers with $m<d$ and take constants $S,K>0$.
Suppose that we are given a compact metric space $M$ with distance function $d_M$ and an injective map $\iota\colon M\to \R^d$ such that there exist a sequence $\{(M_i,g_i)\}_{i=1}^\infty\subset \M_2(m,K)$, a sequence of positive real numbers $\{\epsilon_i\}_{i=1}^\infty\subset \R_{>0}$ with $\lim_{i\to\infty}\epsilon_i=0$, a sequence of isometric immersions $\{\iota_i\colon M_i\to \R^d\}_{i=1}^\infty$ and a sequence of maps
$\{\psi_i\colon M\to M_i\}_{i=1}^\infty$ satisfying the following properties:
\begin{itemize}
\item[(i)] For any $i\in \Z_{>0}$, we have
$
\int_{M_i} |II_i|\,d\Vol_{g_i} \leq S.
$
\item[(ii)] For any $i\in\Z_{>0}$ and $x,y\in M$, we have
$
|d_M(x,y)-d_{g_i}(\psi_i(x),\psi_i(y))|\leq \epsilon_i.
$
\item[(iii)] For any $i\in\Z_{>0}$ and $y\in M_i$, there exists $x\in M$ such that
$
d_{g_i}(y,\psi_i(x))\leq \epsilon_i.
$
\item[(iv)] For any $x\in M$, we have
$
\lim_{i\to \infty}\iota_i(\psi_i(x))=\iota(x).
$
\end{itemize}
Let $\Ha^m_M$ be the $m$-dimensional Hausdorff measure determined from the given distance $d_M$ of $M$, and let $\Ha^m_{\R^d}$ be the $m$-dimensional Hausdorff measure of $\R^d$.
Then, for any Borel subset $A$ of $M$, we have
$$
\Ha^m_M(A)=\Ha^m_{\R^d}(\iota(A)).
$$
\end{Thm}
\begin{proof}
We have $\Ha^m_M(A)\geq \Ha^m_{\R^d}(\iota(A))$ for any Borel subset $A$ of $M$.
Let us show the opposite direction.
We can assume that there exists $v_0>0$ such that $\Vol_{g_i}(M_i) \geq v_0$ holds for any $i\in \Z_{>0}$, otherwise we have $\Ha^m_M(A)=\Ha^m_{\R^d}(\iota(A))=0$  for any Borel subset $A$ of $M$.
The compactness of $M$ and $\lim_{i\to\infty}\diam_{g_i}(M_i)=\diam(M)$ imply that there exists a constant $D>0$ such that $\diam_{g_i}(M_i)\leq D$ hold for any $i\in\Z_{>0}$.

In our proof, we use the spherical Hausdorff measure instead the usual Hausdorff measure.
For any Borel subset $A\subset M$ and $\delta\in(0,\infty)$, we define
\begin{align*}
\Ha_{\R^d,\delta}^m(\iota(A))
:=\inf\left\{\sum_{j=1}^\infty \omega_m r^j: x_j\in M,\, r_j\in (0,\delta]\text{ and } \iota(A)\subset \bigcup_{j=1}^\infty B^{\R^d}(\iota(x_j),r_j)\right\},
\end{align*}
where $\omega_m:=\Vol(B^{\R^m}(0,1))$.
Note that we can take the centers of balls on $\iota(M)$.
Its limit as $\delta\to 0$ is called the spherical Hausdorff measure.
Using closed balls instead of open balls, we get the same measure.
Our definition is slightly different from that of the spherical measure of \cite[p.171]{Federer}, and our measure is a priori greater than or equal to the spherical measure.
Since we have $\Vol_{g_i}(B_r^{M_i}(\psi_i(x)))\to \Ha^m(B_r^M(x))$ as $i\to \infty$ for any $x\in M$ and $r\in(0,\infty)$ by \cite[Theorem 5.9]{CC1}, we have that $M$ is $(\Ha^m_M,m)$-rectifiable in the sense of \cite[p.251]{Federer} by \cite[Theorem 5.5]{CC3}.
Then, $\iota(M)\subset \R^d$ is also $(\Ha^m_{\R^d},m)$-rectifiable by the Lipschitz continuity of $\iota$. 
Thus, by \cite[Theorem 3.2.19]{Federer}, the $m$-dimensional density of $(\iota(M),\Ha^m_{\R^d})$ is equal to $1$, and so the Hausdorff measure and the spherical Hausdorff measure coincide, which means $\lim_{\delta\to 0}\Ha_{\R^d,\delta}^m(\iota(A))=\Ha^n_{\R^d}(\iota(A))$ for any Borel subset $A\subset M$.
The proof of the last step uses a standard argument based on the Vitali covering theorem. See the proof of \cite[Theorem 3.2 (1)]{Simon} and \cite[Theorem 3.5]{Simon}.
A similar statement holds for $(M,\Ha^m_M)$ (see \cite[Remark 10.17]{Ch}), but we do not use this fact in our proof.

Since $M$ is compact, the map $\iota\colon M\to \iota(M)\subset \R^d$ is a homeomorphism.
Combining this with Lemma \ref{UnifConv} (b), we immediately get the following claim .
\begin{Clm}\label{cl1}
For any $\epsilon>0$, there exists $N\in\Z_{>0}$ and $\delta>0$ such that for any $n\in\Z_{>0}$ with $i\geq N$ and for any $x,y\in M_i$, if $\|\iota_i(x)-\iota_i(y)\|\leq \delta$, then $d_{g_i}(x,y)\leq \epsilon$.
\end{Clm}

Take arbitrary $x\in M$ and $r>0$, and put
$$A:=\overline B^M(x,r)\subset M,\quad  A_i:=\overline B^{M_i}\left(\psi_i(x),r\right)\subset M_i.$$
We immediately get the following.
\begin{Clm}\label{cl2}
For any $y\in A_i$, there exists $z\in A$ such that $d(\psi_i(z),y)\leq 4\epsilon_i$.
\end{Clm}
\begin{Clm}\label{cl3}
For any $\epsilon>0$ and $\delta>0$, there exist $N_0\in \Z_{>0}$, $x_1,\ldots,x_{N_0}\in M$ and $r_1,\ldots,r_{N_0}\in (0,\delta]$ such that the following.
\begin{itemize}
\item[(i)] $\iota(A)\subset \bigcup_{j=1}^{N_0} B^{\R^d}(\iota(x_j),r_j)$,
\item[(ii)] $\sum_{j=1}^{N_0}\omega_m r_j^m\leq H_{\R^d,\delta}^m(\iota(A))+\epsilon$.
\item[(iii)] For any $\epsilon_1>0$, there exists $N_1\in \Z_{>0}$ such that for any $i\in\Z_{>0}$ with $i\geq N_1$, we have
$\iota_i(A_i)\subset \bigcup_{j=1}^{N_0} B^{\R^d}(\iota_i(\psi_i(x_j)),r_j+\epsilon_1).$
\end{itemize}
\end{Clm}
\begin{proof}[Proof of Claim \ref{cl3}]
We get (i) and (ii) by the definition of $H_{\R^d,\delta}^m$ and the compactness of $\iota(A)$.
We get (iii) by Lemma \ref{UnifConv} (b) and Claim \ref{cl2}.
\end{proof}
For each $i\in \Z_{>0}$, $\xi_1,\xi_2>0$, $x\in M_i$ and $u\in U_x M_i$, we define
\begin{align*}
S_{u,\xi_1}^i:=&\int_0^{\xi_1}|II_i|\circ \gamma_u(t)\,d t,\\
B_{\xi_2}^i:=&\left\{y\in M_i:\int_{U_y M_i}S_{u,\xi_1}^i\,d u\leq \xi_2\right\}.
\end{align*}
\begin{Clm}\label{cl4}
For any $i\in\Z_{>0}$ and  $\xi_1,\xi_2>0$, we get $\Vol\left(M_i\setminus B_{\xi_2}^i\right)\leq S\Vol(S^{m-1})\xi_1/\xi_2$.
\end{Clm}
\begin{proof}[Proof of Claim \ref{cl4}]
Since we have
$$
\int_{M_i}\int_{U_x M_i}S_{u,\xi_1}^i\, d u\, d x=\Vol(S^{m-1})\xi_1\int_{M_i} |II_i|\,d \Vol_{g_i}\leq S\Vol(S^{m-1})\xi_1
$$
by (\ref{GeodesicFlow}), we get the claim by the Chebyshev inequality.
\end{proof}
\begin{Clm}\label{cl5}
There exists a constant $C=C(m,K,D)>0$ such that
for any $i\in\Z_{>0}$, $\xi_1,\xi_2>0$, $x\in B_{\xi_2}^i$ and $r\in (0,\xi_1]$, we have
$$
\int_{B^{M_i}(x,r)}\frac{d(x,y)-\|\iota_i(x)-\iota_i(y)\|}{d(x,y)}\,d y\leq Cr^m \xi_2.
$$
\end{Clm}
\begin{proof}
For any $t\in(0,\xi_1]$ and $u\in U_x M_i$, we have
$$
\int_0^t |II_i|\circ \gamma_u(s)\,d s\leq S_{u,\xi_1}^i,
$$
and so we get
$$
\frac{t-\|\iota_i(\gamma_u(t))-\iota_i(x)\|}{t}\leq \frac{S_{u,\xi_1}^i}{\sqrt 2}
$$
by Lemma \ref{fund} (ii).
Thus, by Theorem \ref{BishopGromov} (ii), we get
\begin{align*}
\int_{B^{M_i}(x,r)}\frac{d(x,y)-\|\iota_i(x)-\iota_i(y)\|}{d(x,y)}\,d y
\leq &C\int_{U_x M_i}\int_0^r S_{u,\xi_1}^i t^{m-1}\,d t\, d u\\
= &\frac{C}{m}r^m\int_{U_x M_i}S_{u,\xi_1}^i\, d u
\leq C r^m \xi_2.
\end{align*}
This implies the claim.
\end{proof}
\begin{Clm}\label{cl6}
There exists a constant $C_1=C_1(m,K,D,v_0)>0$ such that the following holds.
Take arbitrary $i\in\Z_{>0}$, $\xi_1,\xi_2,\xi_3>0$, $r>0$ and $x,y\in M_i$, and suppose that
\begin{itemize}
\item[(a)] $x\in B_{\xi_2}^i$,
\item[(b)] $d_{g_i}(x,y)+r\leq \xi_1$,
\item[(c)] $d_{g_i}(x,z)-\|\iota_i(x)-\iota_i(z)\|\geq \xi_3 d_{g_i}(x,z)$ for any $z\in B^{M_i}(y,r)$,
\end{itemize}
then we have
$$
r\leq C_1\left(\frac{\xi_2}{\xi_3}\right)^{1/m}(d_{g_i}(x,y)+r).
$$
\end{Clm}
\begin{proof}[Proof of Claim \ref{cl6}]
We have $B^{M_i}(y,r)\subset B^{M_i}(x,d_{g_i}(x,y)+r)$.
Thus, by the assumptions and Claim \ref{cl5}, we get
\begin{align*}
\xi_3\Vol(B^{M_i}(y,r))\leq &\int_{B^{M_i}(y,r)}\frac{d_{g_i}(x,z)-\|\iota_i(x)-\iota_i(z)\|}{d_{g_i}(x,z)}\,d z\\
\leq& \int_{B^{M_i}(x,d_{g_i}(x,y)+r)}\frac{d_{g_i}(x,z)-\|\iota_i(x)-\iota_i(z)\|}{d_{g_i}(x,z)}\,d z\\
\leq &C (d_{g_i}(x,y)+r)^m \xi_2.
\end{align*}
Combining this and the Bishop-Gromov inequality (Theorem \ref{BishopGromov} (ii)), we get
$$
r^m\leq C (d_{g_i}(x,y)+r)^m \frac{\xi_2}{\xi_3}.
$$
This implies the claim.
\end{proof}
Take arbitrary $\epsilon,\xi_1,\xi_2,\xi_3>0$ and suppose that
\begin{equation}\label{asuxi}
\xi_1\leq \frac{\pi}{4 \sqrt{K}},\quad \xi_1<\xi_2,\quad C_1 \left(\frac{\xi_2}{\xi_3}\right)^{1/m}<\frac{1}{12},\quad \xi_3<\frac{1}{2}.
\end{equation}
Then there exist $N_1\in\Z_{>0}$ and $\delta_1\in (0,\xi_1/2)$ such that for any $i\in \Z_{>0}$ with $i\geq N_1$ and $x,y\in M_i$, if $\|\iota_i(x)-\iota_i(y)\|\leq \delta_1$, then $d_{g_i}(x,y)\leq \xi_1/2$
by Claim \ref{cl1}.
By Claim \ref{cl3} (i) and (ii), there exists $N_0\in\Z_{>0}$, $x_1,\ldots,x_{N_0}\in M$ and $r_1,\ldots,r_{N_0}\in (0,\delta_1/2]$ such that
\begin{align*}
\iota(A)\subset& \bigcup_{j=1}^{N_0} B^{\R^d}(\iota(x_j),r_j),\\
\sum_{j=1}^{N_0}\omega_m r_j^m\leq &H_{\R^d,\delta_1/2}^m(\iota(A))+\epsilon.
\end{align*}
Take $\epsilon_1\in (0,\delta_1/2)$ so that
$$
\sum_{j=1}^{N_0}\omega_m (r_j+\epsilon_1)^m\leq \sum_{j=1}^{N_0}\omega_m r_j^m+\epsilon.
$$
Then, there exists $N_2\in \Z_{>0}$ with $N_2\geq N_1$ such that for any $i\in\Z_{>0}$ with $i\geq N_2$, we have
\begin{equation}\label{iotaiAi}
\iota_i(A_i)\subset  \bigcup_{j=1}^{N_0} B^{\R^d}(\iota_i(\psi_i(x_j)),r_j+\epsilon_1)
\end{equation}
by Claim \ref{cl3} (iii).

For each $j\in\{1,\ldots, N_0\}$, put
$$
\widetilde r_j:= (1+2\xi_3)\left(1+12C_1\left(\frac{\xi_2}{\xi_3}\right)^{1/m}\right)(r_j+\epsilon_1).
$$
Take arbitrary $i\in \Z_{>0}$ with $i\geq N_2$, $j\in\{1,\ldots, N_0\}$,
and let us show
\begin{equation}\label{incliota}
\iota_i^{-1}\left(B^{\R^d}(\iota_i(\psi_i(x_j)),r_j+\epsilon_1)\right)\cap B_{\xi_2}^i\subset B^{M_i}(\psi_i(x_j),\widetilde r_j)\subset M_i.
\end{equation}
To do this, take arbitrary
$$x\in\iota_i^{-1}\left(B^{\R^d}(\iota_i(\psi_i(x_j)),r_j+\epsilon_1)\right)\cap B_{\xi_2}^i.$$
Then, we have
$
\|\iota_i(x)-\iota_i(\psi_i(x_j))\|< r_j+\epsilon_1\leq \delta_1,
$
and so 
$
d_{g_i}(x,\psi_i(x_j))\leq \xi_1/2.
$
Putting $r_{i,j}(x):=2 C_1(\xi_2/\xi_3)^{1/m} d_i(x,\psi_i(x_j))$, we get
$d_{g_i}(x,\psi_i(x_j))+r_{i, j}(x)\leq 2d_i(x,\psi_i(x_j))\leq \xi_1$ and
$$
C_1\left(\frac{\xi_2}{\xi_3}\right)^{1/m}(d_{g_i}(x,\psi_i(x_j))+r_{i,j}(x))=\frac{r_{i,j}(x)}{2}+C_1\left(\frac{\xi_2}{\xi_3}\right)^{1/m}r_{i,j}(x)<r_{i,j}(x). 
$$
Thus, by Claim \ref{cl6}, there exists $z\in B^{M_i}(\psi_i(x_j),r_{i,j}(x))$ such that
$$
d(x,z)-\|\iota_i(x)-\iota_i(z)\|\leq \xi_3d_{g_i}(x,z).
$$
Then, we have
\begin{align*}
d_{g_i}(x,\psi_i(x_j))\leq &d_{g_i}(x,z)+d_{g_i}(z,\psi_i(x_j))\leq \frac{1}{1-\xi_3}\|\iota_i(x)-\iota_i(z)\|+r_{i,j}(x)\\
\leq& (1+2\xi_3)\|\iota_i(x)-\iota_i(\psi_i(x_j))\|+6C_1\left(\frac{\xi_2}{\xi_3}\right)^{1/m}d_{g_i}(x,\psi_i(x_j)),
\end{align*}
and so
$$
d_{g_i}(x,\psi_i(x_j))\leq (1+2\xi_3)\left(1+12C_1\left(\frac{\xi_2}{\xi_3}\right)^{1/m}\right)\|\iota_i(x)-\iota_i(\psi_i(x_j))\|<\widetilde r_j.
$$
This shows (\ref{incliota}).

For any $i\in\Z_{>0}$ with $i\geq N_2$, we get, we get
\begin{equation*}
\begin{split}
A_i\cap B_{\xi_2}^i
\subset &\bigcup_{j=1}^{N_0}\iota_i^{-1}\left(B^{\R^d}(\iota_i(\psi_i(x_j)),r_j+\epsilon_1)\right)\cap B_{\xi_2}^i\\
\subset&\bigcup_{j=1}^{N_0} B^{M_i}(\psi_i(x_j),\widetilde r_j)
\end{split}
\end{equation*}
by (\ref{iotaiAi}) and (\ref{incliota}).
We have
$$
\widetilde r_j \leq  (1+2\xi_3)\left(1+12C_1\left(\frac{\xi_2}{\xi_3}\right)^{1/m}\right)\delta_1\leq 2\xi_1.
$$
Thus, by (\ref{asuxi}), Theorem \ref{BishopGromov} (ii),  Claim \ref{cl3} and \ref{cl4}, we get
\begin{align*}
&\Vol_{g_i}(A_i)\\
\leq &\sum_{j=1}^{N_0}\Vol_{g_i}\left(B^{M_i}(\psi_i(x_j),\widetilde r_j)\right) +\Vol_{g_i}(M_i\setminus B_{\xi_2}^i)\\
\leq & (1+C\xi_1^2)(1+2\xi_3)^m \left(1+12C_1\left(\frac{\xi_2}{\xi_3}\right)^{1/m}\right)^m\left(\sum_{j=1}^{N_0}\omega_m r_j^m+\epsilon\right)\\
&\qquad \qquad \qquad \qquad \qquad \qquad \qquad \qquad \qquad \qquad +S\Vol(S^{m-1})\frac{\xi_1}{\xi_2}\\
\leq & \left(1+C\left(\xi_1^2+\xi_3+\left(\frac{\xi_2}{\xi_3}\right)^{1/m}\right)\right)\left(H_{\R^d}^m(\iota(A))+2\epsilon\right)+C\frac{\xi_1}{\xi_2}
\end{align*}
for each $i\in \Z_{>0}$ with $i\geq N_2$.
By the Volume convergence theorem \cite[Theorem 5.9]{CC1}, letting $i\to \infty$, we get
\begin{equation}\label{epeta}
H^m_{M}(A)\leq \left(1+C\left(\xi_1^2+\xi_3+\left(\frac{\xi_2}{\xi_3}\right)^{1/m}\right)\right)\left(H_{\R^d}^m(\iota(A))+2\epsilon\right)+C\frac{\xi_1}{\xi_2}.
\end{equation}

There exists $\xi_0=\xi_0(m,K,D,v_0)>0$ such that for any $\xi_1\in(0,\xi_0)$, putting $\xi_2=\xi_1^{1/2}$ and $\xi_3=\xi_1^{1/4}$, (\ref{asuxi}) holds.
Thus, for any $\epsilon>0$ and $\xi_1\in(0,\xi_0)$, (\ref{epeta}) implies
$$
H^m_{M}(A)\leq \left(1+C\xi_1^{1/4m}\right)\left(H_{\R^d}^m(\iota(A))+2\epsilon\right)+C\xi_1^{1/2}.
$$
Thus, we get $H^m_{M}(A)\leq H_{\R^d}^m(\iota(A))$, and so $H^m_{M}(A)= H_{\R^d}^m(\iota(A))$.

We have shown the theorem for any Borel subset $A$ of the form $A=\overline B^M(x,r)$.
Combining this, the Vitali covering theorem and the doubling property (Bishop-Gromov theorem), we get the theorem for any Borel subset of $M$.
\end{proof}

\section{Sensitivity of the Laplacian Approximation to Singularities}\label{sensitivity}
In this section, we construct an example under Assumption \ref{Asu} in which the approximation (\ref{LapApprox}) does not hold for any $p\in[1,\infty]$ when $\eta|_{[0,1]}\equiv 1$ and $\rho$ is constant. 
Define $M_1:=\partial ([0,1]^2)\subset \R^2$ and parametrize it by
$$
\phi\colon S^1\to M_1,\,(\cos \theta,\sin \theta)\mapsto
\begin{cases}
\left(\frac{2}{\pi}\theta,0\right)&\quad 0\leq\theta\leq \frac{\pi}{2},\\
\left(1,\frac{2}{\pi}\theta-1\right)&\quad \frac{\pi}{2}\leq\theta\leq\pi ,\\
\left(3-\frac{2}{\pi}\theta,1\right)&\quad \pi\leq\theta\leq \frac{3}{2}\pi,\\
\left(0,4-\frac{2}{\pi}\theta\right)&\quad \frac{3}{2}\pi\leq\theta\leq 2\pi.
\end{cases}
$$
Note that $\phi\colon (S^1,d_{S^1})\to \left(M_1,\frac{\pi}{2}d_{M_1}\right)$ is an isometry, where $d_{S^1}$ and $d_{M_1}$ denotes the intrinsic distances of $S^1$ and $M_1$, respectively.
Take integers $m,d\in\Z_{>0}$ with $m<d$ and arbitrary $(m-1)$-dimensional closed submanifold $M_2\subset \R^{d-2}$.
Define
$$
M:=M_1\times M_2\subset \R^2\times \R^{d-2}=\R^d.
$$
It is not difficult to show that $M\subset \R^d$ satisfies Assumption \ref{Asu} for some constants.
For each $\epsilon>0$ and $h\in L^2(M)$, we define
$$
L_\epsilon h(z_0):= \frac{1}{\epsilon^{m+2}}\int_{B^{\R^d}(z_0,\epsilon)\cap M}(h(z_0)-h(z))\,d z \quad (z_0\in M).
$$
Since we consider the kernel function $\eta|_{[0,1]}\equiv 1$, we define
$$
\sigma:=\frac{\Vol(S^{m-1})}{m(m+2)}=\frac{\Vol(B^{\R^m}(0,1))}{m+2}.
$$
For each $\alpha\in\R$, the function
$$
f_\alpha \colon S^1\to \R,\, (\cos \theta,\sin\theta)\mapsto\sin(\theta-\alpha)
$$
is an eigenfunction of the Laplacian on $S^1$.
Thus, the function
$$
F_\alpha\colon M_1\times M_2\to \R,\,(x,y)\mapsto f_\alpha(\phi^{-1}(x))
$$
is an eigenfunction of the Laplacian (without weight) on $M=M_1\times M_2$.
Note that $M$ with intrinsic distance function is isometric to a smooth Riemannian manifold and $F_\alpha\in C^\infty(M)$.
We get the following proposition, which shows that the approximation (\ref{LapApprox}) does not hold for $F_\alpha$.
\begin{Prop}
For any $p\in(1,\infty)$, we have
$$
\lim_{\epsilon\to 0}\int_M \left|L_\epsilon F_\alpha-\frac{1}{2}\sigma\Delta F_\alpha\right|^p\,d \Vol_M=\infty,
$$
and
\begin{align*}
&\lim_{\epsilon\to 0}\int_M \left|L_\epsilon F_\alpha-\frac{1}{2}\sigma\Delta F_\alpha\right|\,d \Vol_M\\
=&2\pi(|\sin \alpha|+|\cos \alpha|)\Vol(M_2) \Vol(B^{\R^{m-1}}(0,1))\int_0^1 |h_m(t)|\,d t\neq 0,
\end{align*}
where $h_m\colon[0,1]\to \R$ is a non-constant function defined by
$$
h_m(t)=\int_t^1 s(1-s^2)^{(m-1)/2}\,d s -\int_0^{(1-t^2)^{1/2}}(s+t)(1-s^2-t^2)^{(m-1)/2}\,d s.
$$
\end{Prop}

\section{Submanifold with Dense Singularities}\label{SwDS}
In this section, we construct an example with dense singularities satisfying Assumption \ref{Asu}.

\begin{Def}
Take a sequence $\theta\colon \Z_{\geq 2}\to \R_{>0}$ such that $\sum_{n=2}^\infty \theta(n)<\infty$.
\begin{itemize}
\item[(i)] For each $n\in \Z_{\geq 0}$, we define
$$
\D_n:=\left\{\frac{k}{2^{n}}:k=0,\ldots,2^n\right\}
$$
and
\begin{align*}
\D:=\bigcup_{n=0}^\infty \D_n=\left\{\frac{k}{2^{n}}:n\in \Z_{\geq 0}, \,k=0,\ldots,2^n\right\}.
\end{align*}
Note that we have $\D_0\subset \D_1\subset \D_2\subset \cdots$, and $\D\subset [0,1]$ is dense.
\item[(ii)] We define $\alpha\colon \D\to\R_{\geq 0}$ inductively as follows.
We first define
$$
\alpha(0)=\alpha(1)=0,\quad \alpha\left(\frac{1}{2}\right)=1.
$$
Then, $\alpha$ has been defined in $\D_1$.1
Now, taking $n\in \Z_{\geq 2}$ and assuming that $\alpha$ has been defined in $\D_{n-1}$, let us define $\alpha$ in
$$
\D_{n}\setminus \D_{n-1}=\left\{\frac{4k-3}{2^n}, \frac{4k-1}{2^n}: k=1,\ldots, 2^{n-2}\right\}
$$
as follows:
\begin{align*}
\alpha\left(\frac{4k-3}{2^n}\right):=&\frac{\theta(n)}{4}\alpha\left(\frac{4k}{2^n}\right)+\frac{1-\theta(n)}{2}\alpha\left(\frac{4k-2}{2^n}\right)
+\frac{2+\theta(n)}{4}\alpha\left(\frac{4k-4}{2^n}\right),\\
\alpha\left(\frac{4k-1}{2^n}\right):=&\frac{\theta(n)}{4}\alpha\left(\frac{4k-4}{2^n}\right)+\frac{1-\theta(n)}{2}\alpha\left(\frac{4k-2}{2^n}\right)
+\frac{2+\theta(n)}{4}\alpha\left(\frac{4k}{2^n}\right).
\end{align*}
See (iii) below for the reason for such a definition.
\item[(iii)] For each $n\in\Z_{\geq}$, we define $f_n\colon [0,1]\to \R_{\geq 0}$ so that
$$
f_n\left(\frac{k+t}{2^n}\right)=(1-t)\alpha\left(\frac{k}{2^n}\right)+t\alpha\left(\frac{k+1}{2^n}\right)
$$
for any $k=0,1,\ldots,2^n-1$ and $t\in[0,1]$.
Note that we have $f_n(0)=f_n(1)=0$, and we have defined $\alpha$ so that
\begin{align*}
\alpha\left(\frac{4k-3}{2^n}\right)=&(1-\theta(n))f_{n-1}\left(\frac{4k-3}{2^n}\right)+\theta(n)f_{n-2}\left(\frac{4k-3}{2^n}\right),\\
\alpha\left(\frac{4k-1}{2^n}\right)=&(1-\theta(n))f_{n-1}\left(\frac{4k-1}{2^n}\right)+\theta(n)f_{n-2}\left(\frac{4k-1}{2^n}\right).
\end{align*}
\item[(iv)] For each $n\in\Z_{\geq 0}$ and $k/2^n\in\D_n\setminus\{1\}$, we define
\begin{align*}
d_n\left(\frac{k}{2^n}\right):=&2^n\left(\alpha\left(\frac{k+1}{2^n}\right)-\alpha\left(\frac{k}{2^n}\right)\right),\\
e_n\left(\frac{k}{2^n}\right):=&d_n\left(\frac{k}{2^n}\right)-d_n\left(\frac{k-1}{2^n}\right),\\
E_n:=&\sum_{k=0}^{n-1}\left|e_n\left(\frac{k}{2^n}\right)\right|,
\end{align*}
where $d_n(-1/2^n):=d_n((2^n -1)/2^n)$.
Then, $d_n\left(k/2^n\right)$ is the slope of $f|_{\left[k/2^n,(k+1)/2^n\right]}$ and $e_n\left(k/2^n\right)$ is the change in slope of $f$ at $k/2^n$.
\end{itemize}
\end{Def}
Straightforward calculations imply the following lemma.
\begin{Lem}\label{lemSwDS}
We have the following.
\begin{itemize}
\item[(i)] For any $n\in\Z_{>0}$, we have
\begin{align*}
\sup_{x\in [0,1]}|f_{n+1}(x)-f_n(x)|=&\sup_{x\in \D_{n+1}\setminus \D_n}|f_{n+1}(x)-f_n(x)|\\
=&\theta(n+1)\sup_{x\in \D_{n+1}\setminus \D_n}|f_{n}(x)-f_{n-1}(x)|\\
=&\frac{\theta(n+1)}{2}\sup_{x\in [0,1]}|f_{n}(x)-f_{n-1}(x)|.
\end{align*}
\item[(ii)] For any $n\in\Z_{\geq 2}$ and $k=0,1,\ldots, 2^{n-2}-1$, we have
\begin{align*}
d_n\left(\frac{4k}{2^n}\right)=&\theta(n)d_{n-2}\left(\frac{4k}{2^n}\right)+(1-\theta(n))d_{n-1}\left(\frac{4k}{2^n}\right)\\
=&\frac{\theta(n)}{2}d_{n-1}\left(\frac{4k+2}{2^n}\right)+\left(1-\frac{\theta(n)}{2}\right)d_{n-1}\left(\frac{4k}{2^n}\right),\\
d_n\left(\frac{4k+1}{2^n}\right)=&-\frac{\theta(n)}{2}d_{n-1}\left(\frac{4k+2}{2^n}\right)+\left(1+\frac{\theta(n)}{2}\right)d_{n-1}\left(\frac{4k}{2^n}\right),\\
d_n\left(\frac{4k+2}{2^n}\right)=&\left(1+\frac{\theta(n)}{2}\right)d_{n-1}\left(\frac{4k+2}{2^n}\right)-\frac{\theta(n)}{2}d_{n-1}\left(\frac{4k}{2^n}\right),\\
d_n\left(\frac{4k+3}{2^n}\right)=&\theta(n)d_{n-2}\left(\frac{4k}{2^n}\right)+(1-\theta(n))d_{n-1}\left(\frac{4k+2}{2^n}\right)\\
=&\left(1-\frac{\theta(n)}{2}\right)d_{n-1}\left(\frac{4k+2}{2^n}\right)+\frac{\theta(n)}{2}d_{n-1}\left(\frac{4k}{2^n}\right).
\end{align*}
In particular, we have
$$
\sup_{x\in \D_n\setminus\{1\}}|d_n(x)|\leq (1+\theta(n)) \sup_{x\in \D_{n-1}\setminus\{1\}}|d_{n-1}(x)|\leq 2\exp\left(\sum_{l=2}^\infty \theta(l)\right).
$$
\item[(iii)] For any $n\in\Z_{\geq 2}$ and $k=0,1,\ldots, 2^{n-2}-1$, we have
\begin{align*}
e_n\left(\frac{4k}{2^n}\right)=&\theta(n)e_{n-2}\left(\frac{4k}{2^n}\right)+(1-\theta(n))e_{n-1}\left(\frac{4k}{2^n}\right)\\
=&\frac{\theta(n)}{2}e_{n-1}\left(\frac{4k+2}{2^n}\right)+e_{n-1}\left(\frac{4k}{2^n}\right)+\frac{\theta(n)}{2}e_{n-1}\left(\frac{4k-2}{2^n}\right),\\
e_n\left(\frac{4k+1}{2^n}\right)=&-\theta(n)e_{n-1}\left(\frac{4k+2}{2^n}\right),\\
e_n\left(\frac{4k+2}{2^n}\right)=&\left(1+\theta(n)\right)e_{n-1}\left(\frac{4k+2}{2^n}\right),\\
e_n\left(\frac{4k+3}{2^n}\right)=&-\theta(n)e_{n-1}\left(\frac{4k+2}{2^n}\right),
\end{align*}
where $e_{n-1}(-1/2^{n-1}):=e_{n-1}((2^{n-1}-1)/2^{n-1})$.
In particular, we have
$$
E_n\leq (1+4\theta(n)) E_{n-1}\leq 8\exp\left(4\sum_{l=2}^\infty \theta(l)\right).
$$
\end{itemize}
\end{Lem}
By Lemma \ref{lemSwDS} (i), there exists a function $f\colon [0,1]\to\R$ such that
$f_n$ converges to $f$ uniformly as $n\to \infty$.
By the construction, we have $f(x)=f_n(x)$ for any $n\in\Z_{>0}$ and $x\in \D_n$.
Extend $f$ periodically to $f\colon \R\to \R_\geq 0$, i.e.,
$f(k+t):=f(t)$ for any $k\in\Z\setminus\{0\}$ and $t\in[0,1)$.
By (ii), we have $\Lip(f)\leq  2\exp\left(\sum_{l=2}^\infty \theta(l)\right)$.
The first and fourth equations of (iii) imply that the signs of $e_n(x)$ and $e_{n-1}(x)$ are equal to each other for any $n\in \Z_{\geq 2}$ and $x\in \D_{n-1}$.
Combining this with the third and fifth equations of (iii), we get that $e_n(x)\neq 0$ for any $n\in\Z_{\geq 1}$ and $x\in \D_n$.
Thus, the first equation of (iii) implies that $e_n(x)\nrightarrow 0$ as $n\to \infty$ for any $x\in\D$.
This means that $f$ is not differentiable at any $x\in \D$.

By smoothing $f_n$, we get the following.
\begin{Lem}
For any $n\geq \Z_{>0}$, there exists a sequence of smooth periodic functions $\{h_{n,k}\colon \R\to \R_{\geq 0}\}_{k=1}^\infty$ of period $1$  such that
$\Lip(h_{n,k})\leq \Lip f_n$,
$\int_0^1 |h_{n,k}''(x)|\,d x=E_n$ for any $k\in \Z_{>0}$ and $h_{n,k}|_{[0,1]}$  converges to $f_n$ uniformly as $k\to \infty$.
\end{Lem}
Since $f_n$ converges to $f$ uniformly as $n\to \infty$, we immediately get the following corollary.
\begin{Cor}
There exists a sequence of smooth periodic functions $\{h_{k}\colon \R\to \R_{\geq 0}\}_{k=1}^\infty$ of period $1$  such that
$\Lip(h_{k})\leq 2\exp\left(\sum_{l=2}^\infty \theta(l)\right)$,
$\int_0^1 |h_{k}''(x)|\,d x\leq 8\exp\left(4\sum_{l=2}^\infty \theta(l)\right)$ for any $k\in \Z_{>0}$ and $h_{k}$  converges to $f$ uniformly as $k\to \infty$.
\end{Cor}

By using the approximation sequence $\{h_k\}$ for $f$, we can construct examples with dense singularities under Assumption \ref{Asu} as follows.
\begin{Prop}
Take arbitrary $m,d\in\Z_{>0}$ with $d>m\geq 2$ and arbitrary $(m-1)$-dimensional closed submanifold $\iota_2:M_2\to \R^{d-2}$ satisfying
$\|II_{\iota_2}\|_{L^\infty}\leq S_2$ and
$d_{M_2}(y_1,y_2)\leq L_2\|\iota_2(y_1)-\iota_2(y_2)\|_{\R^{d-2}}$ for some positive constants $S_2,L_2>0$.
Then, there exist constants $S,L>0$ such that
defining
$$
\iota\colon S^1\times M_2\to \R^d,\,((\cos 2\pi x,\sin 2\pi x),y)\to ((1+f(x))(\cos 2\pi x,\sin 2\pi x),\iota_2(y)),
$$
the pair $\left((S^1,c d_{S^1})\times M_2,\iota\right)$ satisfies Assumption \ref{Asu} for $m,d,S,K=S^2,i_0=\min\{\pi,\pi/S\},L$,
where we define
$$
c:=\frac{1}{2\pi}\int_0^1 \left((2\pi)^2(1+f(x))^2+f'(x)^2\right)^{1/2}\,d x
$$
so that $\iota$ is an isomertic embedding.
\end{Prop}
Note that $f$ is differentiable almost everywhere
since $f$ is a Lipschitz function.
By (the proof of) Theorem \ref{CR}, we have $|\Sect_{M_2}|\leq S^2$ and $\Inj_{M_2}\geq \pi/S$.
We can take $S$ and $L$ to be
\begin{align*}
S=&\left(2\pi +\frac{12}{\pi}\exp\left(4\sum_{l=2}^\infty \theta(l)\right)\right)\Vol(M_2)+2\pi c S_2 \Vol(M_2),\\
L=&\max\left\{\frac{1}{4}\left(16\pi^2+4\exp\left(2\sum_{l=2}^\infty \theta(l)\right)\right)^{1/2},L_2\right\}.
\end{align*}
Since $f$ has dense singularities, so does $\iota$.
\bibliographystyle{amsbook}

\end{document}